\documentclass[11pt]{article}
\usepackage[a4paper,margin=28mm]{geometry}
\usepackage[T1]{fontenc}
\usepackage{lmodern}
\usepackage{amsmath,amssymb,amsthm,mathtools}
\usepackage{microtype}
\usepackage{enumitem}
\usepackage{bm}
\usepackage{stmaryrd}
\SetSymbolFont{stmry}{bold}{U}{stmry}{m}{n}
\usepackage[hidelinks]{hyperref}
\hypersetup{pdftitle={Bulk universality of random regular graphs},
 pdfauthor={Yukun He, Jiaoyang Huang, Xiaoyu Wang}}
\newcommand{\proofstep}[1]{\par\medskip\noindent\textit{#1.}\ }
\numberwithin{equation}{section}
\newtheorem{theorem}{Theorem}[section]
\newtheorem*{mainthm}{Main theorem}
\newtheorem{proposition}[theorem]{Proposition}
\newtheorem{lemma}[theorem]{Lemma}
\newtheorem{corollary}[theorem]{Corollary}
\theoremstyle{remark}
\newtheorem{remark}[theorem]{Remark}

\def\R{\mathbb{R}}
\newcommand{\cG}{\mathcal G}
\newcommand{\bma}{{\bm{a}}}
\newcommand{\bmb}{{\bm{b}}}
\newcommand{\bmt}{\bm{t}}
\newcommand{\bmu}{{\bm{u}}}
\newcommand{\bmv}{{\bm{v}}}
\newcommand{\bmw}{{\bm{w}}}
\newcommand{\bmx}{{\bm{x}}}
\newcommand{\bmy}{{\bm{y}}}
\newcommand{\bmz}{\bm{z}}
\newcommand{\ri}{\mathrm{i}}
\newcommand{\bC}{{\mathbb C}}
\DeclareMathOperator{\Tr}{Tr}
\DeclareMathOperator{\supp}{supp}
\DeclareMathOperator{\diag}{diag}
\renewcommand{\Re}{\mathop{\mathrm{Re}}}
\renewcommand{\Im}{\mathop{\mathrm{Im}}}
\renewcommand{\leq}{\leqslant}
\renewcommand{\geq}{\geqslant}
\newcommand{\bla}{\bm{\lambda}}
\newcommand{\boxi}{\bm{\xi}}
\DeclareMathOperator{\Cov}{Cov}
\DeclareMathOperator{\Var}{Var}
\DeclareMathOperator{\Law}{Law}
\DeclareMathOperator{\TV}{TV}
\newcommand{\E}{\mathbb E}
\newcommand{\PP}{\mathbb P}
\newcommand{\F}{\mathcal F}
\newcommand{\dd}{\,\mathrm d}
\providecommand{\bmn}{\bm{n}}
\providecommand{\bme}{\bm{e}}
\providecommand{\bmg}{\bm{g}}

\usepackage{titlesec}
\titleformat{\subsection}[block]{\large\bfseries}{\thesubsection.}{.5em}{}
\titlespacing{\subsection}{0pt}{2ex plus .1ex minus .2ex}{.8em}
\titleformat{\subsubsection}[runin]{\normalfont\bfseries}{\thesubsubsection.}{.3em}{}[.]
\titlespacing{\subsubsection}{0pt}{2ex plus .1ex minus .2ex}{.5em}

\title{\bf \Large Bulk universality of random regular graphs}
\author{Yukun He\footnote{Shanghai Center for Mathematical Sciences, Fudan University. Email:	\href{mailto:heyukun@fudan.edu.cn}{heyukun@fudan.edu.cn}} \and Jiaoyang Huang\footnote{Courant Institute of Mathematical Sciences, New York University. Email:	\href{mailto:jh4427@nyu.edu}{jh4427@nyu.edu}} \and Xiaoyu Wang\footnote{Courant Institute of Mathematical Sciences, New York University. Email:	\href{mailto:xiaoyu.w@cims.nyu.edu }{xiaoyu.w@cims.nyu.edu }}\vspace{1em}} 
	
\date{}

\begin{document}
\maketitle
\begin{abstract}
We consider the adjacency matrix of a uniformly random simple $d$-regular
 graph on $N$ vertices. For every fixed $d\geq3$ and every fixed bulk
 energy, we prove that the rescaled eigenvalue point process converges
 to the $\mathrm{Sine}_1$ process with intensity $1/\pi$. We also establish
 universality of consecutive gaps at deterministic bulk labels and of
 fixed-energy correlation
 measures. The eigenvector input consists of polynomial estimates for
 mixed fourth moments on a uniformly sampled star. These estimates
 prevent loss of variance in the qualitative joint Gaussian limits of
 Backhausz--Szegedy, yielding independent variance-one Gaussian waves
 for fixed tuples at nearby energies. Exact switching identities then
 give a stationary marked flow and the microscopic loop hierarchy,
 which identifies the point-process limit. The appendices prove the
 stronger correlation and gap statements using counting-moment bounds
 and conditional log-gas laws for tagged limits.
\end{abstract}

\tableofcontents\bigskip

\section{Introduction}
We study a uniformly random simple, undirected $d$-regular graph
$\mathcal G=\mathcal G_{N,d}$ on the labelled vertex set
$\llbracket N\rrbracket=\{1,\ldots,N\}$, where $d\ge3$ is fixed.
Throughout, $N\to\infty$ along integers for which $dN$ is even
and the graph exists. Its adjacency matrix $A$ is real symmetric and
has zero diagonal.
We order its eigenvalues, including multiplicities, as
\[
 d=\lambda_1\ge\lambda_2\ge\cdots\ge\lambda_N.
\]
We choose an orthonormal eigenbasis with
$\bmu_1=N^{-1/2}\bm 1$, where $\bm 1$ is the all-ones vector.
This vector has eigenvalue $d$; every other basis vector is orthogonal
to it, including when the graph is disconnected.
All bulk intervals considered below are compact subsets of
$(-2\sqrt{d-1},2\sqrt{d-1})$, and hence exclude this eigenvalue. 

One of the most prominent questions for a random matrix model is its local eigenvalue statistics. For Wigner matrices, the universality of bulk eigenvalue gaps and correlation functions has been established under general assumptions~\cite{BEYY,EYGap,EYY}. The results were later extended to sparse matrices in \cite{EKYY1,EKYY2,HLY}, including adjacency matrices of Erd\H{o}s-R\'enyi graphs in regimes with polynomially growing average degree. Very recently, Henheik and Knowles \cite{HK26} proved convergence of the bulk point process at a fixed energy for $(\log N)^{2+\kappa}\leq Np\leq N^{1-\kappa}$, for every fixed $\kappa>0$.

The same question becomes much more difficult for random regular graphs, where the adjacency matrix is
highly structured and its entries are no longer independent. A local resampling technique was developed in \cite{BKY}, and it was proved that when $N^\varepsilon\leq d\leq N^{2/3-\varepsilon}$ for a fixed $\varepsilon>0$, the eigenvalue density of $A/\sqrt{d-1}$ matches with the semicircle law on all mesoscopic scales. Together with a comparison between switching dynamics and Dyson Brownian motion, \cite{BHKY} proved the bulk universality in the same degree regime. More recently, Bourgade and Huang~\cite{BH} proved convergence of the bulk point process at a fixed energy for $(\log N)^{24}\ll d\leq N^{1/2}$ using microscopic loop equations.

When $d$ is fixed, the empirical eigenvalue distribution converges to the
Kesten--McKay law~\cite{Kesten,McKay}. Local laws and eigenvector delocalization were proved for sufficiently large fixed degree by Bauerschmidt, Huang and Yau~\cite{BHY}, and for every fixed $d\geq3$ by Huang and Yau~\cite{HY}. Numerical simulations have long
suggested that the rescaled local eigenvalue statistics in the bulk agree
with those of the Gaussian Orthogonal Ensemble (GOE)~\cite{JMRR,OS}. 

In this paper, we prove bulk universality for every fixed $d\geq3$.
We first state convergence of the local point process, and then the
stronger gap and correlation conclusions. Write
\[
 \rho_d(x)=\frac{d\sqrt{4(d-1)-x^2}}{2\pi(d^2-x^2)}
          \mathbf1_{|x|\leq2\sqrt{d-1}},\qquad
 \rho_{\mathrm{sc}}(x)=\frac{\sqrt{4-x^2}}{2\pi}
          \mathbf1_{|x|\leq2}
\]
for the Kesten--McKay and semicircle densities, respectively.
Our GOE convention is the real symmetric matrix whose upper-triangular
entries are independent, with off-diagonal entries $N(0,N^{-1})$
and diagonal entries $N(0,2N^{-1})$. Its limiting density is
$\rho_{\mathrm{sc}}$; write $\mu_1\geq\cdots\geq\mu_N$ for its
ordered eigenvalues.

We use the normalization of Bourgade and Huang~\cite[Section~1.1]{BH}:
$\Theta_1$ denotes the translation-invariant $\mathrm{Sine}_1$ point
process with intensity $1/\pi$, so
\[
 \E\Theta_1(I)=\frac{|I|}{\pi}
 \qquad\text{for every bounded interval }I\subset\R.
\]
Thus the mean microscopic spacing is $\pi$. All eigenvalue and gap
scalings below use this convention. Here $\delta_x$ denotes a unit
point mass at $x$, and point measures count eigenvalues with their
multiplicities.

\begin{mainthm}[Fixed-energy point-process convergence]
\phantomsection\label{th:main-pointprocess}
Fix $d\geq3$ and a deterministic energy
$E\in(-2\sqrt{d-1},2\sqrt{d-1})$. Then, as $N\to\infty$,
\[
 \sum_{j=1}^N\delta_{N\pi\rho_d(E)(\lambda_j-E)}
 \ \Longrightarrow\ \Theta_1
\]
in distribution for the vague topology on locally finite counting
measures on $\R$. The randomness on the left is only that of the
uniformly random simple $d$-regular graph.
\end{mainthm}

The vague topology tests a measure against continuous compactly
supported functions. Equivalently, the conclusion can be written in
terms of Laplace functionals: for every nonnegative $f\in C_c(\R)$,
\[
 \lim_{N\to\infty}\E\exp\left\{-\sum_{j=1}^N
       f\bigl(N\pi\rho_d(E)(\lambda_j-E)\bigr)\right\}
 =\E\exp\left\{-\int_\R f(x)\,\dd\Theta_1(x)\right\}.
\]
The expectations on the two sides are under the graph law and the
$\mathrm{Sine}_1$ law, respectively. The distinguished eigenvalue $\lambda_1=d$
escapes every bounded microscopic interval, so omitting it gives the
same limit.

The main text proves this theorem in Section~\ref{fe:sec}.
Convergence in law does not by itself justify expectations of
unbounded local counting observables, and a fixed-energy window does
not keep track of a deterministic eigenvalue label. The following
stronger theorem treats both issues. Its correlation statement is
proved in Appendix~\ref{app:correlations}; its gap statement is proved
in Appendix~\ref{sec:gap-universality}, using the comparison established
in Appendix~\ref{app:gap-identification}.

For the gap normalization, define the upper-tail quantiles by
\[
 \int_{\gamma_i}^{2\sqrt{d-1}}\rho_d(x)\,\dd x=\frac{i}{N},\qquad
 \int_{\gamma_i^{\mathrm{sc}}}^{2}\rho_{\mathrm{sc}}(x)\,\dd x
 =\frac{i}{N},\qquad 1\leq i\leq N.
\]
These quantiles follow the decreasing eigenvalue order. We use this
convention throughout the main text and the graph gap argument;
Appendix~\ref{app:gap-identification} explicitly reverses the order for
its auxiliary GOE comparison.

\begin{theorem}[Gap and fixed-energy correlation universality]\label{th:gaps}
Fix $d\geq3$ and $k\geq1$.

\noindent\textup{(i)} For every $0<\delta<1/2$ and $F\in C_b(\R^k)$,
\[
\begin{aligned}
 \lim_{N\to\infty}\sup_{\delta N\leq i\leq(1-\delta)N}
 \Big|&\E F\bigl(N\pi\rho_d(\gamma_i)
       (\lambda_i-\lambda_{i+1},\ldots,
        \lambda_{i+k-1}-\lambda_{i+k})\bigr)\\
 &-\E F\bigl(N\pi\rho_{\mathrm{sc}}(\gamma_i^{\mathrm{sc}})
       (\mu_i-\mu_{i+1},\ldots,
        \mu_{i+k-1}-\mu_{i+k})\bigr)\Big|=0.
\end{aligned}
\]
The supremum is over integer indices; all displayed labels exist for
sufficiently large $N$.

\noindent\textup{(ii)} Fix $\tau>0$, a deterministic energy
$E\in[-2\sqrt{d-1}+\tau,2\sqrt{d-1}-\tau]$, and
$\varphi\in C_c(\R^k)$. Put $\widetilde E=E/\sqrt{d-1}$. Then
\begin{equation}\label{th:corrformula}
\begin{aligned}
 \lim_{N\to\infty}\Bigg(&\E\sum_{i_1,\ldots,i_k}^{\ne}
 \varphi\bigl(N\pi\rho_d(E)(\lambda_{i_1}-E),\ldots,
              N\pi\rho_d(E)(\lambda_{i_k}-E)\bigr)\\
 &-\E\sum_{i_1,\ldots,i_k}^{\ne}
 \varphi\bigl(N\pi\rho_{\mathrm{sc}}(\widetilde E)
                         (\mu_{i_1}-\widetilde E),\ldots,
              N\pi\rho_{\mathrm{sc}}(\widetilde E)
                         (\mu_{i_k}-\widetilde E)\bigr)\Bigg)=0.
\end{aligned}
\end{equation}
The sums run over ordered tuples of distinct indices in
$\{1,\ldots,N\}$. In both parts, the graph and GOE expectations are
separate; no coupling of the two ensembles is required.
\end{theorem}

Part (i) concerns a deterministic eigenvalue label and the consecutive
gaps following that label. Part (ii) instead observes the spectrum near
a deterministic energy, without averaging that energy. The first
window moves with the random spectrum, whereas the second does not;
in particular, the two statements require different additional
arguments beyond the \hyperref[th:main-pointprocess]{main point-process
theorem}. The sums in (ii) are factorial-moment observables. We assert
convergence of correlation measures tested against $C_c(\R^k)$, not
pointwise convergence of finite-$N$ correlation densities: the graph
model is discrete, and such densities are not assumed to exist.

The common input to all three eigenvalue conclusions is the analysis
of bulk eigenvectors. It allows us to pass to the microscopic limit in
exact switching identities. In the main text these identities yield
loop equations for the fixed-energy process; in the gap appendix they
also yield conditional log-gas laws for a process viewed from a fixed
label.

A Gaussian wave on the infinite
$d$-regular tree is a centered Gaussian process, invariant under tree
automorphisms, which satisfies the eigenvalue equation at every vertex~\cite{Elon}. We use the normalization in which every coordinate has
variance one, and call its restriction to a root and its $d$ neighbors
the Gaussian star law. Gaussian waves are the tree analogue of Berry's
random wave model~\cite{Berry}.

For an orthonormal selection $\bmu_1,\ldots,\bmu_k$ with eigenvalues
$\bla=(\lambda_1,\ldots,\lambda_k)$, let $\nu_{\bm{X}}$ be the law, conditional on the graph,
of $(\sqrt N \bmu_a(v_j))_{a,j}$ at a uniform root $v_0$ with uniformly
ordered neighbors $v_1,\ldots,v_d$. In this empirical law the graph
and the selected eigenvectors are fixed; the only randomness is the
choice of root and neighbor ordering. Let $\nu_{\bm{Z}_{\bla}}$ be the product
of the centered variance-one Gaussian star laws at these eigenvalues.
Here the subscripts enumerate the selected eigenvectors. For a Gaussian
wave with eigenvalue $\lambda$, the covariance between the root and a
neighbor is $\lambda/d$, and the covariance between distinct neighbors
is $(\lambda^2-d)/(d(d-1))$. The following theorem is the quantitative
eigenvector input used throughout the paper.

\begin{theorem}[Polynomial mixed fourth moments]\label{th:stars}\label{en:fourth}
Fix a compact bulk interval $K$. There are $b,\chi,C>0$ and graph
events of probability at least $1-CN^{-b}$ such that, on each fixed
graph in these events, simultaneously
for $1\leq k\leq4$, every orthonormal selection of $k$ real eigenvectors
with eigenvalues in $K$, and every coordinate choice
$(a_\ell,j_\ell)\in\{1,\ldots,k\}\times\{0,\ldots,d\}$,
$1\leq\ell\leq4$,
\begin{equation}\label{en:fourth-result}
 \left|\int\prod_{\ell=1}^4x_{a_\ell,j_\ell}\dd\nu_{\bm{X}}
 -\int\prod_{\ell=1}^4x_{a_\ell,j_\ell}\dd\nu_{\bm{Z}_{\bla}}\right|
 \leq CN^{-\chi}.
\end{equation}
Repetitions are allowed in the coordinate choices. The selected
eigenvalues may lie anywhere in $K$, and the constants depend only
on $d$ and $K$.
\end{theorem}

\begin{corollary}[Qualitative joint Gaussian stars]\label{en:qualitative}
Fix a compact bulk interval $K$ and $k\geq1$. For every deterministic
sequence $\delta_N\downarrow0$,
\begin{equation}\label{en:qualitative-result}
 \sup_{\substack{\text{orthonormal selections }\bmu_1,\ldots,\bmu_k\\
 \lambda_a\in K,\ \max_{a,b}|\lambda_a-\lambda_b|\leq\delta_N}}
 W_2(\nu_{\bm{X}},\nu_{\bm{Z}_{\bla}})
 \xrightarrow{\PP}0.
\end{equation}
The convergence in probability is with respect to the random graph;
the measures inside $W_2$ are its conditional empirical sampling laws
and their Gaussian comparison laws. The supremum over an empty set is
zero. The cost is Euclidean in the
$k(d+1)$ physical star coordinates. In particular, every fixed tuple
selected measurably from the graph whose eigenvalues converge in
probability to a common bulk energy has independent variance-one
Gaussian star limits, conditionally on the graph. No convergence rate
is asserted. For $k=1$, the restriction on the eigenvalue diameter is
vacuous.
\end{corollary}

The uniformity in the preceding statements is quenched: one first
chooses a graph in a common good event, and then may choose any of the
allowed eigenvectors on that graph. It is stronger than a separate
high-probability statement for each predetermined eigenvector. Within
a multiple eigenspace, any real orthonormal choice is allowed. When
random choices of eigenvectors are needed, we fix a measurable
orthonormal eigenbasis, for example by applying Gram--Schmidt in a
deterministic order to spectral projections of the coordinate vectors.
The results do not depend on this convention, and we do not insert
independent random signs in the empirical star laws.

We use the convention
$W_2(\mu,\nu)^2=\inf\E|\bm{X}-\bm{Y}|^2$, where the infimum is over couplings.
The two statements have distinct roles in the proof. The polynomial
mixed fourth moments control covariances of quadratic eigenvector
observables, including sums over eigenvalues throughout $K$.
Qualitative convergence identifies the joint law of each fixed collection
of eigenvector marks whose eigenvalues approach the same energy.

Previously, for fixed $d\geq 3$, \cite{BS} proved that almost
eigenvectors of random regular graphs converge locally to Gaussian waves
with variance in $[0,1]$. The paper also obtained joint Gaussian limits for
fixed tuples at nearby energies. For edge eigenvectors, \cite{HHY} proved convergence to independent Gaussian waves of variance
one, jointly with and independently of the edge eigenvalues. The result
holds on every fixed-radius ball, with independent random signs attached
to the eigenvectors. For
growing degrees
$N^\varepsilon\leq d\leq N^{2/3-\varepsilon}$ with a fixed $\varepsilon>0$, asymptotic normality of bulk
eigenvectors was established by Bourgade, Huang and Yau~\cite{BHYev}.

This paper considers bulk eigenvalue statistics. Near the edge of the spectrum, Bauerschmidt, Huang, Knowles and
Yau~\cite{BHKYedge} proved universality for
$N^{2/9+o(1)}\leq d\leq N^{1/3-o(1)}$. Huang and Yau~\cite{HYedge} extended
this range to $N^{o(1)}\leq d\leq N^{1/3-o(1)}$, while He~\cite{He} established
edge universality for $N^{2/3+o(1)}\leq d\leq N/2$. In \cite{HMY}, Huang, McKenzie and Yau proved optimal eigenvalue rigidity up to subpolynomial
factors and edge universality for every fixed $d\geq3$.

\subsection{Proof ideas}

The point-process proof and the appendix upgrades use the polynomial
mixed fourth moments in Theorem~\ref{th:stars} and the qualitative Gaussian limit in
Corollary~\ref{en:qualitative}. To obtain the moment estimates, we develop
a quantitative fourth-moment refinement of the Gaussian maximality
argument of Backhausz--Szegedy \cite[Theorem~11.2]{BS}; see
Proposition~\ref{en:quartic}. The heat-flow comparison follows
\cite[Proposition~11.1]{BS}, while fourth-chaos coercivity and a
high-noise entropy expansion yield the explicit moment bound.
For an eigenvector with eigenvalue $\lambda$, we add correlated Gaussian
noise with covariance proportional to
$\Im(A-\lambda-\ri\eta)^{-1}$, where $\eta=N^{-c}$ for a small fixed
$c>0$. The local laws of Huang and Yau~\cite{HY} control both the local
covariance and the total entropy of this perturbation. The resulting
local and global entropy contributions cancel up to a polynomial error.
We need this deficit estimate only at one fixed noise level, as stated
in Proposition~\ref{en:deficit}. Ornstein--Uhlenbeck monotonicity transfers
the bound to the high-noise regime, where the abstract fourth-moment
estimate applies.

For the full law, we use the qualitative joint neighborhood conclusion
of Backhausz--Szegedy for fixed tuples at nearby energies. The fourth-moment
bound makes all quadratic coordinate products uniformly integrable.
The exact empirical covariance therefore passes to every Gaussian limit,
fixing each variance at one and making the channels independent. This
also gives qualitative convergence in $W_2$; a quantitative full-law
stability estimate is unnecessary for the switching argument.

At fixed degree, a simple switching changes adjacency matrix entries
by quantities of order one. We therefore use the exact resolvent
identity for the corresponding rank-two perturbation. In this identity,
products of eigenvector coordinates are summed over the spectrum. The qualitative joint
Gaussian limit identifies each fixed collection with eigenvalues
approaching the same energy; the polynomial mixed fourth-moment
estimate controls the contribution of the remaining eigenvectors
throughout the fixed bulk interval. Together with the switching identities, they give
tightness at the microscopic scale.

Fix a bulk energy $E$, set
$\sigma=\pi\rho_d(E)$ and $a_E=\Re m_d(E+\ri0)$, and, for
$z,w\in\bC^+$, define
\[
 \begin{gathered}
 m_N(z)=\frac1N\Tr(A-z)^{-1},\qquad
 m_d(z)=\int_\R\frac{\rho_d(t)}{t-z}\,\dd t,\\
 s_N(w)=\frac{m_N(E+w/(N\sigma))-a_E}{\sigma}.
 \end{gathered}
\]
Write $s$ for a subsequential limit of $s_N$. The switching identity also
involves the normalized $2\times2$ compression of the resolvent to the
two switching directions.
Its limit, denoted by $S$, has a spectral representation with Gaussian
vectors attached to the eigenvalues. We call these vectors the
eigenvector marks. By the quenched uniformity in Corollary~\ref{en:qualitative},
the limiting marks are independent standard Gaussian vectors conditionally
on the scalar spectral data.

Simple switching is reversible: the original and switched graphs form
an exchangeable pair. Thus the limiting law of $(s,S)$ is invariant
under the switching involution $\mathcal T$. It is also invariant under
rotations $\mathcal R_\theta$ of the Gaussian marks. Conjugating these
rotations by switching gives the flow $\Phi_\theta$ below, which preserves
the same law. Although the switching is discrete, we can differentiate
this flow.
If $\mathcal V$ is the infinitesimal generator and $F$ is a polynomial observable of the scalar and marked
resolvents, then
\[
 \Phi_\theta=\mathcal T\mathcal R_\theta\mathcal T,
 \qquad \E\mathcal V F=0.
\]
Averaging the resulting identities over the Gaussian marks gives the
microscopic loop equations, starting with
\[
 \E\bigl[s'(w)+s(w)^2+1\bigr]=0.
\]
The full hierarchy is proved in Proposition~\ref{lp:criterion}. The
characterization of Bourgade and Huang~\cite[Theorem~1.5]{BH}
identifies the limiting point process as $\Theta_1$, in the same
coordinates as the \hyperref[th:main-pointprocess]{main point-process
theorem}. This proves that theorem and is the endpoint of
Section~\ref{sec:fixed-energy}.
Appendix~\ref{app:correlations} then establishes uniform integrability
of every fixed local counting moment. It upgrades point-process
convergence to the correlation statement in Theorem~\ref{th:gaps}(ii).

For Theorem~\ref{th:gaps}(i), Appendix~\ref{sec:gap-universality}
keeps the eigenvalue at a deterministic bulk label as the origin when
taking the microscopic limit. With this
choice of origin, the switching argument gives loop equations and
integration-by-parts identities for the labelled eigenvalues. The loop
equations yield rigidity and control the force exerted by distant
eigenvalues. To use the integration-by-parts identities for a conditional
density, we first integrate the rooted law over translations, freeing
the tagged eigenvalue from the origin. The identities then show that,
given the exterior eigenvalues, the density in a finite block has the
form of a $\beta=1$ log gas; see Lemma~\ref{gp:gibbs}.

Proposition~\ref{gp:rigid-gibbs} isolates the final implication:
rigidity and these conditional Gibbs laws identify the gaps with those
of the Palm process $\Theta_1^0$ associated with $\Theta_1$, with the
same mean spacing $\pi$.
Its proof, given in Appendix~\ref{app:gap-identification}, completes the
distant exterior by Gaussian classical locations and applies the local
log-gas comparison of Erd\H{o}s and Yau~\cite[Proposition~5.3]{EYGap},
which uses Dyson Brownian motion~\cite{Dyson}. GOE kernel
asymptotics~\cite{DG} and Campbell's formula identify the reference gap
law. The microscopic graph limit is taken first, and the conditional
block size is then sent to infinity.

\subsubsection*{Logical dependencies of the proof}
The main-text proof has three stages. First, the entropy tools, the
weighted Gaussian-star frame, and fourth-chaos coercivity convert a
small entropy deficit into polynomial control of mixed fourth moments.
Second, the random-graph local law supplies this deficit uniformly over
fixed finite collections of eigenvectors, and the qualitative theorem
of Backhausz--Szegedy identifies variance-one joint Gaussian waves at
nearby energies. Third, exact switchings and Gaussian marks yield the
microscopic loop hierarchy and identify the fixed-energy point process.
None of these stages uses the appendix upgrades.

The two parts of Theorem~\ref{th:gaps} are proved separately after this
identification. The correlation argument adds counting-moment bounds
to the point-process limit. The gap argument instead adapts the marked
construction to a deterministic eigenvalue label; rooted loop equations,
rigidity, and coordinate integration by parts give conditional
$\beta=1$ log-gas laws, whose gap statistics are identified by Gaussian
completion. The graph limit is taken before the auxiliary spectral
and conditional-block cutoffs are removed.

\subsubsection*{Organization and notation}
The main text consists of Sections~1--3.
Section~\ref{sec:gaussian-wave} proves the polynomial mixed fourth
moments and the qualitative joint Gaussian limit at nearby energies.
Section~\ref{sec:fixed-energy} constructs the microscopic marked limit
and its stationary switching flow, derives the loop hierarchy, and ends
in Section~\ref{fe:sec} with the proof of the
\hyperref[th:main-pointprocess]{main point-process theorem}.
Appendix~\ref{app:correlations} proves Theorem~\ref{th:gaps}(ii) by
controlling local counting moments.
Appendix~\ref{sec:gap-universality} constructs the tagged limit, proves
rigidity and conditional Gibbs laws, and deduces Theorem~\ref{th:gaps}(i)
from Proposition~\ref{gp:rigid-gibbs}.
Appendix~\ref{app:gap-identification} proves that abstract identification
proposition by Gaussian completion and a GOE comparison.

Unless stated otherwise, the large parameter is $N$, and dependence
on $N$ is suppressed. Constants $c,C>0$ may change from line to line;
their allowed dependence is specified in the relevant statement.
The notation $O_\alpha(\cdot)$ records additional dependence on $\alpha$.
Bold symbols denote vectors, and $\bmu(v)$ is the scalar coordinate
of $\bmu$ at vertex $v$. For a star, $X_{a,j}$ denotes channel $a$
at position $j$: $j=0$ is the root and $1\le j\le d$ are its neighbors.
An index $a$ on a selected eigenpair labels that selection, rather
than the global eigenvalue order.

We write $|\bm v|$ for Euclidean length, $\|M\|$ for the operator
norm of a matrix, $M^T$ for transpose, and $M^*$ for conjugate
transpose. Thus $\Im M=(M-M^*)/(2\ri)$ for a matrix-valued resolvent.
For a scalar random variable and $1\le p<\infty$,
$\|X\|_p=(\E|X|^p)^{1/p}$; the probability law is specified locally.
For Gaussian function spaces it is always stated explicitly.
We use $[M_1,M_2]=M_1M_2-M_2M_1$ and the total variation convention
$\|\mu-\nu\|_{\TV}=\sup_A|\mu(A)-\nu(A)|$.
For graph-dependent matrices and transforms, a tilde denotes the
switched graph; in particular $G$ and $\widetilde G$ are the original
and switched Green functions.

\subsubsection*{Probability spaces and order of limits}
The probability law in use must be distinguished from any additional
sampling on a fixed graph. Unsubscripted $\PP$ and $\E$ refer to the
law currently under discussion: the graph law in a graph-only
statement, the joint law when a proposal or noise is included, and
the law of the limiting state after passage to a limit. A superscript,
such as $\E^{\mu}$, specifies the probability law explicitly.

After the graph and its eigenvectors have been fixed, $\E_\star$
averages only a uniform root and a uniform ordering of its neighbors.
The notation $\E_*$ instead averages the two independent uniform
oriented edges in a switching proposal, before acceptance is imposed.
The corresponding conditional variances are $\Var_\star$ and $\Var_*$.
For the Gaussian channel, $\E_{\bm Y}$ averages only the auxiliary
Gaussian vector, with its signal and covariance fixed. An empirical
profile already contains the finite averages over roots and neighbor
orderings; those samples are not averaged a second time in
$\E_{\bm Y}$. In a limiting marked state, conditioning on the scalar
state leaves only the Gaussian marks random. Independence is used
only where it is specified in the construction.

A compact bulk interval $K$ and the degree $d$ are fixed before $N$
tends to infinity. A finite number of selected eigenvectors is also
fixed before that limit. In marked resolvent sums, we first retain a
fixed microscopic window and a fixed number of marks, then take
$N\to\infty$, and only afterwards remove the cutoffs. In the gap
argument, the limiting tagged process is constructed first; the
conditional block size is then sent to infinity. Keeping these
orders separate is what allows qualitative Gaussian convergence to
be sufficient.

\subsubsection*{Resolvent conventions and elementary estimates}
For a real symmetric matrix $A$ and $z=E+\ri\eta$ with $\eta>0$,
the spectral theorem gives
\[
 G(z)=\sum_j\frac{\bmu_j\bmu_j^T}{\lambda_j-z},\qquad
 \Im G(z)=\eta G(z)G(z)^*,\qquad \|G(z)\|\le\eta^{-1}.
\]
Here $G(z)^*$ denotes the conjugate transpose. In particular,
$\Im G(z)$ is positive definite and
\[
 \frac1N\sum_j\frac1{|\lambda_j-z|^2}
 =\frac{\Im m_N(z)}\eta.
\]
This last equality is the Ward identity in the form used below.
Our Stieltjes convention is $(x-z)^{-1}$, so the imaginary part is
positive in the upper half-plane. An atom of mass $r$ at $x$ contributes
$r/(x-z)$: its \emph{pole weight} is $r$, whereas its complex-analytic
residue at $z=x$ is $-r$.

We use repeatedly the following elementary implication. If
$f(w)=a+\int((x-w)^{-1}-x/(1+x^2))\,\dd\mu(x)$, with $a\in\R$
and $\mu$ a positive measure, then on any compact
$B\subset\bC\setminus\R$ contained in one half-plane,
\[
 \sup_{w\in B}|f(w)-f(\ri)|\le C_B\Im f(\ri)
 \quad\text{in the upper half-plane}.
\]
Indeed,
$(x-w)^{-1}-(x-\ri)^{-1}=(w-\ri)/((x-w)(x-\ri))$, and the
absolute value of this kernel is at most $C_B/(1+x^2)$.
The lower half-plane follows by complex conjugation. Cauchy's
integral formula gives the corresponding bounds for derivatives on
smaller compact sets. Thus a moment estimate at one height provides
local analytic tightness; it does not, by itself, identify the poles.

Finally, our microscopic coordinate is $x=N\sigma(\lambda-E)$,
where $\sigma=\pi\rho_d(E)$. This is the coordinate used in both the
\hyperref[th:main-pointprocess]{main point-process theorem} and
Theorem~\ref{th:gaps}, as well as in the resolvent arguments and the
appendices. The limiting point process is $\Theta_1$ with intensity
$1/\pi$; no further dilation is made when identifying its law.

\subsection*{Acknowledgment}
YH is supported by NSFC No.12322121 and the National Key R\&D Program of China No.2023YFA1010400. The research of JH is partially supported  by the Sloan Research Award.

\subsection*{LLM disclosure}
GPT-6 assisted the proof development at the level of a coauthor. The authors are responsible for the entire content of the paper.

\section{Quantitative Gaussian wave limit}\label{sec:gaussian-wave}

We collect the general entropy and Gaussian tools and apply them to prove
polynomial mixed fourth-moment estimates and qualitative joint Gaussian
limits for empirical eigenvector laws. An empirical star
is random only through the uniform root and the uniform ordering of
its neighbors. Auxiliary Gaussian noise is introduced explicitly when
needed; probabilities of good graph events refer to the random graph.

\subsection{General entropy and Gaussian tools}\label{en:tools}

We collect statements about fixed probability laws, covariance matrices,
and finite-dimensional Gaussian spaces. Each expectation in this
subsection is taken under the law specified in the corresponding
statement. These tools will later be applied after the graph and the
selected eigenvectors have been fixed.

\subsubsection{Entropy and Gaussian identities}\label{en:entropy-tools}

For a probability vector $\bm p=(p_j)_j$, write
$H(\bm p)=-\sum_jp_j\log p_j$, with $0\log0=0$ and natural logarithms.
For a discrete random variable $X$, the abbreviation $H(X)$ means
$H(\Law(X))$. A finite-valued random probability vector needs more care:
$H(\widehat{\bm p})$ is the entropy of its realized entries and is
itself a random variable, whereas $H(\Law(\widehat{\bm p}))$ is
the entropy of the random profile as an outcome. We never use the
first notation for the second quantity. In particular,
$\E H(\widehat{\bm p})$ and $H(\E\widehat{\bm p})$ are generally
different.

For a density $p$ on a Euclidean space, write $h(p)=-\int p\log p$.
The standard Gaussian law on $\R^m$ is denoted by $\gamma_m=N(0,I_m)$.
On a subspace, the reference measure is its induced Euclidean volume.
We also write $h(X)=h(\Law(X))$ when that law has a density on the
specified space. Conditional entropy is an average of entropies of
conditional laws; for example,
\[
 H(X\mid Z)=\E\bigl[H(\Law(X\mid Z))\bigr],\qquad
 h(X\mid Z)=\E\bigl[h(\Law(X\mid Z))\bigr].
\]
In the discrete expression the inner $H$ is the entropy of the
conditional probability vector, evaluated at the realized $Z$.
Finally, put $H_2(t)=-t\log t-(1-t)\log(1-t)$.
All uses below have finite displayed entropies.

If a finite-valued label $J$ selects a component law $P_J$ with
probability $q_J$, write $\bm{q}=(q_j)_j$. The entropy chain rule gives
\begin{equation}\label{en:mixture-entropy}
 \sum_jq_jh(P_j)\le h\left(\sum_jq_jP_j\right)
 \le\sum_jq_jh(P_j)+H(\bm{q}).
\end{equation}
Indeed the difference is the mutual information between the sample
and its component label, which lies between zero and $H(\bm{q})$.
The same bounds hold for discrete entropy.

More explicitly, if $X$ is the mixture sample, then
$h(X\mid J)=\sum_jq_jh(P_j)$ and the chain rule gives
\[
 h(X)-\sum_jq_jh(P_j)=H(J)-H(J\mid X).
\]
The right side lies between zero and $H(J)$. This calculation also explains why the upper
bound is independent of overlap between the components: complete
overlap gives zero mutual information, while perfectly distinguishable
components give $H(J)$.
 For probability laws
$\bm{p},\bm{q}$ on an alphabet of size $M$, with total variation distance $v$,
\begin{equation}\label{en:entropy-continuity}
 |H(\bm{p})-H(\bm{q})|\le H_2(v)+v\log M.
\end{equation}
This follows by coupling the two samples to disagree with probability
$v$ and applying the entropy chain rule to the disagreement indicator.

For completeness, let $(X,Y)$ be such a coupling and
$B=\mathbf1_{X\ne Y}$. Given $Y$ and $B=0$, the value of $X$ is
known; given $Y$ and $B=1$, it has at most $M$ possibilities. Therefore
$H(X\mid Y)\le H_2(v)+v\log M$. The inequality
$H(X)\le H(Y)+H(X\mid Y)$ gives one direction, and exchanging
$X,Y$ gives the other. This proof uses total variation in the convention
$\frac12\sum_j|p_j-q_j|$.

For $\bm{Z}\sim N(\bm{m},Q)$ in $\R^n$, with $Q$ positive definite,
\begin{equation}\label{en:gaussian-entropy}
 h(\bm{Z})=\tfrac12\log\det(2\pi eQ).
\end{equation}
Every density with covariance $Q$ has entropy at most the right-hand
side, by nonnegativity of its relative entropy with respect to the
Gaussian law with the same mean and covariance.

Indeed, writing $g$ for that Gaussian density, its logarithm is a
constant minus a quadratic form. Matching means and covariances implies
$\int p\log g=\int g\log g$, and therefore
$\int p\log(p/g)=h(g)-h(p)\ge0$. This also proves the maximum-entropy
bound when the competing law is a non-Gaussian mixture.

If a Gaussian vector $(\bm{Z}_1,\bm{Z}_2)$ has covariance blocks $Q_{11}$,
$Q_{12}$ and $Q_{22}$, with $Q_{11}$ invertible, then
\begin{equation}\label{en:gaussian-conditional}
 \Cov(\bm{Z}_2\mid \bm{Z}_1)
 =Q_{22}-Q_{12}^TQ_{11}^{-1}Q_{12}.
\end{equation}
This is obtained by subtracting the linear regression on $\bm{Z}_1$;
the residual is Gaussian and independent of $\bm{Z}_1$.

The residual is
$\bm{Z}_2-\E\bm{Z}_2-Q_{12}^TQ_{11}^{-1}
(\bm{Z}_1-\E\bm{Z}_1)$. Its covariance with $\bm{Z}_1$ is zero;
for jointly Gaussian vectors, zero covariance implies independence.
Computing its covariance gives the Schur complement in
\eqref{en:gaussian-conditional}.

We also use the Gaussian Poincar\'e inequality in the form
\begin{equation}\label{en:gaussian-poincare}
 \Var(f(\bm{Z}))\le \|Q\|\E|\nabla f(\bm{Z})|^2,
\end{equation}
for locally Lipschitz $f$ with square-integrable gradient. It follows
from the standard Gaussian inequality by the change of variables
$\bm Z=\bm m+Q^{1/2}\bm G$; smooth approximation gives the
locally Lipschitz version used for the softened quantizer below.

\subsubsection{Gaussian mixtures and discretization}

\begin{lemma}[Entropy stability of Gaussian mixtures]\label{sp:mixture}
Let $w_1,\ldots,w_r\geq0$ satisfy $\sum_{j=1}^r w_j=1$, and let
$\bm m_1,\ldots,\bm m_r\in\R^m$ be fixed.  For $0\leq t\leq1$, set
\[
 p_t=\sum_{j=1}^r p_{t,j},
 \qquad
 p_{t,j}=w_j\,\mathcal N(\bm m_j,K_j(t)),
\]
where $\mathcal N(\bm m,K)$ denotes the Gaussian density with mean
$\bm m$ and covariance $K$.  Assume that every $K_j(t)$ is
differentiable and
\[
 K_j(t)\geq\kappa I,
 \qquad
 \|\dot K_j(t)\|\leq\varepsilon.
\]
Then
\begin{equation}\label{sp:mixture-bound}
 |h(p_1)-h(p_0)|\leq \frac{m\varepsilon}{2\kappa}.
\end{equation}
The estimate is independent of the component means, weights, and number
of components.
\end{lemma}

\begin{proof}
Components of weight zero may be discarded.  All derivatives below are
with respect to the sample variable $\bm x$ unless a $t$-derivative is
shown explicitly. In each integrand $K_j$ and $\dot K_j$ are
evaluated at the same time $t$. The integrals in this proof are over
$\R^m$ with respect to Euclidean volume.

\proofstep{Differentiate a Gaussian with respect to its covariance}
For one component, direct differentiation gives
\begin{align*}
 \partial_t p_{t,j}(\bm x)
 &=\frac12p_{t,j}(\bm x)
 \Bigl[(\bm x-\bm m_j)^T K_j^{-1}\dot K_jK_j^{-1}
             (\bm x-\bm m_j)-\Tr(K_j^{-1}\dot K_j)\Bigr]\\
 &=\frac12\Tr\bigl(\dot K_j\nabla^2p_{t,j}(\bm x)\bigr).
\end{align*}
The second identity may also be checked from the Fourier transform of a
Gaussian.  Since $\int\partial_t p_t=0$, differentiating the entropy and
integrating by parts once gives
\begin{equation}\label{sp:mixture-derivative}
 \frac{\dd}{\dd t}h(p_t)
 =\frac12\sum_{j=1}^r\int
   (\nabla p_{t,j})^T\dot K_j\nabla\log p_t\,\dd\bm x.
\end{equation}
The boundary terms vanish because all component densities and their
first derivatives have Gaussian decay.  Equivalently, one may first
insert a smooth cutoff and then let its radius tend to infinity.

\proofstep{Compare the component scores with the mixture score}
Define
\[
 \bm s_{t,j}=\nabla\log p_{t,j}
   =-K_j^{-1}(\bm x-\bm m_j),
 \qquad
 \bm s_t=\nabla\log p_t.
\]
At every $\bm x$,
\[
 \bm s_t(\bm x)
 =\sum_j\frac{p_{t,j}(\bm x)}{p_t(\bm x)}\bm s_{t,j}(\bm x).
\]
The coefficients $p_{t,j}(\bm x)/p_t(\bm x)$ are the posterior
component probabilities given the sample value $\bm x$; they are
nonnegative and sum to one. Pointwise Jensen followed by integration gives
\begin{equation}\label{sp:mixture-score}
 \int p_t|\bm s_t|^2
 \leq\sum_j\int p_{t,j}|\bm s_{t,j}|^2
 =\sum_jw_j\Tr K_j^{-1}\leq\frac m\kappa.
\end{equation}
The middle equality is the elementary Gaussian identity
$\E|K^{-1}(\bm Z-\bm m)|^2=\Tr K^{-1}$ for $\bm Z\sim N(\bm m,K)$.

\proofstep{Bound and integrate the entropy derivative}
Using $\nabla p_{t,j}=p_{t,j}\bm s_{t,j}$, the operator-norm bound on
$\dot K_j$, and Cauchy--Schwarz jointly in $j$ and $\bm x$, we obtain
\begin{align*}
 \left|\frac{\dd}{\dd t}h(p_t)\right|
 &\leq\frac\varepsilon2\sum_j\int
       p_{t,j}|\bm s_{t,j}|\,|\bm s_t|\\
 &\leq\frac\varepsilon2
   \left(\sum_j\int p_{t,j}|\bm s_{t,j}|^2\right)^{1/2}
   \left(\int p_t|\bm s_t|^2\right)^{1/2}\\
 &\leq\frac\varepsilon2\sum_jw_j\Tr K_j^{-1}
 \leq\frac{m\varepsilon}{2\kappa}.
\end{align*}
Integrating this uniform estimate over $0\leq t\leq1$ proves
\eqref{sp:mixture-bound}.  Notice that the means never enter the bound;
this is important later because the mixture components are centered at
graph-dependent eigenvector coordinates.
\end{proof}

\begin{lemma}[Discretization of Gaussian mixtures]\label{ch:quant}
Fix $m\geq1$ and $\vartheta,\kappa>0$, and put
\[
 R_0=N^\vartheta,\qquad \eta_0=N^{-\vartheta},\qquad
 a_0=N^{-4\vartheta},\qquad K_0=\lfloor R_0/a_0\rfloor.
\]
For $z\in\R$, let $J(z)=\lfloor z/a_0\rfloor$. For an integer $n$, set
$T(n)=\min\{K_0-1,\max\{-K_0,n\}\}$. Both maps act coordinatewise,
so $TJ(\bm Z)$ is the clipped random grid label, not the quantizer map itself.
Suppose that $\bm Z$ is a finite Gaussian mixture in $\R^m$, every
component mean has coordinate size at most $R_0/2$, and every component
covariance $K$ satisfies
\[
 \kappa\eta_0 I\leq K\leq\eta_0^{-1}I.
\]
Then, for all sufficiently large $N$,
\begin{equation}\label{ch:quant-bound}
 \left|H(TJ(\bm Z))+m\log a_0-h(\bm Z)\right|
 \leq C_{m,\kappa}N^{-7\vartheta}
       +C_mN^{5\vartheta}e^{-c_mN^\vartheta}.
\end{equation}
The constants are uniform in the number, weights, and means of the
mixture components.
\end{lemma}

\begin{proof}
Write $f$ for the density of $\bm Z$.  The proof has two independent
parts.  First we compare differential entropy with the entropy of the
unbounded grid label $J(\bm Z)$.  We then show that clipping this label
changes entropy only by an exponentially small amount.

\proofstep{The exact cellwise entropy identity}
Let
\[
 C_{\bm j}=a_0(\bm j+[0,1)^m),\qquad
 p_{\bm j}=\int_{C_{\bm j}}f,\qquad
 \bar f_{\bm j}=p_{\bm j}/a_0^m.
\]
A direct expansion of the definitions gives
\begin{equation}\label{ch:cell-entropy}
 H(J(\bm Z))+m\log a_0-h(\bm Z)
 =\sum_{\bm j}\int_{C_{\bm j}}
       f(\bm x)\log\frac{f(\bm x)}{\bar f_{\bm j}}\,\dd\bm x.
\end{equation}
Every summand is nonnegative by Jensen's inequality.  This also
explains the correction $m\log a_0$: a small cell has probability
approximately $a_0^mf$.

We next obtain a quantitative upper bound without requiring derivative
bounds on individual component means.  Concavity of the logarithm and
symmetrization imply
\begin{align}
 0\leq{}&H(J(\bm Z))+m\log a_0-h(\bm Z)\notag\\
 \leq{}&\frac{a_0^{-m}}2
 \int_{[-a_0,a_0]^m}\int_{\R^m}
 [f(\bm x)-f(\bm x+\bmy)]
 \log\frac{f(\bm x)}{f(\bm x+\bmy)}
 \,\dd\bm x\,\dd\bmy.
 \label{ch:translation-entropy}
\end{align}
For clarity, this follows by applying
$\log(a_0^{-m}\int_C f)\geq a_0^{-m}\int_C\log f$ in each cell,
then averaging the resulting expression with the same expression in
which the two points in the cell are interchanged.  The integrand in
\eqref{ch:translation-entropy} is nonnegative, so enlarging the set of
pairs from points in the same cell to all pairs separated by a vector
in $[-a_0,a_0]^m$ can only increase the integral.

The inner integral is the sum of the two relative entropies between
$f$ and its translate $f(\cdot+\bmy)$. Relative entropy is jointly
convex: the log-sum inequality bounds the relative entropy of two
mixtures with the same weights by the weighted component relative
entropies. Pair each Gaussian component with its translate. The two
Gaussian relative entropies are both $\tfrac12\bmy^TK^{-1}\bmy$,
so their sum is $\bmy^TK^{-1}\bmy$. Consequently,
\[
 \int_{\R^m}[f(\bm x)-f(\bm x+\bmy)]
 \log\frac{f(\bm x)}{f(\bm x+\bmy)}\,\dd\bm x
 \leq\frac{|\bmy|^2}{\kappa\eta_0}.
\]
Substitution in \eqref{ch:translation-entropy} yields
\begin{equation}\label{ch:unclipped-error}
 0\leq H(J(\bm Z))+m\log a_0-h(\bm Z)
 \leq C_{m,\kappa}\frac{a_0^2}{\eta_0}
 =C_{m,\kappa}N^{-7\vartheta}.
\end{equation}

\proofstep{The clipping probability}
Let $B=\{TJ(\bm Z)\neq J(\bm Z)\}$ and $p=\PP(B)$.  If $B$ occurs,
some coordinate of $\bm Z$ has absolute value at least $3R_0/4$ for
all sufficiently large $N$.  Its component mean has absolute value at
most $R_0/2$ and its component variance is at most $\eta_0^{-1}=R_0$.
The one-dimensional Gaussian tail bound, followed by averaging over
components and a union bound over coordinates, gives
\begin{equation}\label{ch:clipping-probability}
 p\leq 2m e^{-c_mR_0}=2m e^{-c_mN^\vartheta}.
\end{equation}
This estimate is uniform in the mixture weights and means.

\proofstep{Entropy carried by the clipped tail}
Because $TJ(\bm Z)$ is a deterministic function of $J(\bm Z)$
and the two labels agree on $B^c$, the chain rule applied also to
the indicator $\mathbf1_B$ gives
\[
 H(J(\bm Z))-H(TJ(\bm Z))=H(J(\bm Z)\mid TJ(\bm Z))
 \leq H_2(p)+pH(J(\bm Z)\mid B).
\]
Here $H(J(\bm Z)\mid B)$ denotes conditioning on the event $B$,
not on its indicator; the product is taken to be zero when $p=0$.
For any $\mathbb Z^m$-valued random vector $\bm V$, comparison with the
probability mass function proportional to $e^{-|\bm j|_1}$ gives
\begin{equation}\label{ch:integer-entropy}
 H(\bm V)\leq C_m+\E|\bm V|_1.
\end{equation}
Indeed, nonnegativity of relative entropy with respect to that reference
law is exactly \eqref{ch:integer-entropy}.  Applying it to the
conditional law of $J(\bm Z)$ given the event $B$ gives
\[
 H(J(\bm Z))-H(TJ(\bm Z))
 \leq H_2(p)+C_mp+\E[|J(\bm Z)|_1\mathbf1_B].
\]
Since $|J(Z_r)|\leq a_0^{-1}|Z_r|+1$ and the component means and
variances satisfy the assumed bounds,
\[
 \E|J(\bm Z)|_1^2\leq C_m a_0^{-2}(R_0^2+\eta_0^{-1})+C_m
 \leq C_mN^{10\vartheta}.
\]
Cauchy--Schwarz and \eqref{ch:clipping-probability} therefore imply
\begin{equation}\label{ch:clipping-entropy}
 0\leq H(J(\bm Z))-H(TJ(\bm Z))
 \leq C_mN^{5\vartheta}e^{-c_mN^\vartheta}.
\end{equation}
Combining \eqref{ch:unclipped-error} and
\eqref{ch:clipping-entropy} proves \eqref{ch:quant-bound}.  No
independence between the coordinates of $\bm Z$ has been used.
\end{proof}

\subsubsection{Gaussian smoothing at high noise}

\begin{lemma}[High-noise smoothing]\label{en:smoothing}
Let $\bm Y$ be an $m$-dimensional random vector satisfying
\[
 \E\bm Y=0,\qquad \Cov(\bm Y)=I_m,
 \qquad |\bm Y|\leq B\quad\text{almost surely},
\]
where $B\geq2$.  Let $\bm G$ be an independent standard Gaussian and
put
\[
 f_s=\frac{\dd\Law(\sqrt s\,\bm Y+
             \sqrt{1-s}\,\bm G)}{\dd\gamma_m},\qquad 0<s<1.
\]
There are constants $c,C>0$, depending only on $m$, such that
\begin{equation}\label{en:mehler}
 \|f_s-1\|_{L^3(\gamma_m)}\leq CBs^{3/2},
 \qquad 0<s\leq cB^{-2}.
\end{equation}
\end{lemma}

\begin{proof}
The first two moment assumptions are used only to cancel the terms of
orders $\sqrt s$ and $s$.  The remaining third-order Taylor term is of
order $s^{3/2}$. Throughout this proof, $\|\cdot\|_3$ is the
$L^3(\gamma_m)$ norm in $\bm z$. After the Gaussian convolution is
written as a density, every remaining expectation is over $\bm Y$.

Put $a=\sqrt s$.  For fixed $\bmy$, the density of
$a\bmy+\sqrt{1-a^2}\,\bm G$ relative to $\gamma_m$ is
\begin{equation}\label{en:mehler-kernel}
 K(a;\bmy,\bm z)
 =(1-a^2)^{-m/2}
 \exp\left(
 \frac{2a\bm z\cdot\bmy-a^2(|\bm z|^2+|\bmy|^2)}{2(1-a^2)}
 \right).
\end{equation}
Consequently $f_s(\bm z)=\E K(a;\bm Y,\bm z)$.

\proofstep{Cancellation of the first two Taylor coefficients}
Direct differentiation at $a=0$ gives
\begin{align*}
 K(0;\bmy,\bm z)&=1,\\
 \partial_aK(0;\bmy,\bm z)&=\bm z\cdot\bmy,\\
 \partial_a^2K(0;\bmy,\bm z)
   &=(\bm z\cdot\bmy)^2- |\bmy|^2-|\bm z|^2+m.
\end{align*}
Now
$\E(\bm z\cdot\bm Y)=0$,
$\E(\bm z\cdot\bm Y)^2=|\bm z|^2$, and
$\E|\bm Y|^2=m$.  Hence the expectations of the first and second
derivatives vanish.  This is the precise reason that matching both the
mean and covariance is necessary.

\proofstep{A uniform third-derivative bound}
Write $\ell=\log K$.  When $0\leq t\leq a$, $aB\leq1/4$, and
$|\bmy|\leq B$, we have $t\leq1/4$ and $t|\bmy|\leq1/4$.  From
\eqref{en:mehler-kernel},
\[
 \partial_t\ell
 =\frac{mt}{1-t^2}
  +\frac{(1+t^2)\bm z\cdot\bmy
          -t(|\bm z|^2+|\bmy|^2)}{(1-t^2)^2}.
\]
Differentiating twice more shows, for $j=1,2,3$,
\[
 |\partial_t^j\ell|
 \leq C_m(1+|\bmy|)^j(1+|\bm z|^2).
\]
Also, the negative quadratic terms in the exponent and
$t|\bmy|\leq1/4$ give $K(t;\bmy,\bm z)\leq C_me^{C_m|\bm z|}$.
Since
\[
 \partial_t^3K
 =K\bigl[(\partial_t\ell)^3
       +3(\partial_t\ell)(\partial_t^2\ell)
       +\partial_t^3\ell\bigr],
\]
Gaussian integrability of every polynomial times $e^{C|\bm z|}$
yields
\begin{equation}\label{en:kernel-third-derivative}
 \|\partial_t^3K(t;\bmy,\cdot)\|_{L^3(\gamma_m)}
 \leq C_m(1+|\bmy|^3).
\end{equation}
All constants in this estimate are independent of $B$ once
$t|\bmy|\leq1/4$.

\proofstep{Taylor's integral remainder}
Taylor's formula and the two cancellations give
\[
 f_s(\bm z)-1
 =\frac12\E\int_0^a(a-t)^2
       \partial_t^3K(t;\bm Y,\bm z)\,\dd t.
\]
Minkowski's inequality and \eqref{en:kernel-third-derivative} imply
\[
 \|f_s-1\|_3
 \leq Ca^3(1+\E|\bm Y|^3).
\]
Finally, $|\bm Y|\leq B$ and covariance normalization give
$\E|\bm Y|^3\leq B\E|\bm Y|^2=mB$.  Since $a^3=s^{3/2}$ and
$B\geq2$, this proves \eqref{en:mehler} whenever
$aB\leq1/4$, which is ensured by $s\leq cB^{-2}$.
\end{proof}

\subsubsection{Gaussian star laws}\label{en:star-laws}

Fix $d\ge3$, $k\ge1$ and
$K\Subset(-2\sqrt{d-1},2\sqrt{d-1})$. Let
$\bla=(\lambda_1,\ldots,\lambda_k)\in K^k$ be fixed parameters.
For $\lambda\in K$, define the $(d+1)\times(d+1)$ covariance matrix
\begin{equation}\label{def:star-covariance}
 \begin{gathered}
 (\Sigma_\lambda)_{ii}=1\quad(0\le i\le d),\qquad
 (\Sigma_\lambda)_{0i}=(\Sigma_\lambda)_{i0}=\lambda/d
 \quad(1\le i\le d),\\
 (\Sigma_\lambda)_{ij}=(\lambda^2-d)/(d(d-1))
 \quad(1\le i\ne j\le d).
 \end{gathered}
\end{equation}
Set
\[
 H_\lambda=\bigl\{\bmy\in\R^{d+1}:\textstyle\sum_{i=1}^d y_i=\lambda y_0\bigr\},
 \qquad H_{\bla}=\bigoplus_{a=1}^k H_{\lambda_a}.
\]
The kernel of $\Sigma_\lambda$ is spanned by $(-\lambda,1,\ldots,1)^T$.
Its nonzero eigenvalues are $(d+\lambda^2)/d$ once and
$(d^2-\lambda^2)/(d(d-1))$ with multiplicity $d-1$.

Here is the underlying decomposition. Vectors with zero root
coordinate and leaf coordinates summing to zero form a $(d-1)$-dimensional
subspace on which the eigenvalue is
$1-(\lambda^2-d)/(d(d-1))$. Its orthogonal complement is the
root--constant-leaf plane. On that plane one eigenvalue is zero,
because the vector $(-\lambda,1,\ldots,1)$ is normal to $H_\lambda$;
the other follows either by a $2\times2$ calculation or by subtracting
the transverse contribution from $\Tr\Sigma_\lambda=d+1$.
Thus the Gaussian is nondegenerate on $H_\lambda$, but singular in
$\R^{d+1}$. Entropy must consequently be computed in $H_\lambda$,
not in the ambient space.
 Let $\bm{Z}_{\bla}$ have independent centered Gaussian
components with covariances $\Sigma_{\lambda_a}$. For a law $P$ on
$H_{\bla}$ with finite entropies define
\[
 \F_{\bla}(P)=h_{H_{\bla}}(P)-\frac d2h(P_e),
\]
where $P_e$ is the marginal on the $k$ root coordinates and their first
neighbor coordinates, of total dimension $2k$, and subspace
entropy uses induced Euclidean volume, as in \cite[Section~7]{BS}.
For one channel this is the functional of \cite[Definition~11.1]{BS};
we use its joint version on the direct sum of the wave subspaces.
As with entropy, $\F_{\bla}(\bm X)$ means $\F_{\bla}(\Law(\bm X))$
when the required star and edge densities exist.
This functional is unchanged by
a common scalar dilation.

Indeed, dilation by $r>0$ adds $dk\log r$ to the star entropy and
$2k\log r$ to the edge entropy, and these terms cancel.
The product structure of $\bm Z_{\bla}$ makes the functional additive in its $k$ channels.

\begin{lemma}[Entropy of the Gaussian star]\label{en:star-entropy}
Let
$L_d(\lambda)=\int\rho_d(t)\log|t-\lambda|\,\dd t$.  Then, for every
$\lambda$ in the open Kesten--McKay bulk,
\begin{equation}\label{def:gauss-entropy}
 L_d(\lambda)
 =\frac{d-1}{2}\log(d-1)
  -\frac{d-2}{4}\log(d^2-\lambda^2),
\end{equation}
and
\begin{equation}\label{en:star-functional-value}
 \F_{\bla}(\bm Z_{\bla})
 =\sum_{a=1}^k
 \left(\frac12\log(d+\lambda_a^2)-L_d(\lambda_a)\right).
\end{equation}
\end{lemma}

\begin{proof}
\proofstep{The logarithmic potential of the Kesten--McKay law}
Put
\[
 g(z)=\frac{-z+\sqrt{z^2-4(d-1)}}{2(d-1)},
\]
where $z\in\mathbb C^+$ and the square root is chosen to be
asymptotic to $z$ at infinity. The logarithms below are analytic
branches fixed by their asymptotics as $z=\ri y$, $y\to\infty$.
The tree resolvent recursion gives
\[
 (d-1)g^2+zg+1=0,
 \qquad m_d(z)=\frac{g(z)}{1-g(z)^2}.
\]
Differentiating the recursion, or substituting it directly, gives
\[
 \frac{\dd}{\dd z}
 \left[-\log(-g(z))-\frac{d-2}{2}\log(1-g(z)^2)\right]
 =-m_d(z).
\]
On the other hand,
$\frac{\dd}{\dd z}\int\rho_d(t)\log(z-t)\,\dd t=-m_d(z)$
with our Stieltjes-transform convention.  Comparison as $z\to\infty$
fixes the integration constant.  Taking the real part of the boundary
value at a bulk point $\lambda$, and using
\[
 |g(\lambda+\ri0)|=(d-1)^{-1/2},
 \qquad
 |1-g(\lambda+\ri0)^2|
   =\frac{\sqrt{d^2-\lambda^2}}{d-1},
\]
proves \eqref{def:gauss-entropy}. Taking this boundary value under
the integral is justified by local integrability of the logarithmic
singularity and boundedness of $\rho_d$ near the fixed bulk point.

\proofstep{The star--edge Gaussian entropy difference}
For one channel, the nonzero eigenvalues of $\Sigma_\lambda$ are
\[
 \frac{d+\lambda^2}{d}
 \quad\text{and}\quad
 \frac{d^2-\lambda^2}{d(d-1)}
 \quad\text{with multiplicity }d-1.
\]
Therefore the determinant of the covariance restricted to the
$d$-dimensional space $H_\lambda$ is
\[
 \det(\Sigma_\lambda|_{H_\lambda})
 =\frac{d+\lambda^2}{d}
   \left(\frac{d^2-\lambda^2}{d(d-1)}\right)^{d-1}.
\]
The root--first-neighbor covariance is
$\begin{psmallmatrix}1&\lambda/d\\ \lambda/d&1\end{psmallmatrix}$,
whose determinant is $(d^2-\lambda^2)/d^2$.
Using the Gaussian entropy formula on $H_\lambda$ and in the
$2$-dimensional edge space, the factors $2\pi e$ cancel because
$\dim H_\lambda=d=(d/2)\cdot2$.  We obtain
\[
 \F_\lambda(\bm Z_\lambda)
 =\frac12\log\det(\Sigma_\lambda|_{H_\lambda})
  -\frac d4\log\frac{d^2-\lambda^2}{d^2}.
\]
Substituting the two determinants and simplifying with
\eqref{def:gauss-entropy} gives
$\F_\lambda(\bm Z_\lambda)=\frac12\log(d+\lambda^2)-L_d(\lambda)$.
The $k$ channels of $\bm Z_{\bla}$ are independent, so both star and
edge entropy are additive.  Summing the one-channel identity proves
\eqref{en:star-functional-value}.
\end{proof}

\subsubsection{Entropy of Gaussian resolvent noise}

\begin{lemma}[Entropy per vertex of resolvent noise]\label{sp:noise}
Let $A$ be a fixed real symmetric $N\times N$ matrix with eigenvalues
$d=\lambda_1\ge\lambda_2\ge\cdots\ge\lambda_N\ge-d$ satisfying
$\max_{2\le j\le N}|\lambda_j-\gamma_j|\le N^{-a}$ for some $a>0$.
Put $Q_\lambda=\Im(A-\lambda-\ri\eta)^{-1}$, with
$\eta=N^{-c_0}$ and $0<c_0<a$. For $\bm{Y}\sim N(0,Q_\lambda)$,
\[
 \frac1Nh(\bm{Y})=\frac12\log(2\pi e\eta)-L_d(\lambda)
 +O(N^{-a}/\eta+\eta+N^{-1}\log N)
\]
uniformly for $\lambda\in K$.
\end{lemma}
\begin{proof}
Equation~\eqref{en:gaussian-entropy} gives exactly
\[
 \frac1Nh(\bm{Y})=\frac12\log(2\pi e\eta)
 -\frac1N\sum_j\log|\lambda_j-\lambda-\ri\eta|.
\]
The summand has Lipschitz constant at most $\eta^{-1}$, so replacing the
nontrivial eigenvalues by their quantiles incurs an error of at most
$N^{-a}/\eta$. Its total variation on the spectral support is
$O(1+|\log\eta|)$, which bounds the quantile Riemann-sum error by
$O(N^{-1}\log N)$. The trivial eigenvalue contributes $O(N^{-1})$.
It remains to bound the smoothing error:
\[
 \frac12\int\rho_d(t)\log\left(1+\frac{\eta^2}{(t-\lambda)^2}\right)\dd t
 \leq C\eta\int_\R\log(1+r^{-2})\dd r=O(\eta).
\]
The deterministic approximation just used can also be viewed as a
transport estimate. The quantiles give a Riemann sum for the probability
measure $\rho_d(t)\,\dd t$. For a function of bounded variation on the
spectral interval, its quantile Riemann error is bounded by its total
variation divided by $N$. For
$\log|t-\lambda-\ri\eta|$ that variation is $O(1+|\log\eta|)$.
The single trivial eigenvalue changes the normalized sum by at most
$O(N^{-1}\log N)$ as well. Thus no microscopic rigidity estimate is
being used in this entropy computation; the coarse displacement
$N^{-a}$ is sufficient because the noise height is larger.
\end{proof}

\subsubsection{Gaussian star geometry and fourth-chaos coercivity}\label{en:geometry}

We retain the fixed parameters and Gaussian star laws from
Subsubsection~\ref{en:star-laws}. The weighted frame and
Fisher-information argument adapt
\cite[Lemma~11.3 and Proposition~11.1]{BS} componentwise, allowing
different parameters $\lambda_a$. The mixed-derivative calculation and
fourth-chaos coercivity below are finite-dimensional statements,
with constants uniform over $\bla\in K^k$.

The wave constraint depends on $\lambda$, so different star laws
initially live on different hyperplanes. We want to move all of them
to one fixed Euclidean space, with the standard Gaussian as reference.
Let $U=\R^{d-1}$ contain the vertices
$\bme_1,\ldots,\bme_d$ of a regular unit simplex. Thus
\[
 \bme_i\cdot\bme_j=\begin{cases}1,&i=j,\\-1/(d-1),&i\ne j,\end{cases}
 \qquad \sum_{i=1}^d\bme_i=0,\qquad
 \sum_{i=1}^d\bme_i\bme_i^T=\frac d{d-1}I_U.
\]

A concrete realization helps interpret these formulas.  Identify $U$
isometrically with the zero-sum hyperplane
\[
 \left\{(u_1,\ldots,u_d)\in\R^d:\sum_{i=1}^du_i=0\right\}
\]
and take
\[
 \bme_i=\sqrt{\frac d{d-1}}
 \left(\bm\varepsilon_i-\frac1d\mathbf1\right),
\]
where $\bm\varepsilon_i$ is the $i$th coordinate vector.  The displayed
simplex identities then follow by direct calculation.  We use the
abstract space $U\simeq\R^{d-1}$ only to avoid carrying the redundant
zero-sum coordinate.

A \emph{standard Gaussian on $U$} means a centered Gaussian whose
covariance is the identity operator on the Euclidean space $U$.  In the
concrete model above, start with $\bm G\sim N(0,I_d)$ and orthogonally
project it onto this realization of $U$:
\[
 \bm G_U=\bm G-\frac1d(\mathbf1\cdot\bm G)\mathbf1.
\]
For every $\bmu,\bmv\in U$,
\[
 \E[(\bmu\cdot\bm G_U)(\bmv\cdot\bm G_U)]=\bmu\cdot\bmv.
\]
Thus $\bm G_U$ is standard intrinsically on $U$, even though its $d$
ambient coordinates are dependent.  In particular,
$\Cov(\bme_i\cdot\bm G_U,\bme_j\cdot\bm G_U)=\bme_i\cdot\bme_j$.

For a single channel, the space is $\R+U$, where $+$ denotes the
orthogonal direct sum. A point is a pair $(y_0,\bmy)$ with
$y_0\in\R$ and $\bmy\in U$. The root coordinate is $y_0$, and the
$i$th leaf coordinate is
\[
 cy_0+\rho\,\bme_i\cdot\bmy,
 \qquad c=\lambda/d,\qquad \rho=\sqrt{1-c^2}.
\]
The leaf coordinates sum to $dcy_0=\lambda y_0$, because the simplex
vertices sum to zero. For a standard Gaussian pair $(y_0,\bmy)$,
the root--leaf covariance is $c$, each leaf has variance
$c^2+\rho^2=1$, and two distinct leaves have covariance
\[
 c^2-\frac{\rho^2}{d-1}
 =\frac{\lambda^2-d}{d(d-1)}.
\]
Thus this map gives exactly the covariance $\Sigma_\lambda$.

For $k$ channels, take $k$ copies of the same space and set
\[
 E=(\R+U)^k,\qquad V_i=(\R+\R\bme_i)^k.
\]
We write a point of $E$ as
$\bmy=((y_{a,0},\bmy_a))_{a=1}^k$, where
$y_{a,0}\in\R$ and $\bmy_a\in U$. The inner product is the sum of
the inner products in the $k$ blocks. We identify $\R^k$ with the
subspace in which every $\bmy_a=0$, and $U^k$ with the subspace in
which every $y_{a,0}=0$. These are the scalar and transverse
coordinates, respectively. The space $E$ and its edge subspaces
$V_i$ do not depend on $\bla$.

Let $P_i$ be the orthogonal projection onto $V_i$. It acts separately
in each block:
\[
 P_i\bmy
 =\bigl((y_{a,0},(\bme_i\cdot\bmy_a)\bme_i)\bigr)_{a=1}^k.
\]
In this subsubsection $E$ denotes the Euclidean space, not a spectral
energy. Its dimension is $dk$, while $\dim V_i=2k$. Write $\gamma_E$ for
the standard Gaussian law on $E$, with the analogous convention on
its subspaces. A standard Gaussian on $E$ has independent standard
Gaussian scalar and transverse coordinates across all $k$ blocks.
For $1\le a\le k$, put
\[
 c_a=\lambda_a/d,\qquad \rho_a=(1-c_a^2)^{1/2}.
\]
Let $\bmw_a\in E$ have $a$th block $(1,\bm0)$ and all other blocks
zero. Let $\bmv_{i,a}\in E$ have $a$th block
$(c_a,\rho_a\bme_i)$ and all other blocks zero.
Thus $\bmw_a$ and $\bmv_{i,a}$ represent the root and the $i$th leaf
in channel $a$, and $V_i$ is their span over all $a$.

The physical coordinate map $\mathcal L_{\bla}$ is consequently
\[
 (\mathcal L_{\bla}\bmy)_{a,0}=y_{a,0},\qquad
 (\mathcal L_{\bla}\bmy)_{a,i}
 =c_a y_{a,0}+\rho_a\,\bme_i\cdot\bmy_a.
\]
Equivalently, its root coordinates are $\bmw_a\cdot\bmy$ and its
neighbor coordinates are $\bmv_{i,a}\cdot\bmy$.
For a fixed $i$, knowing these two physical coordinates in every channel
is equivalent to knowing $P_i\bmy$.  Indeed, the root coordinate gives
$y_{a,0}$, and then
\[
 \bme_i\cdot\bmy_a
 =\frac{(\mathcal L_{\bla}\bmy)_{a,i}
       -c_a(\mathcal L_{\bla}\bmy)_{a,0}}{\rho_a}.
\]
This is why the physical oriented-edge marginal becomes the ordinary
orthogonal projection onto the fixed subspace $V_i$ after whitening.
The one-channel calculation above applies independently in every
block, so
\[
 \mathcal L_{\bla}\mathcal L_{\bla}^T
 =\diag(\Sigma_{\lambda_1},\ldots,\Sigma_{\lambda_k}).
\]
In each block, $\mathcal L_{\bla}^T\mathcal L_{\bla}$ has scalar and
transverse eigenvalues $1+dc_a^2$ and $d\rho_a^2/(d-1)$.
Thus $\mathcal L_{\bla}$ is invertible onto $H_{\bla}$, the map and
its inverse are uniformly bounded on $K^k$, and it sends a standard
Gaussian on $E$ to $\bm Z_{\bla}$.

An explicit inverse clarifies this change of coordinates. For
$\bm x\in H_{\bla}$, the $a$th block of the inverse is
\[
 (\mathcal L_{\bla}^{-1}\bm x)_a
 =\left(x_{a,0},\frac{d-1}{d\rho_a}
                  \sum_{i=1}^d x_{a,i}\bme_i\right).
\]
Indeed, the simplex identities and
$\sum_i x_{a,i}=dc_ax_{a,0}$ give
\[
 \bme_j\cdot\left(\frac{d-1}{d\rho_a}
                  \sum_{i=1}^d x_{a,i}\bme_i\right)
 =\frac{x_{a,j}-c_ax_{a,0}}{\rho_a}.
\]
Substitution in the physical coordinate map recovers every leaf
coordinate. The change of variables adds a deterministic logarithmic
Jacobian to star and edge entropies; these constants disappear when
the functional of a law is subtracted from that of its Gaussian
reference.

Let $R_i$ be the orthogonal involution exchanging $\bmw_a,\bmv_{i,a}$ for
all $a$ and fixing $V_i^\perp$. It represents simultaneous endpoint
reversal on $V_i$. Permuting the simplex vertices gives the
orthogonal action of neighbor permutations on $E$. The statements
below concern laws invariant under this action whose $V_i$ marginals
are invariant under $R_i$.

Define directions in $V_i$ by
\[
\begin{aligned}
\bma_{i,a}
 &=\frac{\bmw_a+\bmv_{i,a}}{2\sqrt{1+c_a}}
   +\frac{\bmw_a-\bmv_{i,a}}{2\sqrt{1-c_a}},\\
\bmb_{i,a}
 &=\frac{\bmw_a+\bmv_{i,a}}{2\sqrt{1+c_a}}
   -\frac{\bmw_a-\bmv_{i,a}}{2\sqrt{1-c_a}}.
\end{aligned}
\]
The sum $\bmw_a+\bmv_{i,a}$ and difference $\bmw_a-\bmv_{i,a}$
are orthogonal, with squared lengths $2(1+c_a)$ and $2(1-c_a)$.
Each summand in the definition of $\bma_{i,a}$ therefore has squared
length $1/2$. Their sum and difference, $\bma_{i,a}$ and $\bmb_{i,a}$,
are orthonormal. Different channels are orthogonal, so
$\{\bma_{i,a},\bmb_{i,a}:1\leq a\leq k\}$ is an orthonormal basis
of $V_i$. Endpoint reversal fixes the sum and negates the difference;
consequently
\[
 R_i\bma_{i,a}=\bmb_{i,a},\qquad
 R_i\bmb_{i,a}=\bma_{i,a}.
\]

Put
\[
t_{1,a}=\frac{\rho_a-1+2/d}{2\rho_a},\qquad
t_{2,a}=\frac{\rho_a+1-2/d}{2\rho_a}.
\]
The open bulk condition gives
$\rho_a>\sqrt{1-4(d-1)/d^2}=1-2/d$.  Since $K$ is compactly
contained in the bulk, both weights are bounded below by a positive
constant depending only on $d,K$, and $t_{1,a}+t_{2,a}=1$.
The weights are chosen so that the following positive frame is exactly
the identity:
\begin{equation}\label{en:frame}
\sum_{i=1}^d\sum_{a=1}^k
\left(t_{1,a}\bma_{i,a}\bma_{i,a}^T
     +t_{2,a}\bmb_{i,a}\bmb_{i,a}^T\right)=I_E.
\end{equation}

We verify this identity explicitly.  Work in one channel and decompose
vectors into their scalar part in $\R$ and their transverse part in
$U$.  A direct calculation from the definitions gives
\[
 |\bmw_a\cdot\bma_{i,a}|^2=\frac{1+\rho_a}{2},\qquad
 |\bmw_a\cdot\bmb_{i,a}|^2=\frac{1-\rho_a}{2}.
\]
Because both vectors have norm one, their transverse squared lengths
are the two quantities in the opposite order.  Therefore the weighted
scalar coefficient contributed by a fixed neighbor $i$ is
\[
 t_{1,a}\frac{1+\rho_a}{2}
 +t_{2,a}\frac{1-\rho_a}{2}=\frac1d,
\]
whereas the weighted transverse coefficient is
\[
 t_{1,a}\frac{1-\rho_a}{2}
 +t_{2,a}\frac{1+\rho_a}{2}=\frac{d-1}{d}.
\]
The scalar--transverse cross term is a scalar multiple of $\bme_i$.
It vanishes after summing over $i$ because $\sum_i\bme_i=0$.  Finally,
\[
 \sum_{i=1}^d\frac1d I_{\R}=I_{\R},\qquad
 \frac{d-1}{d}\sum_{i=1}^d\bme_i\bme_i^T=I_U.
\]
This proves the frame identity in one channel, and the sum over the
orthogonal channel blocks proves \eqref{en:frame} on $E$.

The special case $\lambda_a=0$ is a useful check.  Then
$c_a=0$, $\rho_a=1$, and, in the $a$th block,
\[
 \bma_{i,a}=\bmw_a,\qquad
 \bmb_{i,a}=(0,\bme_i),\qquad
 t_{1,a}=\frac1d,\qquad t_{2,a}=\frac{d-1}{d}.
\]
Thus \eqref{en:frame} reduces to one copy of the scalar direction with
total weight one and the standard simplex tight frame on $U$.

Let $p$ be a smooth positive Lebesgue density on $E$ with finite
Fisher information, and let $p_i$ be its marginal density on $V_i$.
Assume that $p_i\circ R_i=p_i$. In this paragraph $\E_p$ means
expectation under the density $p$.  We regard $\log p_i(P_i\bmx)$ as a
function on $E$.  If $\bm Y$ has density $p$, differentiating the
marginal under the integral gives the score-projection identity
\begin{equation}\label{en:conditional-score}
 \E_p[P_i\nabla\log p(\bm Y)\mid P_i\bm Y]
 =\nabla\log p_i(P_i\bm Y).
\end{equation}
For completeness, for any $\bmu\in V_i$ and smooth compactly supported
$\phi$ on $V_i$, integration by parts gives
\[
 \E_p[\phi(P_i\bm Y)\,\partial_{\bmu}\log p(\bm Y)]
 =-\E_p[\partial_{\bmu}\phi(P_i\bm Y)]
 =\E_p[\phi(P_i\bm Y)\,\partial_{\bmu}\log p_i(P_i\bm Y)],
\]
which is exactly \eqref{en:conditional-score}.

Write
\[
 I(p)=\int_Ep|\nabla\log p|^2\,\dd\bmx,
 \qquad
 I(p_i)=\int_{V_i}p_i|\nabla\log p_i|^2.
\]
For every unit vector $\bmu\in V_i$, conditional expectation is an
orthogonal projection in $L^2(p)$, so
\begin{align*}
 &\int_Ep\,
 |\partial_{\bmu}(\log p-\log p_i\circ P_i)|^2\,\dd\bmx\\
 &\qquad=
 \int_Ep|\partial_{\bmu}\log p|^2\,\dd\bmx
 -\int_{V_i}p_i|\partial_{\bmu}\log p_i|^2.
\end{align*}
Sum this identity in the directions of the weighted frame.  The full
terms sum to $I(p)$ by \eqref{en:frame}.  For the marginal terms,
$R_i$ preserves $p_i$ and exchanges $\bma_{i,a}$ with
$\bmb_{i,a}$, hence
\[
 \int p_i|\partial_{\bma_{i,a}}\log p_i|^2
 =\int p_i|\partial_{\bmb_{i,a}}\log p_i|^2.
\]
Since $t_{1,a}+t_{2,a}=1$ and the paired directions form an
orthonormal basis of $V_i$, their weighted marginal contribution is
$I(p_i)/2$.  We have proved the exact sum-of-squares identity
\begin{equation}\label{en:fisher}
\begin{split}
I(p)-\frac12\sum_{i=1}^dI(p_i)
=\sum_{i=1}^d\sum_{a=1}^k\int_Ep\bigl(
&t_{1,a}|\partial_{\bma_{i,a}}(\log p-\log p_i\circ P_i)|^2\\
&+t_{2,a}|\partial_{\bmb_{i,a}}(\log p-\log p_i\circ P_i)|^2
\bigr)\,\dd\bmx\geq0.
\end{split}
\end{equation}
No independence among the $k$ channels is used here.

The same proof applies to scores relative to the standard Gaussian:
replace $\log p$ by the logarithm of the density of the same law
relative to $\gamma_E$, and use the analogous relative marginal
density on $V_i$.  The Gaussian score $-\bmx$ projects
compatibly under $P_i$, so the conditional-score calculation and the
frame identity are unchanged.  This relative version is the one used
for Ornstein--Uhlenbeck interpolation below.

\begin{lemma}[Kernel of mixed third derivatives]\label{en:tensor}
Let $T$ be a symmetric trilinear form on $E=(\R+U)^k$. Suppose that
\[
 T((I-P_i)\bmx,P_i\bmy,\bmz)=0
 \qquad\text{for every }i\text{ and every }\bmx,\bmy,\bmz\in E.
\]
Then $T$ depends only on the scalar coordinates of its arguments.
Equivalently, $T(\bmx,\bmy,\bmz)=0$ whenever at least one of its
arguments belongs to $U^k$.
\end{lemma}

\begin{proof}
Symmetry means that the value of $T$ is unchanged by any permutation
of its three arguments. Fixing the last argument $\bmz$ gives a
self-adjoint operator $B_{\bmz}$ on $E$, defined by
\[
 \langle B_{\bmz}\bmx,\bmy\rangle=T(\bmx,\bmy,\bmz).
\]
The hypothesis says that this operator has no block between $V_i$
and $V_i^\perp$. More explicitly,
$(I-P_i)B_{\bmz}P_i=0$; taking adjoints gives
$P_iB_{\bmz}(I-P_i)=0$. Hence $B_{\bmz}$ commutes with every $P_i$.
We first determine the operators with this property, and then use
symmetry in the third argument.

\proofstep{Separate scalar and transverse coordinates}
The simplex identities give, for
$\bmy=((y_{a,0},\bmy_a))_{a=1}^k$,
\[
 \sum_{i=1}^dP_i\bmy
 =\left(\left(dy_{a,0},\frac d{d-1}\bmy_a\right)\right)_{a=1}^k.
\]
Thus $\R^k$ and $U^k$ are the eigenspaces of $\sum_iP_i$ with
eigenvalues $d$ and $d/(d-1)$, respectively. These eigenvalues are
distinct because $d\ge3$. Since $B_{\bmz}$ commutes with their sum,
it preserves both subspaces. Consequently,
\[
 T(\bmx,\bmy,\bmz)=0
 \qquad(\bmx\in\R^k,\ \bmy\in U^k,\ \bmz\in E).
\]
By symmetry, $T$ vanishes whenever one argument is scalar and another
is transverse. Only the case in which all three arguments are
transverse remains to be considered.

\proofstep{Determine the action on transverse coordinates}
The intersection $V_i\cap U^k$ consists of the vectors
$((0,a_a\bme_i))_{a=1}^k$ with $\bma\in\R^k$.
It is invariant under $B_{\bmz}$, so there is a self-adjoint
$k\times k$ matrix $M_i(\bmz)$ such that
\[
 B_{\bmz}\bigl((0,a_a\bme_i)\bigr)_{a=1}^k
 =\bigl((0,(M_i(\bmz)\bma)_a\bme_i)\bigr)_{a=1}^k.
\]
Self-adjointness follows because
$\bma\mapsto((0,a_a\bme_i))_{a=1}^k$ preserves inner products.
For $i\ne j$, self-adjointness of $B_{\bmz}$ further gives
\[
 (\bme_i\cdot\bme_j)\langle M_i(\bmz)\bma,\bmb\rangle
 =(\bme_i\cdot\bme_j)\langle\bma,M_j(\bmz)\bmb\rangle.
\]
Since $\bme_i\cdot\bme_j=-1/(d-1)\ne0$, all the matrices
$M_i(\bmz)$ agree; write their common value as $M_{\bmz}$.
The simplex vertices span $U$, so for arbitrary transverse vectors
$\bmx=((0,\bmx_a))_{a=1}^k$ and
$\bmy=((0,\bmy_a))_{a=1}^k$, linearity yields
\[
 T(\bmx,\bmy,\bmz)
 =\sum_{a,b=1}^k(M_{\bmz})_{ab}\,\bmx_b\cdot\bmy_a.
\]
In words, the same matrix mixes the $k$ channels in every direction
of $U$.

\proofstep{Use symmetry to eliminate the transverse part}
Choose a unit vector $\bmu\in U$ and put
$\bmz=((0,z_a\bmu))_{a=1}^k$, with arbitrary scalar coefficients.
Since $\dim U\ge2$, choose a unit vector $\bmv\perp\bmu$.
For arbitrary $\bma,\bmb\in\R^k$, take
\[
 \bmx=((0,a_a\bmv))_{a=1}^k,\qquad
 \bmy=((0,b_a\bmv))_{a=1}^k.
\]
The preceding formula and symmetry give
\[
 \langle M_{\bmz}\bma,\bmb\rangle
 =T(\bmx,\bmy,\bmz)
 =T(\bmx,\bmz,\bmy)=0.
\]
For the last equality, every inner product between a transverse
block of $\bmx$ and a transverse block of $\bmz$ is zero, since
$\bmv\perp\bmu$. Thus $M_{\bmz}=0$ for all such $\bmz$.
These vectors span $U^k$, so linearity in the last argument proves
that $T$ vanishes when all three arguments are transverse as well.
Together with the first step, this is the claimed conclusion.

The assumption $d\ge3$ has been used twice: to separate the scalar
and transverse eigenspaces of $\sum_iP_i$, and to choose two
orthogonal directions in $U$. Preservation of every edge subspace
alone does not eliminate the transverse part; symmetry in all three
arguments is essential.
\end{proof}

We recall the finite-dimensional Gaussian terminology used in the
next lemma. The $n$th Gaussian chaos is the span of multivariate
Hermite polynomials of total degree $n$. These are ordinary
polynomials with lower-degree corrections chosen so that different
chaoses are orthogonal in $L^2(\gamma_E)$. The
Ornstein--Uhlenbeck operator $\Delta-\bm x\cdot\nabla$ acts as
multiplication by $-n$ on the $n$th chaos. Gaussian integration by
parts therefore gives
$\int|\nabla h|^2\,\dd\gamma_E=n\int h^2\,\dd\gamma_E$.
Conditional expectation onto a linear subspace is an orthogonal
projection in $L^2$ and preserves each chaos: in coordinates adapted
to the subspace, all Hermite factors of positive degree in a discarded
coordinate have expectation zero.

Let $\mathcal C_i$ denote Gaussian conditional expectation onto
$V_i=(\R+\R\bme_i)^k$. Thus $\mathcal C_i h$ keeps the scalar
coordinates and the components of each transverse block along
$\bme_i$, and averages over the remaining Gaussian coordinates.
More explicitly, if $\bm G$ is standard Gaussian on $E$, then
\[
 (\mathcal C_i h)(\bm y)
 =\E\bigl[h(\bm G)\mid P_i\bm G=P_i\bm y\bigr].
\]
The right side denotes the Gaussian conditional integral on the
orthogonal complement of $V_i$, and is lifted back to a function on $E$.
In particular, if $f$ is a density relative to $\gamma_E$, then
$\mathcal C_i f$ is its marginal density relative to $\gamma_{V_i}$.
Integrating a lifted marginal over $\gamma_E$ is the same as
integrating it over $\gamma_{V_i}$.
For a function $h$ on $E$, put
\[
Q(h)=\|h\|_2^2-\frac d2\|\mathcal C_1h\|_2^2.
\]
Until the end of Subsubsection~\ref{en:moments}, these function norms
are with respect to $\gamma_E$, and $Q(h)$ denotes this quadratic
functional, not a covariance matrix.

\begin{lemma}[Fourth-chaos coercivity]\label{en:quartic-coercivity}
Suppose $h$ belongs to Gaussian chaos $n\ge1$, is invariant under
neighbor permutations, and satisfies
$(\mathcal C_i h)\circ R_i=\mathcal C_i h$ for every $i$. Then
$Q(h)\ge0$. If $n=4$, there is $\kappa>0$, uniform for
$\bla\in K^k$, such that
\begin{equation}\label{en:coercivity}
Q(h)\ge\kappa\|h\|_2^2.
\end{equation}
\end{lemma}

\begin{proof}
The frame identity, Gaussian conditional expectation, and the
reflection symmetry of the marginals give
\[
nQ(h)=\sum_{i,a}\bigl(
 t_{1,a}\|(I-\mathcal C_i)\partial_{\bma_{i,a}}h\|_2^2
+t_{2,a}\|(I-\mathcal C_i)\partial_{\bmb_{i,a}}h\|_2^2\bigr).
\]
Indeed, the full derivative energy is $n\|h\|_2^2$. Reflection
equates the paired projected derivative energies, so their weighted
sum is $(n/2)\|\mathcal C_i h\|_2^2$ for each $i$; neighbor
symmetry then gives the displayed identity. This proves nonnegativity.

For degree four it suffices to show that $Q(h)=0$ forces $h=0$.
Since the weights are positive, equality implies
$P_i\nabla(h-\mathcal C_i h)=0$ for every $i$.
Thus every derivative of $h$ in a direction of $V_i$ depends only
on $P_i\bmx$. Differentiating in a direction perpendicular to $V_i$
and then in an arbitrary direction gives
\[
 D^3h(\bmx)[(I-P_i)\bmu,P_i\bmv,\bmz]=0
 \qquad(\bmx,\bmu,\bmv,\bmz\in E).
\]
Here $D^3h(\bmx)$ is the symmetric trilinear form obtained by taking
three directional derivatives of $h$ at $\bmx$.
Apply Lemma~\ref{en:tensor} at each point $\bmx$:
this third derivative vanishes whenever one of its arguments
belongs to $U^k$. Differentiating once more, and using symmetry of
fourth derivatives, shows that $D^4h$ vanishes whenever any of its
four arguments is transverse.

A fourth-chaos polynomial is determined by its fourth derivative:
in the product Hermite basis, each basis element has a distinct
leading monomial of total degree four. The preceding vanishing therefore
means that $h$ depends only on the $k$ scalar coordinates, with no
transverse coordinate present. These scalar coordinates are retained
by every $P_i$, so $\mathcal C_i h=h$ for every $i$. It follows that
\[
 0=Q(h)=\left(1-\frac d2\right)\|h\|_2^2,
\]
and hence $h=0$, since $d\ge3$.
The unit sphere of fourth chaos is finite dimensional,
and the symmetry constraints depend continuously on $\bla$.
Compactness over $\bla\in K^k$ now proves~\eqref{en:coercivity}.
To spell out the uniform compactness argument, suppose no positive
uniform $\kappa$ existed. There would be $\bla^{(r)}\in K^k$ and
admissible fourth-chaos polynomials $h_r$ with
$\|h_r\|_2=1$ and $Q(h_r)\to0$. Passing to subsequences in the
compact parameter set and the finite-dimensional unit sphere gives
$\bla^{(r)}\to\bla$ and $h_r\to h$ in $L^2$. The reflection and
permutation equations pass to the limit, as do conditional expectation
and $Q$. Thus $h$ is admissible, $\|h\|_2=1$, and $Q(h)=0$,
contradicting the kernel computation. This is a finite-dimensional
spectral gap, not a quantitative stability theorem for arbitrary
probability laws.
\end{proof}

\subsubsection{An abstract fourth-moment estimate}\label{en:moments}

The following proposition gives a quantitative fourth-moment
refinement of the Gaussian maximality argument of
Backhausz--Szegedy \cite[Theorem~11.2]{BS}. In the scalar setting, their
theorem identifies the Gaussian law as the unique maximizer of the
star--edge entropy functional under the prescribed covariance and
symmetry constraints. The heat-flow comparison below follows
\cite[Proposition~11.1]{BS}, with the weighted-frame calculation carried
out in the present block coordinates $E=(\R+U)^k$.

To extract an explicit moment bound, we combine the fourth-chaos
coercivity in Lemma~\ref{en:quartic-coercivity} with the high-noise
remainder estimate in Lemma~\ref{en:smoothing}. Uniqueness of the
Gaussian maximizer alone does not provide this coercivity or a
power-law estimate on moments. The argument applies to the joint law
of $k$ channels with possibly different eigenvalue parameters,
uniformly for $\bla\in K^k$.

Let $\bm{Y}$ be an $E$-valued vector with
the neighbor and edge-reflection symmetries above, and suppose
\[
\E \bm{Y}=0,\qquad \Cov(\bm{Y})=I_{E},\qquad |\bm{Y}|\le B,
\qquad B\ge2.
\]
For an independent standard Gaussian $\bm{G}$ on $E$, define
\[
f_s=\frac{\dd\Law(\sqrt s\,\bm{Y}+\sqrt{1-s}\,\bm{G})}
{\dd\gamma_{E}},\qquad
\Delta_s=\int f_s\log f_s\,\dd\gamma_{E}
-\frac d2\int(\mathcal C_1f_s)\log(\mathcal C_1f_s)\,\dd\gamma_{E}.
\]

\begin{proposition}[Fourth moments from an entropy deficit]\label{en:quartic}
If $0\le\Delta_{1/2}\le\varepsilon\le cB^{-6}$,
then every homogeneous polynomial $P$ of degree four whose coefficients
in the fixed orthonormal coordinates of $E$ are bounded by one satisfies
\begin{equation}\label{en:quartic-bound}
|\E P(\bm{Y})-\E P(\bm{G})|\le CB^{4/3}\varepsilon^{1/18}.
\end{equation}
Moreover, $0\le\Delta_s\le\varepsilon$ for $0<s\le1/2$. The constants
depend only on $d,k,K$.
\end{proposition}

The conclusion controls fourth moments, rather than a distance between
the full law and its Gaussian reference. We do not optimize the
exponent $1/18$; it results from balancing the entropy deficit against
the high-noise remainder in the proof.

\begin{proof}
\proofstep{1. Heat-flow monotonicity}
We follow the Fisher-information argument of
\cite[Proposition~11.1]{BS} in the block coordinates introduced above.
For a positive density $f$ relative to $\gamma_E$, write
$I(f)=\int|\nabla f|^2/f\,\dd\gamma_E$ for its relative Gaussian
Fisher information. The reference measure here is Gaussian, whereas
it was Lebesgue measure in~\eqref{en:fisher}. The same projected-score
calculation gives $I(f)\ge(d/2)I(\mathcal C_1f)$ under the stated
symmetries.
Under Ornstein--Uhlenbeck evolution,
\[
\frac{\dd}{\dd t}\Delta_{se^{-2t}}
=-I(f_{se^{-2t}})+\frac d2I(\mathcal C_1f_{se^{-2t}})\le0.
\]
Boundedness of $\bm{Y}$ justifies differentiation and integration by parts.
Lemma~\ref{en:smoothing} gives $f_s\to1$ in $L^2$ as $s\downarrow0$; the
relative entropy is at most $\|f_s-1\|_2^2$, giving
$\Delta_s\to0$ and $0\le\Delta_s\le\varepsilon$.

\proofstep{2. Isolating the fourth chaos}
Lemma~\ref{en:quartic-coercivity} gives $Q\geq0$ on each positive
Hermite degree with the stated symmetries and
$Q(h)\geq\kappa\|h\|_2^2$ on degree four.

Use probabilists' multivariate Hermite polynomials normalized by
$\int H_\alpha H_\beta\,\dd\gamma_{E}
=\alpha!\mathbf1_{\alpha=\beta}$. Set
\[
g_4=\sum_{|\alpha|=4}\frac{\E H_\alpha(\bm{Y})}{\alpha!}H_\alpha,
\qquad A_4=\|g_4\|_2^2.
\]
The Hermite generating function gives, for every multi-index,
\[
 \E H_\alpha(\sqrt s\,\bm Y+\sqrt{1-s}\,\bm G)
 =s^{|\alpha|/2}\E H_\alpha(\bm Y).
\]
Consequently the fourth chaos of $f_s-1$ is $s^2g_4$.
The zeroth chaos vanishes because $\int f_s=1$; the first and
second vanish because the smoothed vector is centered and isotropic.
If $g_n$ denotes the corresponding coefficient of degree $n$, then
orthogonality and preservation of degree by $\mathcal C_1$ give
\[
 Q(f_s-1)=\sum_{n\ge3}s^nQ(g_n).
\]
Each summand is nonnegative by Lemma~\ref{en:quartic-coercivity},
including the possible third-degree term. We may therefore retain
only degree four. Hence
\begin{equation}\label{en:q-fourth}
Q(f_s-1)\ge\kappa s^4A_4.
\end{equation}

\proofstep{3. Entropy expansion and choice of the noise level}
Lemma~\ref{en:smoothing} gives the bound~\eqref{en:mehler} for
$0<s\le cB^{-2}$. This compares the quadratic deficit with entropy.
The global remainder estimate
\[
\bigl|(1+u)\log(1+u)-u-u^2/2\bigr|\le C|u|^3,
\qquad u\ge-1,
\]
applies to $f_s-1$ and its marginal. Since conditional expectation
contracts $L^3$, equations~\eqref{en:q-fourth} and~\eqref{en:mehler} yield
\[
\varepsilon\ge\Delta_s
\ge\tfrac12Q(f_s-1)-C\|f_s-1\|_3^3
\ge\tfrac\kappa2s^4A_4-CB^3s^{9/2}.
\]
Dividing the preceding inequality by $s^4$ gives
$A_4\le C(\varepsilon s^{-4}+B^3s^{1/2})$.
For $\varepsilon>0$, balance the two terms by taking
$s=(\varepsilon/B^3)^{2/9}$. After decreasing the constant in
$\varepsilon\le cB^{-6}$, this choice satisfies
$s\le\min(1/2,cB^{-2})$. Both terms then equal a constant times
$B^{8/3}\varepsilon^{1/9}$, so
$A_4\le CB^{8/3}\varepsilon^{1/9}$; for $\varepsilon=0$, let
$s\downarrow0$. A homogeneous quartic has Hermite degrees four, two and zero.
Covariance normalization fixes the last two, so its moment error is
at most $C\sqrt{A_4}$, proving~\eqref{en:quartic-bound}.
\end{proof}

\subsection{Empirical stars and spectral inputs}\label{sp:section}

\subsubsection{Sampling a fixed graph}\label{sp:sampling}

Fix a simple $d$-regular graph $\mathcal G$ and an orthonormal selection
$\bmu_1,\ldots,\bmu_k$ of real eigenvectors with eigenvalues
$\bla=(\lambda_1,\ldots,\lambda_k)\in K^k$. The graph, eigenvectors,
and eigenvalues are held fixed throughout this subsubsection.
Here $a=1,\ldots,k$ labels the selected eigenpairs; it does not
refer to the full ordered spectrum, whose first eigenvalue is $d$.
Choose $v_0$ uniformly from $\llbracket N\rrbracket$, then order its
neighbors $v_1,\ldots,v_d$ uniformly among their $d!$ permutations, and put
$X_{a,j}=\sqrt N\bmu_a(v_j)$. The empirical star law $\nu_{\bm{X}}$
is the distribution under this sampling alone.

Write $\E_\star$ and $\Cov_\star$ for expectation and
covariance over the sampled root and neighbor ordering. In particular,
\[
 \E_\star f(X_{a,0})=\frac1N\sum_{v=1}^N f(\sqrt N\bmu_a(v)).
\]
For the full star, $\E_\star$ additionally averages the $d!$ neighbor
orderings at each vertex. Thus all empirical moments below are finite
averages on a fixed graph. The pair $(v_0,v_1)$ is a uniform oriented
edge. Probabilities of good graph events will be specified separately.

\begin{lemma}[Exact empirical covariance]\label{sp:covariance}
\begin{equation}\label{def:cov}
 \E_\star X_{a,i}=0,\qquad
 \E_\star[X_{a,i}X_{b,j}]=\mathbf1_{a=b}(\Sigma_{\lambda_a})_{ij}.
\end{equation}   
\end{lemma}

\begin{proof}
Every vertex appears as the root with probability $1/N$.
It also appears at any fixed neighbor position with probability $1/N$:
it has $d$ possible roots, and each such root places it in that
position with probability $1/d$. Therefore, for every $0\le i\le d$,
\[
 \E_\star X_{a,i}X_{b,i}
 =\sum_{v=1}^N\bmu_a(v)\bmu_b(v)=\mathbf1_{a=b}.
\]
Since each $\lambda_a\in K$ differs from $d$, its eigenvector is
orthogonal to the constant vector. The same uniformity proves
$\E_\star X_{a,i}=N^{-1/2}\sum_v\bmu_a(v)=0$.

For a root and a neighbor, use the eigenrelation at the root:
\[
 \E_\star X_{a,0}X_{b,1}
 =\frac1d\sum_{v=1}^N\bmu_a(v)
           \sum_{w\sim v}\bmu_b(w)
 =\frac{\lambda_b}{d}\mathbf1_{a=b}.
\]
Interchanging the two endpoints gives the same result with the
positions reversed. For two distinct neighbor positions, multiply
the eigenrelations and subtract terms with the same neighbor:
\begin{align*}
 d(d-1)\E_\star X_{a,1}X_{b,2}
 &=\E_\star\left[
     \left(\sum_{r=1}^dX_{a,r}\right)
     \left(\sum_{s=1}^dX_{b,s}\right)
       -\sum_{r=1}^dX_{a,r}X_{b,r}\right]\\
 &=(\lambda_a\lambda_b-d)\mathbf1_{a=b}.
\end{align*}
These three cases give exactly~\eqref{def:cov}.
Uniform neighbor ordering makes the law invariant under simultaneous
permutations of all channels' neighbors. Finally, $(v_0,v_1)$ is a
uniform oriented edge, and its reversal is uniform as well. This
proves the endpoint symmetry of the joint edge marginal. All these
identities are deterministic finite averages on the given graph;
no random-graph estimate has been used.
\end{proof}

The eigenrelations imply $\bm{X}\in H_{\bla}$. Define
$\bm{Y}=\mathcal L_{\bla}^{-1}\bm{X}$ using the deterministic map of
Subsubsection~\ref{en:geometry}. Equation~\eqref{def:cov} gives
\[
 \E_\star\bm{Y}=0,\qquad \Cov_\star(\bm{Y})=I_E.
\]
Uniform neighbor ordering gives neighbor-permutation symmetry, and
reversal of a uniform oriented edge gives the $R_i$-invariance of each
$V_i$ marginal after whitening. These are precisely the hypotheses on
symmetry and covariance in Proposition~\ref{en:quartic}; they hold for
every fixed graph and orthonormal selection as above.

\subsubsection{Spectral inputs for the random graph}

We now let $\mathcal G$ be uniformly random. Probabilities in
Proposition~\ref{sp:inputs} refer to this graph randomness. Once a graph
in the stated event is chosen, all its spectral estimates hold
deterministically, and the empirical-star expectations remain
$\E_\star$. In this subsubsection, $\lambda_1\ge\cdots\ge\lambda_N$
again denotes the full ordered spectrum of $A$.

For a fixed energy $E\in(-2\sqrt{d-1},2\sqrt{d-1})$, define
\[
 \sigma=\pi\rho_d(E),\qquad a_E=-\frac{(d-2)E}{2(d^2-E^2)},\qquad
 x_j=N\sigma(\lambda_j-E).
\]
The resolvents and microscopic variables are
\[
 \begin{gathered}
 G(z)=(A-z)^{-1},\qquad m_N(z)=\frac1N\Tr G(z),\qquad
 \mu_N=\sum_j\delta_{x_j},\\
 s_N(w)=\frac{m_N(E+w/(N\sigma))-a_E}{\sigma}.
 \end{gathered}
\]
Thus $s_N(w)=\sum_j(x_j-w)^{-1}-a_E/\sigma$. All complex spectral
parameters lie off the real axis, with reflection to the lower half-plane.
The measure $\mu_N$ is exactly the point measure in the
\hyperref[th:main-pointprocess]{main point-process theorem}. Its
coordinates agree with the intensity-$1/\pi$ convention of~\cite{BH},
and its claimed limit is $\Theta_1$. The sum here includes all $N$
eigenvalues.

\begin{proposition}[Spectral and geometric input events]\label{sp:inputs}
There are $a,b_0>0$ and $C<\infty$ such that, on a graph event of probability
$1-O(N^{-b_0})$,
\begin{equation}\label{sp:rigidity}
 \max_{2\leq j\leq N}|\lambda_j-\gamma_j|\leq N^{-a},\qquad
 \sup_{\|\bmu\|_2=1,\ A\bmu=\lambda\bmu}\sqrt N\|\bmu\|_\infty\leq(\log N)^C.
\end{equation}
For each fixed $0<c_0<1/2$, on such an event all but at most $N^{1/2}$
roots satisfy, uniformly for $\lambda\in K$,
\begin{equation}\label{sp:star}
 \max_{v,w\in B_1(o)}
 |G_{vw}(\lambda+\ri N^{-c_0})-G_{\mathbb T_d,vw}(\lambda+\ri N^{-c_0})|
 \leq N^{-1/128}.
\end{equation}
The comparison entries use distances zero, one, and two on a tree star.
For every fixed $\alpha>0$, another such event satisfies
\begin{equation}\label{sp:meso}
 \Im m_N(x+\ri\eta)\leq C_K,
 \qquad x\in K,\quad N^{-1+\alpha}\leq\eta\leq1.
\end{equation}
For every fixed radius, the fraction of roots whose neighborhood is not
a tree tends to zero on these events. All spectral-parameter bounds are
uniform on their stated domains.
\end{proposition}
\begin{proof}
We state explicitly the normalization change needed to use
\cite{HY}. If $H=A/\sqrt{d-1}$ and $G_H$ is its resolvent, then
\[
 G_A(z)=\frac1{\sqrt{d-1}}
          G_H\left(\frac z{\sqrt{d-1}}\right).
\]
Since $d$ is fixed and $K$ stays a positive distance from the spectral
edges, this change only modifies constants. The input theorem is
uniform over its spectral domain; evaluating it later at a random
eigenvalue therefore requires no additional union bound.

The estimates in \cite[Theorems 1.2 and 1.4]{HY}, with eigenvalues multiplied by
$\sqrt{d-1}$, give \eqref{sp:rigidity} after decreasing $a$ if necessary.
For \eqref{sp:star}, take geometric parameter $1/2$ and approximation radius
$(\log_{d-1}N)/64$. Except at $N^{1/2}$ roots, the radius
$(\log_{d-1}N)/8$ neighborhood is a tree, so the tree extension of the
approximation neighborhood of a star is $\mathbb T_d$.
The error from \cite[Theorem 4.2 and (4.2)--(4.6)]{HY} is at most
\[
 C(\log N)^C\bigl(N^{-1/64}+N^{-(1-c_0)/2}+N^{-2(1-c_0)/3}\bigr)
 \leq N^{-1/128}.
\]
The stability denominator is bounded below in $K$, and the slack in
the exponents absorbs the fixed factors from the adjacency normalization.
The same theorem and its averaged estimate (4.7) give \eqref{sp:meso}.
The geometric event has polynomial failure probability by
\cite[Proposition 2.1]{HY}, while the local-law failures conditional on it
have arbitrarily large polynomial exponents. In each application we take
the intersection of the finitely many required events. Since the estimates
are uniform, they apply at graph-dependent eigenvalues and to all
orthonormal selections.
\end{proof}

For each graph satisfying~\eqref{sp:rigidity}, Lemma~\ref{sp:noise}
applies with its adjacency matrix held fixed. Its Gaussian vector is
auxiliary noise; the entropy in that lemma is taken over this Gaussian
law with the graph fixed.

\subsection{Counting and a uniform Gaussian channel}\label{ch:section}

This subsection proves a uniform bound on the number of colorings
with a prescribed local profile and uses it to derive an entropy
inequality for correlated Gaussian noise. The counting estimate
produces a single graph event on which the Gaussian inequality holds
for every admissible signal and covariance.

Throughout this subsection, $\cG$ is a uniformly random simple
$d$-regular graph on $\llbracket N\rrbracket$, where $d\geq3$ is fixed and $dN$ is even.
All logarithms are natural. A coloring with $q$ colors means an
arbitrary map from $\llbracket N\rrbracket$ to $\llbracket q\rrbracket$; it need not be a proper coloring.

For a fixed graph and a fixed coloring, choose a uniform vertex and
uniformly order its $d$ neighbors. Let $\bm p$ be the law of the
resulting ordered list of $d+1$ colors, with the root color first.
Let $\bm p_e$ be its marginal on the first two coordinates.
Equivalently, $\bm p_e$ is the law of the colors at the endpoints
of a uniform oriented edge. Here the graph and the coloring are
fixed: the only randomness is the choice of the root and the
ordering of its neighbors.

Write $Z_{\bm p}$ for the number of colorings of the fixed graph
with star law $\bm p$.
The next estimate is a finite-$N$ version of the counting argument
in \cite[Lemma~6.3]{BS}, with explicit dependence on the palette
size.

\begin{lemma}[Uniform counting of empirical color profiles]
\label{ch:count}
For every integer $q\geq1$ and every $t>0$, except on a graph event
of probability at most $C_de^{-t}$,
\begin{equation}\label{ch:count-eq}
 Z_{\bm p}
 \leq
 \exp\left(
 N\left[H(\bm p)-\frac d2H(\bm p_e)\right]+C_dq^{d+1}\log(N+1)+t
 \right)
 \qquad\text{for every }\bm p.
\end{equation}
There are at most $(N+1)^{q^{d+1}}$ possible star profiles.
The constants depend only on $d$, and $q$ may depend on $N$.
\end{lemma}

\begin{proof}
We first work in the configuration model. Each vertex has $d$
labelled half-edges, and we choose a uniform perfect matching of
the $Nd$ half-edges. Until we condition on simplicity, loops and
multiple edges are allowed. The half-edge labels give a fixed
ordering of the $d$ neighbor slots at each vertex. In this proof,
$\PP_{\rm conf}$ and $\E_{\rm conf}$ refer to the uniform matching
of these labelled half-edges, before conditioning on simplicity.
The palette is fixed during the count, but may vary with $N$.

\proofstep{Counting ordered profiles}
For an ordered star word $s=(s_0,\ldots,s_d)\in\llbracket q\rrbracket^{d+1}$, let
$n_s$ be the number of vertices with this word in the fixed
half-edge ordering. Thus
\[
 \sum_s n_s=N,
 \qquad
 \bm p^{\rm ord}(s)=\frac{n_s}{N}.
\]
For $a,b\in\llbracket q\rrbracket$, set
\[
 m_{ab}
 =
 \sum_{s:s_0=a}n_s
       \#\{i\in\llbracket d\rrbracket:s_i=b\}.
\]
This is the number of half-edges whose incident vertex has color
$a$ and whose prescribed neighbor has color $b$. A necessary
consistency condition is
\[
 m_{ab}=m_{ba},
 \qquad
 m_{aa}\text{ is even},
 \qquad
 \sum_{a,b}m_{ab}=Nd.
\]

Let $Z^{\rm ord}_{\bm p^{\rm ord}}$ be the number of colorings
with this ordered profile. For a consistent profile,
\begin{equation}\label{ch:ordered-count}
 \E_{\rm conf}Z^{\rm ord}_{\bm p^{\rm ord}}
 =
 \frac{N!}{\prod_s n_s!}\,
 \frac{\prod_{a<b}m_{ab}!\prod_a(m_{aa}-1)!!}
      {(Nd-1)!!}.
\end{equation}
Indeed, the first factor assigns the prescribed words to the
labelled vertices. Once the words are assigned, they determine
the vertex colors and the color demanded by each half-edge.
For $a<b$, the $m_{ab}$ half-edges demanding an $a$--$b$ edge
must be matched bijectively to the $m_{ba}$ half-edges demanding
a $b$--$a$ edge, giving $m_{ab}!$ choices. The $m_{aa}$ half-edges
demanding an $a$--$a$ edge can be paired in $(m_{aa}-1)!!$ ways.
Finally, $(Nd-1)!!$ is the total number of perfect matchings.
We use the conventions $0!=1$ and $(-1)!!=1$. An inconsistent
profile has expected count zero.

The oriented-edge law associated with this ordered profile is
\[
 \bm p_e(a,b)
 =
 \frac1d\sum_{i=1}^d
 \sum_{\substack{s:s_0=a\\s_i=b}}
       \bm p^{\rm ord}(s)
 =
 \frac{m_{ab}}{Nd}.
\]
Thus, before symmetrization, $\bm p_e$ is the average of the
root--neighbor marginals over all $d$ slots, rather than
necessarily the marginal on the first two coordinates.

Stirling's estimate, uniformly for integers between $0$ and $Nd$,
gives
\[
 \log\frac{N!}{\prod_s n_s!}
 =
 NH(\bm p^{\rm ord})
 +O_d\bigl(q^{d+1}\log(N+1)\bigr).
\]
For the matching factor, symmetry of $(m_{ab})$ gives
\begin{align*}
 &\log
 \frac{\prod_{a<b}m_{ab}!\prod_a(m_{aa}-1)!!}
      {(Nd-1)!!}\\
 &\qquad=
 \frac12\sum_{a,b}m_{ab}\log m_{ab}
 -\frac{Nd}{2}\log(Nd)
 +O_d\bigl(q^2\log(N+1)\bigr)\\
 &\qquad=
 -\frac{Nd}{2}H(\bm p_e)
 +O_d\bigl(q^2\log(N+1)\bigr).
\end{align*}
Here and below, $0\log0=0$. Consequently,
\begin{equation}\label{ch:ordered-count-bound}
 \E_{\rm conf}Z^{\rm ord}_{\bm p^{\rm ord}}
 \leq
 \exp\left(
 N\left[H(\bm p^{\rm ord})-\frac d2H(\bm p_e)\right]
 +C_dq^{d+1}\log(N+1)
 \right).
\end{equation}

\proofstep{Symmetrization and the simultaneous bound}
Averaging the ordered profile over all permutations of the
neighbor slots gives the star law used in the statement:
\[
 \bm p(s_0,\ldots,s_d)
 =
 \frac1{d!}\sum_{\pi}
 \bm p^{\rm ord}(s_0,s_{\pi(1)},\ldots,s_{\pi(d)}),
\]
where the sum is over all permutations of $\llbracket d\rrbracket$.
Concavity of entropy gives
\[
 H(\bm p)\geq H(\bm p^{\rm ord}),
\]
and symmetrization preserves the averaged oriented-edge law
$\bm p_e$.

There are at most $(N+1)^{q^{d+1}}$ ordered count vectors
$(n_s)_s$, since each $n_s$ is an integer between $0$ and $N$.
Every symmetrized profile is determined by such a vector, so the
same bound holds for the number of possible symmetrized profiles.
This bound is over all configurations, not just those arising
from one fixed graph.

For each symmetrized profile $\bm p$, sum
\eqref{ch:ordered-count-bound} over the ordered profiles whose
symmetrization is $\bm p$. After increasing $C_d$, we obtain
\[
 \E_{\rm conf}Z_{\bm p}
 \leq
 \exp\left(
 N\left[H(\bm p)-\frac d2H(\bm p_e)\right]+C_dq^{d+1}\log(N+1)
 \right).
\]
For a fixed $\bm p$, Markov's inequality therefore gives
\begin{align*}
 &\PP_{\rm conf}\left(
 Z_{\bm p}>
 \exp\left(
 N\left[H(\bm p)-\frac d2H(\bm p_e)\right]
 +(C_d+1)q^{d+1}\log(N+1)+t
 \right)
 \right)\\
 &\qquad\leq
 e^{-t}(N+1)^{-q^{d+1}}.
\end{align*}
Taking a union bound over all possible profiles proves
\eqref{ch:count-eq} in the configuration model, with failure
probability at most $e^{-t}$.

\proofstep{Conditioning on simplicity}
The standard configuration-model enumeration
\cite{Bollobas} gives
\[
 \PP_{\rm conf}(\text{simple})
 \longrightarrow
 \exp\left(-\frac{d^2-1}{4}\right)>0.
\]
Conditioning on simplicity therefore increases the failure
probability by at most a factor depending only on $d$.
Moreover, every simple labelled $d$-regular graph has exactly
$(d!)^N$ half-edge representations: at each vertex, assign its
$d$ incident edges bijectively to its $d$ labelled half-edges.
Thus the conditional graph law is uniform, which proves the
lemma.
\end{proof}

We next adapt the entropy argument of
\cite[Theorem~6.1]{BS} to correlated Gaussian noise on a finite
graph. Both the signal and the covariance may depend on the
graph. The proof first bounds the entropy of the entire
discretized coloring by its empirical star and edge entropies.
Concavity handles the star entropy; only the edge entropy
requires a concentration estimate.

\begin{lemma}[A uniform correlated Gaussian channel]\label{ch:channel}
Fix an integer $k\geq1$ and $\kappa>0$, and put
\[
 m=k(d+1),
 \qquad
 \vartheta=\frac1{100(m+1)}.
\]
For sufficiently large $N$, there is a graph event of probability
at least
\[
 1-C_d\exp(-N^{1-\vartheta})
\]
on which the following holds simultaneously for every
$\bm x\in\R^{kN}$ and every covariance matrix $Q$ satisfying
\[
 \|\bm x\|_\infty\leq\frac12N^\vartheta,
 \qquad
 \kappa N^{-\vartheta}I\leq Q\leq N^\vartheta I.
\]

Fix a graph in this event and an admissible pair $\bm x,Q$.
Index the $kN$ scalar coordinates by $(a,v)$, and write
$\bm x_v=(x_{a,v})_{a=1}^k$ and $\bm Y_v=(Y_{a,v})_{a=1}^k$
for the vertex blocks. Introduce an auxiliary Gaussian vector
$\bm Y\sim N(0,Q)$. Independently of $\bm Y$, choose a uniform
vertex $v$ and a uniform ordering $v_1,\ldots,v_d$ of its
neighbors. Let
\begin{align*}
 P_S
 &=
 \Law\bigl(
 \bm x_v+\bm Y_v,
 \bm x_{v_1}+\bm Y_{v_1},
 \ldots,
 \bm x_{v_d}+\bm Y_{v_d}
 \bigr),\\
 P_e
 &=
 \Law\bigl(
 \bm x_v+\bm Y_v,
 \bm x_{v_1}+\bm Y_{v_1}
 \bigr).
\end{align*}
These laws average over both the Gaussian noise and the local
sampling, with the graph and $\bm x,Q$ fixed. Then
\begin{equation}\label{ch:channel-eq}
 h(P_S)-\frac d2h(P_e)
 \geq
 \frac1Nh(\bm Y)
 -C_{d,k,\kappa}N^{-\vartheta}\log N.
\end{equation}
The same graph event works for all admissible $\bm x,Q$,
including graph-dependent choices.
\end{lemma}

\begin{proof}
Set
\[
 \delta=N^{-\vartheta},
 \qquad
 a_0=N^{-4\vartheta},
 \qquad
 K_0=\lfloor N^{5\vartheta}\rfloor,
 \qquad
 q_0=2K_0,
 \qquad
 q=q_0^k.
\]
We use the quantizer from Lemma~\ref{ch:quant}: each scalar
coordinate is divided by $a_0$, rounded down, and then clipped
to $\{-K_0,\ldots,K_0-1\}$. A vertex carries $k$ scalar labels,
so the vertex palette has $q=q_0^k$ colors.

The choice of $\vartheta$ ensures that
\[
 (5m+1)\vartheta<1,
 \qquad
 (10k+8)\vartheta<\frac12.
\]
In particular, the counting error satisfies
\begin{equation}\label{ch:palette-error}
 q^{d+1}\log(N+1)
 \leq C_{d,k}N^{5m\vartheta}\log(N+1)
 =o(N\delta).
\end{equation}

Apply Lemma~\ref{ch:count} with this palette and
$t=N\delta=N^{1-\vartheta}$. Fix a graph in the resulting
simultaneous event, and then fix any admissible $\bm x,Q$.
From this point onward, $\E_{\bm Y}$, $\PP_{\bm Y}$ and
$\Var_{\bm Y}$ refer only to $\bm Y\sim N(0,Q)$, with
$\mathcal G$, $\bm x$ and $Q$ held fixed.
The empirical profiles already include the finite averages
over vertices and neighbor orderings.

\proofstep{1. Entropy of the entire coloring}
Let
\[
 \bm{\mathsf J}
 =
 \left\lfloor\frac{\bm x+\bm Y}{a_0}\right\rfloor
\]
coordinatewise, and let $\bm{\mathsf C}$ be obtained by clipping
each coordinate to $\{-K_0,\ldots,K_0-1\}$. Thus
$\bm{\mathsf J}=J(\bm x+\bm Y)$ and
$\bm{\mathsf C}=T\bm{\mathsf J}$. Write $\bm{\mathsf C}(v)$ for
the $k$-tuple of scalar labels at vertex $v$.
For each realization of $\bm Y$, let
$\widehat{\bm p}_S,\widehat{\bm p}_e$ be the star and
oriented-edge laws of the resulting fixed coloring
$\bm{\mathsf C}$. Explicitly, for a star word $\bm s$ of $d+1$
vertex colors and vertex colors $\bm a,\bm b$,
\[
\begin{aligned}
 \widehat{\bm p}_S(\bm s)
 &=\frac1{Nd!}\sum_{v_0=1}^N
   \sum_{\substack{(v_1,\ldots,v_d)\text{ orders}\\
                     \text{the neighbors of }v_0}}
   \mathbf1_{\{(\bm{\mathsf C}(v_0),\ldots,
                         \bm{\mathsf C}(v_d))=\bm s\}},\\
 \widehat{\bm p}_e(\bm a,\bm b)
 &=\frac1{Nd}\sum_{v=1}^N\sum_{u\sim v}
   \mathbf1_{\{\bm{\mathsf C}(v)=\bm a,
                          \bm{\mathsf C}(u)=\bm b\}}.
\end{aligned}
\]
These formulas already perform the root and neighbor-ordering
averages. The profiles are random only because the coloring depends
on $\bm Y$.

The star profile is a function of the coloring and takes at most
$(N+1)^{q^{d+1}}$ possible values, by Lemma~\ref{ch:count}. Therefore
\[
 H\!\left(\Law_{\bm Y}(\widehat{\bm p}_S)\right)
 \leq q^{d+1}\log(N+1).
\]
This is the entropy of the random profile as an outcome, not the
random number $H(\widehat{\bm p}_S)$ obtained by computing the
entropy of its realized entries.

For each profile $\bm p$ with positive probability, the conditional
coloring is supported on the $Z_{\bm p}$ colorings having that profile.
A distribution on at most $Z_{\bm p}$ points has entropy at most
$\log Z_{\bm p}$, whether or not it is uniform. Hence
\[
\begin{aligned}
 H(\bm{\mathsf C}\mid\widehat{\bm p}_S=\bm p)
 &\leq\log Z_{\bm p},\\
 H(\bm{\mathsf C}\mid\widehat{\bm p}_S)
 &=\sum_{\bm p}\PP_{\bm Y}(\widehat{\bm p}_S=\bm p)
       H(\bm{\mathsf C}\mid\widehat{\bm p}_S=\bm p)\\
 &\leq\E_{\bm Y}\log Z_{\widehat{\bm p}_S}.
\end{aligned}
\]
The sum runs only over profiles of positive probability, for which
$Z_{\bm p}\geq1$. Since the profile is determined by the coloring,
its conditional entropy given the coloring is zero. The chain rule
for the pair consisting of the coloring and its profile gives
\begin{align}
 H(\bm{\mathsf C})
 &=H\!\left(\Law_{\bm Y}(\widehat{\bm p}_S)\right)
       +H(\bm{\mathsf C}\mid\widehat{\bm p}_S)\notag\\
 &\leq q^{d+1}\log(N+1)
       +\E_{\bm Y}\log Z_{\widehat{\bm p}_S}
 \notag\\
 &\leq
 N\E_{\bm Y}\left[
 H(\widehat{\bm p}_S)-\frac d2H(\widehat{\bm p}_e)
 \right]
 +C_dq^{d+1}\log(N+1)+N\delta.
 \label{ch:coloring-entropy}
\end{align}
The last line applies the simultaneous counting bound to every
realized profile and absorbs the profile-entropy bound into $C_d$.
Inside the expectation, for example,
$H(\widehat{\bm p}_S)=-\sum_{\bm s}\widehat{\bm p}_S(\bm s)
\log\widehat{\bm p}_S(\bm s)$. We compute this number for the
realized coloring before averaging over $\bm Y$.

For the unbounded grid coloring, the cellwise Jensen argument
used in Lemma~\ref{ch:quant} gives
\[
 H(\bm{\mathsf J})+kN\log a_0
 \geq h(\bm x+\bm Y)=h(\bm Y).
\]
Indeed, each grid cell has volume $a_0^{kN}$, and the difference
on the left is the relative entropy of the Gaussian density
with respect to its average on each cell.

Clipping acts coordinatewise. Hence conditional subadditivity
and the fact that conditioning reduces discrete entropy give
\begin{align*}
 H(\bm{\mathsf J})-H(\bm{\mathsf C})
 &=H(\bm{\mathsf J}\mid\bm{\mathsf C})\\
 &\leq\sum_{a=1}^k\sum_{v=1}^N
       H(\mathsf J_{a,v}\mid\bm{\mathsf C})\\
 &\leq\sum_{a=1}^k\sum_{v=1}^N
       H(\mathsf J_{a,v}\mid\mathsf C_{a,v})\\
 &\leq CkN^{1+5\vartheta}e^{-cN^\vartheta}.
\end{align*}
The last inequality is the scalar clipping estimate from the
proof of Lemma~\ref{ch:quant}. It applies uniformly to each
coordinate, since its mean has absolute value at most
$N^\vartheta/2$ and its variance is at most $N^\vartheta$.
No independence between coordinates is needed.

Combining these estimates with
\eqref{ch:coloring-entropy} and \eqref{ch:palette-error}, we obtain
\begin{equation}\label{ch:empirical-entropy}
 \E_{\bm Y}\left[
 H(\widehat{\bm p}_S)-\frac d2H(\widehat{\bm p}_e)
 \right]
 \geq
 \frac{h(\bm Y)}N-k\log a_0-C_{d,k}\delta.
\end{equation}
All discrete and differential entropies used here are finite
by the Gaussian density and moment bounds.

\proofstep{2. Passing to the averaged profiles}
Put
\[
 \overline{\bm p}_S=\E_{\bm Y}\widehat{\bm p}_S,
 \qquad
 \overline{\bm p}_e=\E_{\bm Y}\widehat{\bm p}_e.
\]
These are exactly the laws obtained by applying the clipped
quantizer to $P_S$ and $P_e$, respectively.

Concavity gives
$H(\overline{\bm p}_S)\geq\E_{\bm Y} H(\widehat{\bm p}_S)$.
Consequently,
\begin{align}
 H(\overline{\bm p}_S)-\frac d2H(\overline{\bm p}_e)
 \geq{}&
 \E_{\bm Y}\left[
 H(\widehat{\bm p}_S)-\frac d2H(\widehat{\bm p}_e)
 \right]
 \notag\\
 &-\frac d2\left[
 H(\overline{\bm p}_e)-\E_{\bm Y} H(\widehat{\bm p}_e)
 \right].
 \label{ch:averaging-profiles}
\end{align}
Thus we only need to control the entropy gap for the edge
profile. There is no need to prove concentration of the star
profile.

\proofstep{3. Smoothing and concentration of the edge profile}
The hard quantizer is discontinuous, so we first replace its
label indicators by Lipschitz functions. Set
\[
 \ell=\frac{\delta^{3/2}}{10q_0}.
\]
The scalar clipped quantizer has $q_0-1$ jump points, spaced
$a_0$ apart. Since
\[
 \frac{\ell}{a_0}=O(N^{-5\vartheta/2}),
\]
the intervals of radius $\ell$ around these jumps are disjoint
for sufficiently large $N$.

For each scalar label $b\in\{-K_0,\ldots,K_0-1\}$, let
$\varphi_b:\R\to[0,1]$ agree with the hard label indicator away
from these intervals. At the jump point $ba_0$ between labels
$b-1$ and $b$, replace the two indicators on
$[ba_0-\ell,ba_0+\ell]$ by
\[
 \varphi_{b-1}(z)=\frac{ba_0+\ell-z}{2\ell},
 \qquad
 \varphi_b(z)=\frac{z-ba_0+\ell}{2\ell},
\]
and set all other labels equal to zero there. Thus
\[
 \sum_b\varphi_b(z)=1,
 \qquad
 0\leq\varphi_b(z)\leq1,
 \qquad
 \operatorname{Lip}(\varphi_b)\leq\frac1{2\ell}.
\]

An edge word consists of $2k$ scalar labels. For
$w=(w_1,\ldots,w_{2k})$, define
\[
 S_w(\bm Y)
 =
 \frac1{Nd}
 \sum_{\substack{u,v\in\llbracket N\rrbracket\\u\sim v}}
 \prod_{r=1}^k
 \varphi_{w_r}(x_{r,u}+Y_{r,u})
 \varphi_{w_{k+r}}(x_{r,v}+Y_{r,v}),
\]
where the sum is over oriented edges, so each unoriented edge
appears twice. Write $\bm S=(S_w)_w$. The partition-of-unity
property implies that $\bm S$ is a probability vector for
every realization of $\bm Y$.

For a fixed scalar coordinate $Y_{r,v}$, only the $d$ oriented
edges starting at $v$ and the $d$ oriented edges ending at $v$
contribute to its derivative. Since all other factors are at
most one, almost everywhere
\[
 \left|\frac{\partial S_w}{\partial Y_{r,v}}\right|
 \leq
 \frac{2d}{Nd}\frac1{2\ell}
 =
 \frac1{N\ell}.
\]
Summing over the $kN$ scalar coordinates gives
\[
 |\nabla S_w|^2\leq\frac{k}{N\ell^2}.
\]
By the Gaussian Poincar\'e inequality
\eqref{en:gaussian-poincare} and $Q\leq\delta^{-1}I$,
\[
 \Var_{\bm Y}(S_w)
 \leq
 \E_{\bm Y}\bigl[(\nabla S_w)^TQ\nabla S_w\bigr]
 \leq
 \frac{k}{N\delta\ell^2}.
\]
The inequality applies to these Lipschitz functions by smooth
approximation.

There are $q_0^{2k}$ edge words. Cauchy--Schwarz therefore gives
\begin{align}
 \E_{\bm Y}\|\bm S-\E_{\bm Y}\bm S\|_1
 &\leq
 \sum_w\sqrt{\Var_{\bm Y}(S_w)}
 \notag\\
 &\leq
 \frac{C_kq_0^{2k}}{\sqrt{N\delta}\,\ell}
 =
 \frac{C_kq_0^{2k+1}}{\sqrt N\,\delta^2}
 \notag\\
 &\leq
 C_kN^{-1/2+(10k+7)\vartheta}
 \leq C_k\delta.
 \label{ch:soft-concentration}
\end{align}
The last inequality uses $(10k+8)\vartheta<1/2$.
This estimate uses the covariance bound, not independence of
the coordinates of $\bm Y$.

We now compare the soft profile with the hard profile.
Each scalar coordinate of $\bm x+\bm Y$ has variance at least
$\kappa\delta$, and hence its density is bounded by
$(2\pi\kappa\delta)^{-1/2}$. The total length of the transition
intervals is at most $2q_0\ell$. Thus, uniformly in $v,r$,
\[
 \PP_{\bm Y}\bigl(
 x_{r,v}+Y_{r,v}
 \text{ belongs to a transition interval}
 \bigr)
 \leq
 \frac{2q_0\ell}{\sqrt{2\pi\kappa\delta}}
 \leq C_\kappa\delta.
\]

For a fixed edge, the hard word and its soft distribution agree
unless at least one of the $2k$ scalar coordinates belongs to
a transition interval. Their $\ell^1$ distance is always at
most two, so it is bounded by twice the number of coordinates
in transition intervals. Averaging over oriented edges and
then over $\bm Y$ gives
\[
 \E_{\bm Y}\|\widehat{\bm p}_e-\bm S\|_1
 \leq C_{k,\kappa}\delta.
\]
Jensen's inequality also gives
\[
 \|\overline{\bm p}_e-\E_{\bm Y}\bm S\|_1
 \leq
 \E_{\bm Y}\|\widehat{\bm p}_e-\bm S\|_1.
\]
Combining these inequalities with
\eqref{ch:soft-concentration}, we obtain
\begin{align}
 \E_{\bm Y}\|\widehat{\bm p}_e-\overline{\bm p}_e\|_1
 &\leq
 \E_{\bm Y}\|\widehat{\bm p}_e-\bm S\|_1
 +\E_{\bm Y}\|\bm S-\E_{\bm Y}\bm S\|_1
 +\|\E_{\bm Y}\bm S-\overline{\bm p}_e\|_1
 \notag\\
 &\leq C_{k,\kappa}\delta.
 \label{ch:edge-concentration}
\end{align}

\proofstep{4. Controlling the edge entropy gap}
The edge profiles are probability vectors on an alphabet of
size $q_0^{2k}$. Apply the entropy continuity bound
\eqref{en:entropy-continuity}, with total variation distance
$\frac12\|\widehat{\bm p}_e-\overline{\bm p}_e\|_1$.
Concavity of entropy and then concavity of $H_2$ give
\begin{align*}
 0
 &\leq
 H(\overline{\bm p}_e)-\E_{\bm Y} H(\widehat{\bm p}_e)\\
 &\leq
 \E_{\bm Y}\left|
 H(\overline{\bm p}_e)-H(\widehat{\bm p}_e)
 \right|\\
 &\leq
 H_2\left(
 \frac12\E_{\bm Y}\|\widehat{\bm p}_e-\overline{\bm p}_e\|_1
 \right)
 +\frac12\log(q_0^{2k})
  \E_{\bm Y}\|\widehat{\bm p}_e-\overline{\bm p}_e\|_1\\
 &\leq C_{k,\kappa}\delta\log N.
\end{align*}
In the last step we used \eqref{ch:edge-concentration},
$\log q_0=O(\log N)$, and
$H_2(u)\leq u\log(e/u)$ for sufficiently small $u$.

Substituting this estimate and \eqref{ch:empirical-entropy}
into \eqref{ch:averaging-profiles}, we find
\begin{equation}\label{ch:averaged-entropy}
 H(\overline{\bm p}_S)-\frac d2H(\overline{\bm p}_e)
 \geq
 \frac{h(\bm Y)}N-k\log a_0
 -C_{d,k,\kappa}\delta\log N.
\end{equation}

\proofstep{5. Returning to differential entropy}
With the graph fixed, $P_S$ and $P_e$ are finite Gaussian
mixtures, indexed by the sampled stars and oriented edges.
The graph is simple, so the vertices in a sampled star are
distinct. Consequently, each component covariance is a
principal submatrix of $Q$, up to a permutation of coordinates.
It therefore lies between $\kappa\delta I$ and $\delta^{-1}I$.
The component means also satisfy the coordinate bound
$N^\vartheta/2$.

Lemma~\ref{ch:quant} thus applies in dimensions $m=k(d+1)$
and $2k$. Since $\overline{\bm p}_S$ and
$\overline{\bm p}_e$ are the corresponding clipped
quantization laws, it gives
\begin{align*}
 &\left|
 H(\overline{\bm p}_S)+m\log a_0-h(P_S)
 \right|
 +
 \left|
 H(\overline{\bm p}_e)+2k\log a_0-h(P_e)
 \right|\\
 &\qquad\leq
 C_{d,k,\kappa}N^{-7\vartheta}
 +C_{d,k}N^{5\vartheta}e^{-c_{d,k}N^\vartheta}
 \leq C_{d,k,\kappa}\delta.
\end{align*}
The logarithmic terms combine according to
\[
 m-\frac d2(2k)=k(d+1)-dk=k.
\]
Hence
\begin{align*}
 h(P_S)-\frac d2h(P_e)
 &\geq
 H(\overline{\bm p}_S)-\frac d2H(\overline{\bm p}_e)
 +k\log a_0-C_{d,k,\kappa}\delta\\
 &\geq
 \frac{h(\bm Y)}N
 -C_{d,k,\kappa}\delta\log N,
\end{align*}
where the last inequality follows from
\eqref{ch:averaged-entropy}. This is
\eqref{ch:channel-eq}.

Finally, the only exceptional graph event used in the proof
is the simultaneous counting event from
Lemma~\ref{ch:count}. After fixing a graph in that event, every
subsequent estimate holds for each admissible $\bm x,Q$.
In particular, we have not introduced an exceptional event
depending on $\bm x$ or $Q$, nor taken a union bound over these
parameters. This proves the asserted uniformity, including
for graph-dependent choices.
\end{proof}

\subsection{The graph entropy deficit at a fixed noise level}\label{en:graph}

The lower entropy comparison for smooth typical eigenvector processes
is given in \cite[Theorem~10.3]{BS}. For our finite-$N$ argument, the error
must be polynomially small and uniform over exact eigenvectors.
Lemma~\ref{ch:channel}, applied with graph-dependent resolvent noise,
gives the required comparison at the single noise level used in
Proposition~\ref{en:quartic}.

Fix the graph and the selected eigenvectors. In this subsection
$\bm Y=\mathcal L_{\bla}^{-1}\bm X$ is the whitened empirical star,
not the global Gaussian noise from Lemma~\ref{ch:channel}.
The independent Gaussian star is $\bm Z_{\bla}$.
Whitening the star and edge and then dividing by $\sqrt2$ gives
\begin{equation}\label{en:change}
\F_{\bla}(\bm{Z}_{\bla})-\F_{\bla}(\bm{X}+\bm{Z}_{\bla})
=\Delta_{1/2}.
\end{equation}
Here $\Delta_{1/2}$ is the relative-entropy deficit defined in
Subsubsection~\ref{en:moments}. The whitening Jacobians cancel between the
two laws, as in \cite[proof of Theorem~11.2 and Lemma~B.2]{BS}.
The Gaussian quadratic terms cancel because the normalized law has
covariance $I_E$, and $dk=(d/2)(2k)$ makes the star--edge functional
invariant under common scalar dilation.

\begin{proposition}[Polynomial entropy deficit for empirical stars]\label{en:deficit}
For fixed $d,k,K$, there are $b,\gamma>0$ and $C<\infty$ such that,
on a graph event of probability at least $1-CN^{-b}$, every
orthonormal selection of $k$ real eigenvectors with eigenvalues in $K$
satisfies
\begin{equation}\label{en:graph-deficit}
0\le \F_{\bla}(\bm{Z}_{\bla})
-\F_{\bla}\!\left(\frac{\bm{X}+\bm{Z}_{\bla}}{\sqrt2}\right)
\le N^{-\gamma}.
\end{equation}
For each graph in this event, $\bm{X}$ has the empirical law under
$\E_\star$, and $\bm{Z}_{\bla}$ is an independent auxiliary Gaussian.
The entropies average these two sources of randomness while holding
the graph and eigenvectors fixed.
\end{proposition}

\begin{proof}
The proof compares the entropy of a Gaussian perturbation in two
ways. Globally, its entropy is a log determinant of the resolvent
covariance. Locally, the same noise looks like a Gaussian wave in
the tangent space $H_{\bla}$, with a much smaller Gaussian component
normal to that space. The normal entropy is singular as $\eta\to0$,
but it is explicitly computable and cancels the corresponding
global term. We keep this normal component until the cancellation
has been performed.

Let $a_*>0$ be the coarse-rigidity exponent in
Proposition~\ref{sp:inputs}, and let $\vartheta>0$ be the exponent in
Lemma~\ref{ch:channel}. Fix
\[
0<c<\min(a_*/2,\vartheta/4,1/2),\qquad
\eta=N^{-c}.
\]
Work on the intersection of the counting and spectral events for this
$c$; its complement has polynomially small graph probability.
Fix a graph in this event and a selected eigenvector tuple for the
remainder of the proof. All local-law errors below use the explicit
exponent $1/128$ from~\eqref{sp:star}.

For $1\le a\le k$ put
\[
s_a=\pi\rho_d(\lambda_a),\qquad
G_a=(A-\lambda_a-\ri\eta)^{-1},\qquad
Q_a=\frac1{s_a}\Im G_a
=\frac{\eta}{s_a}\bigl((A-\lambda_a)^2+\eta^2\bigr)^{-1}.
\]
Independently of the empirical root and neighbor ordering, let
$\bm{\zeta}^a$ be independent centered Gaussian vectors on $\R^N$
with covariances $Q_a$. Write
$\bm\zeta=(\bm\zeta^1,\ldots,\bm\zeta^k)$ for the full noise vector.
Since $s_a$ is bounded above and below on $K$,
$c_K\eta I\le Q_a\le C_K\eta^{-1}I$. Their direct sum satisfies
Lemma~\ref{ch:channel} for all large $N$.
The signal bound follows from delocalization for
$x_{a,v}=\sqrt N \bmu_a(v)$.
Averaging $\bmx+\bm{\zeta}$ over the full star and edge samples and
the auxiliary Gaussian noise, always with the graph fixed, gives
\begin{equation}\label{en:channel}
h((\bmx+\bm{\zeta})_S)-\frac d2h((\bmx+\bm{\zeta})_e)
\ge\frac1Nh(\bm{\zeta}^1,\ldots,\bm{\zeta}^k)
-CN^{-\vartheta}\log N.
\end{equation}
Lemma~\ref{sp:noise}, with the scalar covariance factor $1/s_a$, yields
\begin{equation}\label{en:global}
\frac1Nh(\bm{\zeta}^1,\ldots,\bm{\zeta}^k)
=\sum_a\left(\frac12\log\frac{2\pi e\,\eta}{s_a}
-L_d(\lambda_a)\right)
+O(N^{-a_*}/\eta+\eta+N^{-1}\log N).
\end{equation}

On a good-root star from~\eqref{sp:star},
$\bmn_a=(-\lambda_a,1,\ldots,1)=(A-\lambda_a)\bme_o$ when embedded in the
graph. The useful identities are exact, not consequences of the local law.
Writing $T=A-\lambda_a$, spectral calculus gives
\[
 T\Im(T-\ri\eta)^{-1}=\eta\Re(T-\ri\eta)^{-1},\qquad
 T\Im(T-\ri\eta)^{-1}T
       =\eta I-\eta^2\Im(T-\ri\eta)^{-1}.
\]
Take the $o,v$ or $o,o$ entry, respectively, and divide by $s_a$.
This gives
\begin{equation}\label{en:normal}
\begin{split}
\bmn_a^T(Q_a)_S\bmn_a
&=\frac{\eta}{s_a}\bigl(1-\eta\Im(G_a)_{oo}\bigr),\\
\bmn_a^T(Q_a)_S\bme_v
&=\frac{\eta}{s_a}\Re(G_a)_{ov},\qquad v\in S.
\end{split}
\end{equation}
On the tree, the boundary imaginary part of the resolvent is
$s_a\Sigma_{\lambda_a}$ on a star. Indeed, the diagonal is $s_a$,
the resolvent equation at the root gives the root--neighbor value
$s_a\lambda_a/d$, and applying it at a neighbor gives the
distance-two value $s_a(\lambda_a^2-d)/(d(d-1))$.
The explicit tree recursion varies by $O_K(\eta)$ on this fixed star
between height zero and height $\eta$. Thus
\[
(Q_a)_S=\Sigma_{\lambda_a}+O_K(N^{-1/128}+\eta).
\]

Let $J$ be the uniform root conditioned to be good, together with its
uniform neighbor ordering; this is conditioning the sampling on the
fixed graph.
Let $\bm{X}^{\rm g}$ and $\bm{\mathcal X}^{\rm g}$ be the signal star and the
noisy star sampled with this conditional root law. Write
$(\bm{\mathcal X}_\parallel,\bm{\mathcal X}_\perp)$ for $\bm{\mathcal X}^{\rm g}$ in
orthonormal tangent coordinates on $H_{\bla}$ and unit normal
coordinates $\bmn_a/\sqrt{d+\lambda_a^2}$. Put
\[
\Sigma_\parallel=\bigoplus_a
\Sigma_{\lambda_a}|_{H_{\lambda_a}},\qquad
Q_*=\diag_a\frac{\eta}{s_a(d+\lambda_a^2)}.
\]
The label $J$ specifies both the good root and its neighbor ordering.
Given $J$, the signal is deterministic and only the auxiliary Gaussian
noise remains random. Without conditioning on $J$, local entropies
average over this finite sampling law and the Gaussian noise. Its
conditional covariance has blocks
\[
\begin{pmatrix}
Q_{\parallel,J}&Q_{\times,J}\\
Q_{\times,J}^T&Q_{\perp,J}
\end{pmatrix},\qquad
\begin{aligned}
Q_{\parallel,J}&=\Sigma_\parallel+O(N^{-1/128}+\eta),\\
Q_{\times,J}&=O(\eta),\\
Q_{\perp,J}&=Q_*^{1/2}(I+O(\eta))Q_*^{1/2}.
\end{aligned}
\]
Equation~\eqref{en:normal} and bounded good-root resolvent entries give
the last two estimates. The eigenrelation gives
$\E[\bm{\mathcal X}_\perp\mid J]=0$, and
$Q_{\parallel,J}^{-1}=O(1)$ gives the conditional normal covariance
\[
Q_{\perp,J}-Q_{\times,J}^TQ_{\parallel,J}^{-1}Q_{\times,J}
=Q_*^{1/2}(I+O(\eta))Q_*^{1/2}.
\]
Equations~\eqref{en:gaussian-entropy}--\eqref{en:gaussian-conditional}
and the covariance averaged over $J$ bound the two outer terms in
\[
h(\bm{\mathcal X}_\perp\mid\bm{\mathcal X}_\parallel,J)
\le h(\bm{\mathcal X}_\perp\mid\bm{\mathcal X}_\parallel)
\le h(\bm{\mathcal X}_\perp).
\]
For the leftmost term, condition on $J$ and on the tangent
coordinate, and use the Gaussian Schur complement. Its entropy is
$\frac12\log\det(2\pi eQ_*)+O(\eta)$ uniformly in $J$.
For the rightmost term, its covariance is
$Q_*^{1/2}(I+O(\eta))Q_*^{1/2}$ because the normal mean vanishes
for every $J$. Gaussian maximum entropy gives the same expression
as an upper bound. The entropy sandwich therefore gives
\[
 h(\bm{\mathcal X}_\perp\mid\bm{\mathcal X}_\parallel)
 =\tfrac12\log\det(2\pi eQ_*)+O(\eta).
\]
In this argument it is important that the signal is exactly tangent.
An uncontrolled normal component of the signal would contribute to
the covariance and destroy this comparison.
Apply Lemma~\ref{sp:mixture} directly to the tangential law; its
constant is independent of the signal means.
The entropy chain rule gives
\begin{equation}\label{en:local-entropies}
\begin{split}
h(\bm{\mathcal X}^{\rm g})
&=h_{H_{\bla}}(\bm{X}^{\rm g}+\bm{Z}_{\bla})
+\frac12\sum_a\log\frac{2\pi e\,\eta}{s_a(d+\lambda_a^2)}
+O(N^{-1/128}+\eta),\\
h((\bm{\mathcal X}^{\rm g})_e)
&=h((\bm{X}^{\rm g}+\bm{Z}_{\bla})_e)+O(N^{-1/128}+\eta).
\end{split}
\end{equation}
For the edge law, use the root and its first uniformly ordered
neighbor in the same good-root sample. Its covariance is uniformly
nondegenerate, so the same comparison applies directly.

We now restore the omitted fraction $p\le N^{-1/2}$ of roots.
The actual component covariances are principal submatrices of $\bigoplus_aQ_a$
because star vertices are distinct. Their eigenvalues lie between
$c_K\eta$ and $C_K\eta^{-1}$; the ideal covariances on $H_{\bla}$ are
uniformly bounded and nondegenerate. The means have at most
polynomial size. Entropy concavity and the Gaussian covariance bound
therefore give $|h|\le C\log N$ for every conditional mixture.
The mixture entropy bounds~\eqref{en:mixture-entropy}, applied to
the good-root and bad-root components, now restore the full sampling law.
Thus \eqref{en:local-entropies} holds with $\bm{\mathcal X}^{\rm g}$ and
$\bm{X}^{\rm g}$ replaced by the full-sample stars $(\bmx+\bm{\zeta})_S$ and $\bm{X}$,
after adding an error of at most
$Cp\log N+H_2(p)=O(N^{-1/2}\log N)$. The star entropies throughout
use their respective ambient spaces: the actual law has full
dimension $k(d+1)$ and the ideal tangential law has dimension $dk$.

The local comparison expresses the left side of~\eqref{en:channel}
as $\F_{\bla}(\bm X+\bm Z_{\bla})$ plus the explicit normal entropy
in~\eqref{en:local-entropies}, up to the displayed errors.
Subtract that normal entropy from~\eqref{en:global}, channel by channel:
\[
 \frac12\log\frac{2\pi e\eta}{s_a}-L_d(\lambda_a)
 -\frac12\log\frac{2\pi e\eta}{s_a(d+\lambda_a^2)}
 =\frac12\log(d+\lambda_a^2)-L_d(\lambda_a).
\]
Thus all singular normal terms cancel exactly, leaving
\[
\begin{aligned}
\F_{\bla}(\bm{X}+\bm{Z}_{\bla})
\ge{}&\sum_a\left(\frac12\log(d+\lambda_a^2)-L_d(\lambda_a)\right)\\
&-C\bigl(N^{-\vartheta}\log N+N^{-a_*+c}
+N^{-1/128}+N^{-c}+N^{-1/2}\log N\bigr).
\end{aligned}
\]
The sum is $\F_{\bla}(\bm{Z}_{\bla})$ by
Lemma~\ref{en:star-entropy}, specifically~\eqref{en:star-functional-value}.
Any $0<\gamma<\min(\vartheta,a_*-c,1/128,c,1/2)$ absorbs constants
and logarithms for large $N$. The Ornstein--Uhlenbeck argument in the
proof of Proposition~\ref{en:quartic} shows that $\Delta_s$ increases
from zero as $s$ increases. Thus~\eqref{en:change} gives the lower
bound. Scale invariance permits division by $\sqrt2$ in
\eqref{en:graph-deficit}. All estimates hold on the same graph event,
uniformly in the selected vectors.
\end{proof}

\subsection{Fourth moments and the qualitative Gaussian limit}\label{en:conclusions}

We return to the physical star coordinates to prove
Theorem~\ref{th:stars} and Corollary~\ref{en:qualitative}.

\begin{proof}[Proof of Theorem~\ref{th:stars}]
The abstract fourth-moment estimate applies to the whitened empirical
law, while the theorem is stated in physical coordinates. We check
both the common graph event and this change of coordinates explicitly.
Intersect the events of Proposition~\ref{en:deficit} for $1\le k\le4$
with the spectral event, and use a common $\gamma>0$. Fix any graph
in this event. Under the empirical sampling law $\E_\star$,
$\bm{Y}=\mathcal L_{\bla}^{-1}\bm{X}$ has the required symmetries,
mean zero and covariance $I$. Delocalization and the uniform bound on the whitening
map give
$|\bm{Y}|\le B=(\log N)^C$ for every selection.

Equations~\eqref{en:change} and~\eqref{en:graph-deficit} give
$0\le\Delta_{1/2}\le N^{-\gamma}$ in the notation of
Proposition~\ref{en:quartic}. For large $N$ this is at most $cB^{-6}$.
The proposition bounds every whitened quartic moment error by
$CB^{4/3}N^{-\gamma/18}$. Composing a physical coordinate monomial
with $\mathcal L_{\bla}$ gives a homogeneous quartic with uniformly
bounded coefficients: each of its four linear factors is a row of
$\mathcal L_{\bla}$, whose coefficients are bounded on $K^k$.
The dimension is bounded because $k\le4$. Dividing by a fixed
coefficient bound reduces to the normalization in
Proposition~\ref{en:quartic}, and multiplying back changes only $C$.
Therefore any fixed $0<\chi<\gamma/18$
gives~\eqref{en:fourth-result} after absorbing the polylogarithmic
factor, on the same graph event for all choices.
\end{proof}

\begin{proof}[Proof of Corollary~\ref{en:qualitative}]
Work first on the event of Theorem~\ref{th:stars}. Its scalar case
applies simultaneously to every eigenvector, so it controls each
coordinate of a tuple of any fixed size $k$. Taking four copies of
that coordinate in~\eqref{en:fourth-result} gives, for all large $N$,
\[
 \E_\star|X_{a,j}|^4\leq4,\qquad
 \E_\star|\bm{X}|^4\leq4\bigl(k(d+1)\bigr)^2.
\]
These are empirical moments under $\E_\star$: their only randomness
is the uniform vertex and neighbor ordering on the fixed graph.
Convergence in probability in~\eqref{en:qualitative-result} instead
refers to the random graph.
The second inequality follows from
$(\sum_{a,j}X_{a,j}^2)^2\leq k(d+1)\sum_{a,j}X_{a,j}^4$.
Thus the squares and quadratic products of all star coordinates are
uniformly integrable, uniformly over the selected vectors. For example,
\[
 \E_\star\bigl[|\bm X|^2\mathbf1_{|\bm X|>R}\bigr]
 \le R^{-2}\E_\star|\bm X|^4\le C_{d,k}R^{-2}.
\]
A product of two coordinates is bounded in absolute value by the
sum of their squares divided by two, so it has the same tail control.
This is the exact step that rules out mass escaping to rare, very
large coordinates; convergence in distribution alone would not
preserve variance.

We now apply the qualitative joint neighborhood conclusion of
\cite{BS}.
It is uniform over fixed orthonormal tuples of almost eigenvectors
at a common energy. In the present setting, every unit linear
combination $\sum_a c_a\bmu_a$ satisfies
\[
 \left\|(A-\lambda_1)\sum_a c_a\bmu_a\right\|_2^2
 =\sum_a c_a^2(\lambda_a-\lambda_1)^2\leq\delta_N^2,
 \qquad \sum_a c_a^2=1.
\]
Hence the shrinking eigenvalue diameter is precisely the
almost-eigenvector condition in that result. Its joint neighborhood conclusion applies in particular to
all channels of the sampled star.

To see both the normalization and uniformity, argue by contradiction.
If~\eqref{en:qualitative-result} fails, there are a subsequence and
selections on graphs in the fourth-moment and Backhausz--Szegedy
events for which the Wasserstein distance stays bounded away from
zero. The uniform formulation of their theorem permits these
graph-dependent choices; a diagonal choice of its accuracy parameters
makes all of them tend to zero along this subsequence.
Compactness of $K$ and the diameter assumption allow a further
subsequence with $\lambda_a\to\lambda$ for every $a$.
The displayed fourth-moment bound gives tightness of the empirical
star laws. Backhausz--Szegedy identifies each subsequential limit as
a centered joint Gaussian wave, possibly with a loss of covariance.
Uniform integrability and~\eqref{def:cov} now give its covariance
exactly:
\[
 \lim_{N\to\infty}\E_\star X_{a,i}X_{b,j}
 =\mathbf1_{a=b}(\Sigma_\lambda)_{ij}.
\]
In particular, each root has variance one, and distinct channels
are independent because the joint limit is Gaussian. The limiting
star law is therefore $\nu_{\bm{Z}_{(\lambda,\ldots,\lambda)}}$.

Weak convergence together with uniform integrability of the squared
norm implies convergence in $W_2$. One way to see this implication is
to couple a weakly convergent subsequence almost surely in the fixed
finite-dimensional Euclidean space. The squared distance is bounded
by twice the sum of the two squared norms, whose tails are uniformly
integrable. Truncation followed by dominated convergence then gives a
coupling with mean squared distance tending to zero. The Gaussian comparison laws
$\nu_{\bm{Z}_{\bla}}$ also converge in $W_2$ by continuity of
$\mathcal L_{\bla}$. The triangle inequality contradicts the
positive discrepancy. The fourth-moment event has polynomially
small complement, while the qualitative Backhausz--Szegedy events
have probability tending to one; no rate for their intersection
is required. For a graph-measurable tuple converging in probability
to a common energy, choose a deterministic $\delta_N\downarrow0$
that bounds its eigenvalue diameter with probability tending to one,
and apply~\eqref{en:qualitative-result}.
\end{proof}

\begin{remark}
The scalar fourth-moment bound alone prevents loss of variance in
the qualitative step. Mixed fourth moments are needed later to
control covariances between quadratic marks at different eigenvectors.
Those estimates retain their polynomial rate for arbitrary bulk
energies. Restricting the coupling in
Corollary~\ref{en:qualitative} to edge coordinates gives conditional
bounded-Lipschitz convergence for every fixed graph-measurable
selection at nearby energies. All uses of this convergence keep the
number of selected vectors fixed while taking $N\to\infty$.
\end{remark}

\section{Point-process convergence and loop equations}\label{sec:fixed-energy}

We prove the \hyperref[th:main-pointprocess]{main point-process theorem}.
Switching identities first give microscopic tightness and a Gaussian
marked limit. The stationary marked flow then yields the full loop
hierarchy, which identifies every subsequential point-process limit.
The moment bounds needed to pass to correlation measures are proved
separately in Appendix~\ref{app:correlations}; they are not inputs to
this section.

\subsection{Switching and microscopic tightness}\label{rs:section}

Fix a bulk energy $E$ and set $z_N(w)=E+w/(N\sigma)$.
The polynomial mixed fourth-moment estimate controls the conditional
variance of the switching compression. Reversibility then yields
scalar tightness and capped moments.

\subsubsection{The exact switching identities}

Conditional on the graph, sample two independent uniform oriented edges
$(l,a)$ and $(b,c)$. If their four endpoints are distinct and neither
$\{l,b\}$ nor $\{a,c\}$ is an edge, replace $\{l,a\},\{b,c\}$ by
$\{l,b\},\{a,c\}$. Otherwise retain the graph. Write $\widetilde A$ for the
resulting adjacency matrix and call a proposal accepted when it meets
these conditions. Set
\[
D=\begin{pmatrix}1&0\\0&-1\end{pmatrix},\qquad
P=\begin{pmatrix}0&1\\1&0\end{pmatrix},\qquad
L=PD=\begin{pmatrix}0&-1\\1&0\end{pmatrix},
\]
Let $\bme_v$ denote the coordinate unit vector at $v$ in $\R^N$.
For every proposal, define
\[
\bmw_+=\frac{\bme_l+\bme_b-\bme_a-\bme_c}{2},\qquad
\bmw_-=\frac{\bme_l+\bme_a-\bme_b-\bme_c}{2},\qquad W=(\bmw_+,\bmw_-).
\]
The notation $\E_*$ denotes conditional expectation over this independent
proposal, before imposing acceptance. Put
\[
K(z)=W^TG(z)W.
\]

\begin{lemma}[Exact rank-two switching identities]\label{rs:switch}
The graph pair $(A,\widetilde A)$ is exchangeable. The rejection probability is
$O_d(N^{-1})$, uniformly in the initial simple regular graph. On accepted
proposals, $W^TW=I_2$, the reverse proposal has $\widetilde W=W P$, and
\begin{align}
\widetilde A-A&=2W DW^T,\label{rs:increment}\\
\widetilde G(z)&=G(z)-G(z)W
 \bigl(K(z)+D/2\bigr)^{-1}W^TG(z),\label{rs:woodbury}\\
\mathcal D(z):=\frac{\det(\widetilde A-z)}{\det(A-z)}
 &=\det(I_2+2DK(z)),\label{rs:determinant}\\
\widetilde m_N(z)-m_N(z)&=-\frac1N\partial_z\log\mathcal D(z),\label{rs:trace-shift}\\
\widetilde K(z)&=P\Bigl[D/2-\tfrac14D
 \bigl(K(z)+D/2\bigr)^{-1}D\Bigr]P.
 \label{rs:compression-shift}
\end{align}
Here tildes indicate the switched graph and its reverse proposal; thus
$\widetilde G(z)=(\widetilde A-z)^{-1}$,
$\widetilde m_N(z)=N^{-1}\Tr\widetilde G(z)$, and
$\widetilde K(z)=\widetilde W^T\widetilde G(z)\widetilde W$.
The logarithmic derivative is independent of the logarithm branch. For every
nonreal $z$,
\begin{equation}\label{rs:rank}
|\widetilde m_N(z)-m_N(z)|\leq\frac{2\pi}{N|\Im z|}.
\end{equation}
\end{lemma}

\begin{proof}
The accepted proposal $(l,a),(b,c)$ has inverse $(l,b),(a,c)$; both have
conditional probability $(Nd)^{-2}$. This correspondence is an involution
on the graph-proposal space, with rejected moves as self-loops, and
therefore proves detailed balance for the uniform graph law. Each endpoint
of a uniform
oriented edge is a uniform vertex, independently of the other sampled
edge. Each forbidden equality has probability $N^{-1}$ and each forbidden
cross adjacency has probability $d/N$. A union bound proves the rejection
estimate.

On the four coordinates ordered as $(l,a,b,c)$, the increment is
\[
 \widetilde A-A=
 \begin{pmatrix}
 0&-1&1&0\\-1&0&0&1\\1&0&0&-1\\0&1&-1&0
 \end{pmatrix},
\]
with all other entries zero. Expansion of
$2\bm w_+\bm w_+^T-2\bm w_-\bm w_-^T$ gives this matrix.
The two displayed columns have norm one and zero inner product
when the endpoints are distinct. The reverse proposal interchanges
these columns. This proves~\eqref{rs:increment} and identifies the
reverse compression basis. The determinant lemma gives
\eqref{rs:determinant}. Its determinant is nonzero for nonreal $z$, since
both adjacency matrices are real symmetric. Woodbury's identity gives
\eqref{rs:woodbury}. Compressing it and using
\[
K-K(K+D/2)^{-1}K
 =D/2-\tfrac14D(K+D/2)^{-1}D,
\]
proves \eqref{rs:compression-shift}. For the sign in the trace identity, recall that
$\partial_z\log\det(A-z)=-\Tr(A-z)^{-1}$.
Differentiating the logarithm of the determinant ratio therefore gives
$-N\widetilde m_N+Nm_N$, proving~\eqref{rs:trace-shift}.
This matrix identity does not require commutativity: substitute
$K=(K+D/2)-D/2$ and expand in the indicated order.
We used $D^2=I_2$.

The increment has rank at most two. The two eigenvalue counting functions
therefore differ by at most two. Integration by parts against
$(t-z)^{-1}$, followed by
$\int_\R|t-z|^{-2}\dd t=\pi/|\Im z|$, proves \eqref{rs:rank}.
\end{proof}

For $t\in\bC$, write
\[
V(t)=\diag(1-t/d,1+t/d).
\]
If $\bmu_\alpha$ is a basis eigenvector orthogonal to the constant
vector, put
$\bmg_\alpha=\sqrt N W^T\bmu_\alpha$. Orthogonality and oriented-edge averaging
give
\begin{equation}\label{rs:mark-covariance}
\E_*\bmg_\alpha=0,\qquad
\E_*\bmg_\alpha \bmg_\beta^T=\mathbf1_{\alpha=\beta}V(\lambda_\alpha).
\end{equation}
The two components are half the sum of the edge differences and half
the difference of the edge sums. Thus \eqref{rs:mark-covariance} follows
from independence of the edges and their exact covariance; edge reversal
makes each difference orthogonal to each sum. Also $W^T\bm{1}=0$.

Consequently, with
\[
\widehat m_N(z)=m_N(z)-\frac1{N(d-z)},
\]
the exact conditional mean is
\begin{equation}\label{rs:mean}
\overline K(z):=\E_*K(z)
=V(z)\widehat m_N(z)-(1-N^{-1})D/d.
\end{equation}
This formula remains valid for a disconnected graph: take the constant
vector as one basis vector and all other eigenvectors orthogonal to it.

Here is also a direct derivation of the mean formula.
The spectral representation and~\eqref{rs:mark-covariance} give
\[
 \E_*K(z)=\frac1N\sum_{i\ne1}
                   \frac{V(\lambda_i)}{\lambda_i-z}.
\]
Since $V(\lambda_i)=V(z)-(\lambda_i-z)D/d$, the sum is exactly
$V(z)\widehat m_N(z)-(1-N^{-1})D/d$.
All these conditional expectations use the proposal \emph{before}
acceptance. Conditioning first on acceptance would slightly change
the sampling law and lose these exact covariance identities. We
instead keep the rejection error separate in each reversibility
argument.

\subsubsection{A conditional variance estimate}

Fix compact bulk intervals $J\Subset K$ with $E$ in the interior
of $J$. Let $\chi>0$ be the mixed fourth-moment exponent supplied by
Theorem~\ref{th:stars} on $K$.
Choose
\[
0<\varepsilon<\min(\chi/4,1/4).
\]
Let $\Omega_N$ be the intersection of its graph event with the spectral
event of Proposition~\ref{sp:inputs} at the scale $N^{-1+\varepsilon}$
and the local tree event in that proposition. After reducing $\beta>0$,
\[
\PP(\Omega_N^c)\leq CN^{-\beta}.
\]
The intervals and $\varepsilon$ remain fixed throughout this subsection.

\begin{lemma}[Conditional variance of the compressed resolvent]\label{rs:variance}
Fix $c_*>0$ and a sufficiently small $\eta_*>0$. Uniformly for
$z=E+\ri\eta$ with $c_*/N\leq\eta\leq\eta_*$, on $\Omega_N$,
\begin{align}
\frac1N\sum_{\lambda_i\in J}|\lambda_i-z|^{-1}
 &\leq CN^\varepsilon,\label{rs:absolute-sum}\\
\E_*\|K(z)-\overline K(z)\|^2
 &\leq C\left(1+\frac{\Im m_N(z)}{N\eta}\right)+N^{-\tau}
 \label{rs:variance-bound}
\end{align}
for some $\tau>0$, with $J$ the inner bulk interval fixed above.
The absolute-sum estimate also holds, with a compact-dependent constant,
at $z=z_N(w)$ uniformly for $w$ in a fixed compact subset of $\bC^+$.
\end{lemma}

\begin{proof}
This estimate uses fourth moments, not convergence of the full mark
law. The diagonal terms in the variance sum are controlled by the
Ward identity. For distinct eigenvectors, the polynomial error from
Theorem~\ref{th:stars} must be small enough to survive an absolute
sum over all bulk eigenvalues. The choice $2\varepsilon<\chi$ is
made for precisely this reason.

The averaged local law in Proposition~\ref{sp:inputs} gives
\[
\frac1N\#\{i:|\lambda_i-x|\leq h\}\leq Ch,
\qquad h\geq N^{-1+\varepsilon},
\]
for the real energies needed here. At $z=E+\ri c_*/N$, the central interval of width $N^{-1+\varepsilon}$
contributes $O(N^\varepsilon)$ and the dyadic annuli contribute
$O(\log N)$. Increasing $\eta$ decreases the weights; denominator
comparison gives the compact-uniform assertion for $z_N(w)$.

The part of $G(z)$ from eigenvalues outside $J$ has uniformly
bounded operator norm. Since $\|W\|$ is uniformly bounded on all
proposals, its centered conditional second moment is bounded by a
constant. The part inside $J$ is
\[
\frac1N\sum_{\lambda_i\in J}\frac{\bmg_i \bmg_i^T}{\lambda_i-z}.
\]
For entries indexed by $a,b,c,e\in\{1,2\}$, combining the mixed fourth-moment estimate
Theorem~\ref{th:stars} with the independent two-edge sampling gives
\[
\bigl|\Cov_*((\bmg_i)_a(\bmg_i)_b,(\bmg_j)_c(\bmg_j)_e)\bigr|
\leq C\mathbf1_{i=j}+CN^{-\chi}.
\]
In the four-factor expansion, the distributions $(4,0)$, $(2,2)$,
and $(3,1)$ between the two edges use, respectively, the fourth-moment
theorem, exact second moments, and vanishing first moments. Gaussian
quadratic marks have zero covariance at distinct eigenvectors and
bounded covariance at the same eigenvector.

Summing entrywise and using the equivalence of matrix norms in
dimension two, we bound the centered variance of the bulk part by
\[
\frac C{N^2}\sum_i|\lambda_i-z|^{-2}
+CN^{-\chi}\left(\frac1N\sum_{\lambda_i\in J}
 |\lambda_i-z|^{-1}\right)^2
\leq\frac{C\Im m_N(z)}{N\eta}+CN^{-\chi+2\varepsilon}.
\]
The first term is the exact Ward sum. Combining the inside and outside
parts proves \eqref{rs:variance-bound}.
More explicitly, the central interval contains at most
$CN\cdot N^{-1+\varepsilon}=CN^\varepsilon$ eigenvalues, each
contributing at most $C$ to the normalized reciprocal-distance sum
when $\eta\ge c_*/N$. On an annulus of radii $h,2h$ with
$h\ge N^{-1+\varepsilon}$, there are at most $CNh$ eigenvalues
and each contributes at most $(Nh)^{-1}$, giving a constant per
annulus. This explains why the absolute sum is only a small power of
$N$, rather than the potentially much larger bound $\eta^{-1}$.
Since $\chi-2\varepsilon>0$, it can be absorbed into the exponent
$\tau$ in the stated result.
\end{proof}

\subsubsection{Reversibility and scalar moments}

\begin{proposition}[Microscopic scalar tightness at a fixed energy]\label{rs:scalar}
With the event $\Omega_N$ above, there are $M_0,C,\kappa>0$ such that
\begin{equation}\label{rs:factorial-tail}
\PP\bigl(\Omega_N\cap\{|m_N(E+\ri\eta)|>M\}\bigr)
\leq C\exp\bigl(-\kappa M\log(2+M)\bigr)+CN^{-\beta},
\qquad M\geq M_0,
\end{equation}
uniformly for $c_*/N\leq\eta\leq\eta_*$. Define
\[
\Gamma_{N,\eta}=\Omega_N\cap
 \{|m_N(E+\ri\eta)|\leq\log N\}.
\]
Then, uniformly in this range of heights,
\begin{equation}\label{rs:height-moments}
\PP(\Gamma_{N,\eta}^c)\leq CN^{-\beta},\qquad
\E\bigl[\mathbf1_{\Gamma_{N,\eta}}|m_N(E+\ri\eta)|^r\bigr]
\leq C_r\quad(0<r<\infty).
\end{equation}
There is also a single graph event $\Gamma_N$ of probability at least
$1-CN^{-\beta}$ such that, for every compact $B\subset\bC\setminus\R$
and fixed $0<r<\infty$,
\begin{equation}\label{rs:compact-moments}
\E\bigl[\mathbf1_{\Gamma_N}\sup_{w\in B}|s_N(w)|^r\bigr]
\leq C_{B,r}.
\end{equation}
In particular $s_N$ is tight for locally uniform analytic convergence
on both half-planes, related by Schwarz reflection.
\end{proposition}

\begin{proof}
The mechanism is a contradiction between a large forward determinant
and a large reverse determinant. A large scalar transform makes each
determinant typically quadratic in that transform, but the two
determinants must be reciprocal. Thus a large scalar transform forces
an atypical compression in at least one direction. Iterating this
observation gives a much stronger tail than a single variance bound.

Write $T_A=|m_N(z)|$. By \eqref{rs:variance-bound},
\[
\E_*\|K-\overline K\|^2\leq C(1+T_A)
\]
throughout the indicated height range. The coefficient of
$\widehat m_N^2$ in $\det(I_2+2D\overline K)$ is
$4(z^2/d^2-1)$; its modulus is uniformly bounded below. All remaining
coefficients are bounded, and $|\widehat m_N-m_N|\leq C/N$. It follows
that, for fixed small $\delta>0$ and sufficiently large $M_0$,
\[
T_A\geq M_0,\quad \|K-\overline K\|\leq\delta T_A
\quad\Longrightarrow\quad
|\mathcal D(z)|\geq\delta T_A^2.
\]
Here the determinant perturbation is at most
$C\delta T_A^2+C\delta T_A$. Chebyshev's inequality bounds the probability
that the compression fails this condition:
\begin{equation}\label{rs:relative-failure}
\PP_*(\|K-\overline K\|>\delta T_A)\leq C/T_A
\quad\text{on }\Omega_N\cap\{T_A\geq M_0\}.
\end{equation}

The rank shift in \eqref{rs:rank} is at most $A_0=2\pi/c_*$. On an
accepted proposal the reverse determinant is $\mathcal D^{-1}$. If both
graphs lie in $\Omega_N$, both compressions satisfy their favorable
events, and $T_A>M$ for large $M$, then
\[
1=|\mathcal D\mathcal D_{\rm rev}|
\geq\delta^2T_A^2T_{\widetilde A}^2>1,
\]
since $T_{\widetilde A}\geq T_A-A_0$. Thus, when the scalar transform is
large and the forward compression is favorable, the reverse proposal
must fail the concentration estimate.

Set $F_N(M)=\PP(\Omega_N\cap\{T_A>M\})$. Rejected proposals contribute
at most $(C/N)F_N(M)$. The initial concentration failure contributes at
most $(C/M)F_N(M)$ by \eqref{rs:relative-failure}. The reversed failure
contributes at most $C(M-A_0)^{-1}F_N(M-A_0)$, using the exact reversal
on accepted proposals. The probability that either needed graph event
fails is at most $CN^{-\beta}$ by stationarity. Therefore, after
absorption and an increase of $M_0$,
\[
F_N(M)\leq\frac CM F_N(M-A_0)+CN^{-\beta}.
\]
To see the gain from iteration, set
$n=\lfloor(M-M_0)/A_0\rfloor$. The homogeneous coefficient after
$n$ steps is bounded by
\[
 \prod_{j=0}^{n-1}\frac{C}{M-jA_0}
 \le \frac{(C/A_0)^n}{n!}
 \le C\exp\{-\kappa M\log(2+M)\},
\]
after changing constants for bounded $M$. Once
$C/M_0\leq1/2$, the inhomogeneous products sum to a bounded geometric
series. This proves \eqref{rs:factorial-tail} uniformly in $\eta$.
Integrating the tail only up to $\log N$ gives
\eqref{rs:height-moments}, because $N^{-\beta}(\log N)^r$ tends to zero
for every fixed $r$.

For \eqref{rs:compact-moments}, use $\Gamma_N=\Gamma_{N,1/(N\sigma)}$,
choosing $c_*\leq1/\sigma$. The spectral representation with unit masses
at $x_i$ gives, on each compact $B\subset\bC^+$,
\[
|s_N(w)-s_N(\ri)|
\leq |w-\ri|\sum_i\frac1{|x_i-w|\,|x_i-\ri|}
\leq C_B\Im s_N(\ri).
\]
Here $|x-\ri|/|x-w|$ is uniformly bounded for real $x$ and $w\in B$.
Thus $\sup_B|s_N|\leq C_B|s_N(\ri)|$, with the same conclusion on the
lower half-plane by reflection. The height moment estimate, together
with the fixed real subtraction, now proves \eqref{rs:compact-moments}
and hence analytic tightness by normal-family compactness.
Cauchy's formula also gives the corresponding capped
moments for each fixed spectral derivative.
The cap in the moment statement is essential at this stage.
The additive error $N^{-\beta}$ in the tail estimate cannot be
integrated to infinity with a uniform bound, since the trivial
finite-$N$ bound on the transform is of order $N$ at microscopic
height. It can be integrated up to $\log N$, because
$N^{-\beta}(\log N)^r\to0$. For tightness of the analytic field,
caps whose probability tends to one are already enough, and this is
all that the point-process proof needs. Only for the correlation
upgrade do we remove the cap on counting observables; this is done
in Lemmas~\ref{co:repair}--\ref{co:ui} of
Appendix~\ref{app:correlations}.
\end{proof}

\subsection{The Gaussian marked limit}\label{rs:limit-section}

We identify finitely many eigenvector marks using
Corollary~\ref{en:qualitative} and control the remaining spectral sum
using Theorem~\ref{th:stars}. The number of retained marks stays fixed
while $N\to\infty$.

Define the fixed matrices and constants
\[
V=V(E),\qquad H=a_EV-D/d,\qquad
q=\frac{d-2}{\sqrt{4(d-1)-E^2}},\qquad c=\sqrt{1+q^2},
\]
and the normalized marked transform
\begin{equation}\label{rs:marked-definition}
S_N(w)=\sigma^{-1}V^{-1/2}\bigl(K(z_N(w))-H\bigr)V^{-1/2}.
\end{equation}
Its imaginary part is positive semidefinite on $\bC^+$.
Here $V,H,q,c$ depend only on the fixed energy $E$ and degree $d$;
$K,S_N$ also depend on the edge proposal. For a basis eigenvector
orthogonal to the constant vector, with eigenvalue in the chosen
bulk interval, set
\[
\boxi_i=V(\lambda_i)^{-1/2}\bmg_i.
\]
These finite-$N$ marks come from the independent edge proposal.
By \eqref{rs:mean},
\begin{equation}\label{rs:marked-mean}
\E_*S_N(w)=s_N(w)I_2+O_B\bigl(N^{-1}(1+|m_N(z_N(w))|)\bigr),
\end{equation}
uniformly on each fixed compact $B\subset\bC^+$.

\subsubsection{The centered spectral tail}

\begin{lemma}[Removal of the remote eigenvector marks]\label{rs:tail}
For every compact $B\subset\bC^+$ and sufficiently large fixed $R$, there
are graph events $\Gamma_{N,R}$ of probability tending to one such that
\begin{equation}\label{rs:tail-bound}
\limsup_{N\to\infty}\E\left[\mathbf1_{\Gamma_{N,R}}
 \sup_{w\in B}\left\|S_N(w)-s_N(w)I_2
 -\sum_{|x_i|\leq R}\frac{\boxi_i\boxi_i^T-I_2}{x_i-w}\right\|^2\right]
\leq\frac{C_B}{R}.
\end{equation}
The constant does not depend on $R$. The events are graph-measurable;
the expectation in~\eqref{rs:tail-bound} is over both the graph and
the independent edge proposal. Equivalently it is the graph
expectation of the conditional proposal average $\E_*$.
\end{lemma}

\begin{proof}
We split the spectrum twice. A smooth macroscopic cutoff separates
energies away from $E$, where the resolvent is a continuous function
of $A$ and local tree geometry is enough. Inside the bulk cutoff,
we retain the microscopic window $[-R,R]$ and use the fourth-moment
estimate for the remaining centered sum. Centering is crucial: the
uncentered tail need not tend to zero, because it carries the
macroscopic background of the Stieltjes transform.

Take $\Gamma_{N,R}$ to be the intersection of $\Omega_N$, $\Gamma_N$,
and $\Gamma_{N,R/(N\sigma)}$. For every fixed $R\geq1$, the last event
is defined for all sufficiently large $N$. Its height-moment constants
in \eqref{rs:height-moments} are independent of $R$. Finite-prefix Gaussianity will be applied separately through
Corollary~\ref{en:qualitative}, whose error tends to zero in probability.

Choose a smooth real function $\vartheta$, equal to one near $E$ and
supported in a compact bulk interval $J$. The outer part of the
compression has scalar weight
\[
f_{N,w}(t)=\frac{1-\vartheta(t)}{t-z_N(w)}.
\]
Uniformly for $t\in[-d,d]$ and $w$ in a fixed compact set, this converges
to the continuous real function $f(t)=(1-\vartheta(t))/(t-E)$, extended
across $E$. Given $\delta>0$, choose a real polynomial $Q_\delta$ with
$\sup_{[-d,d]}|f-Q_\delta|\leq\delta$. Functional calculus bounds the
compression error by $C\delta+o(1)$.

For the fixed degree of $Q_\delta$, the local tree event in
Proposition~\ref{sp:inputs} ensures that all but a vanishing fraction of
roots have the needed tree neighborhood. Two independently sampled edge
neighborhoods of this fixed radius are disjoint, with no path of length
at most $\deg Q_\delta$ between them, except with probability $O(N^{-1})$.
On the remaining proposals, the compression of $Q_\delta(A)$ is the
same deterministic two-edge tree compression: its local entries count
the same walks and its cross-edge entries vanish. Since its norm is
bounded on every proposal, its centered conditional second moment
tends to zero. The outer compression thus has centered variance at most
$C\delta^2+o(1)$. First let $N\to\infty$, then $\delta\to0$.

For the bulk part, the covariance calculation in Lemma~\ref{rs:variance}
gives the following bound for its centered contribution from $|x_i|>R$:
\[
C_B\sum_{|x_i|>R}|x_i-w|^{-2}
+CN^{-\chi}\left(\sum_{\lambda_i\in J}|x_i-w|^{-1}\right)^2.
\]
The second term is $O_B(N^{-\chi+2\varepsilon})$ on $\Omega_N$, by
\eqref{rs:absolute-sum}. If $R\geq2\sup_{w\in B}|w|+1$, then
\begin{equation}\label{rs:tail-ward}
\sum_{|x_i|>R}|x_i-w|^{-2}
\leq C_B\sum_i\frac1{x_i^2+R^2}
=\frac{C_B}{R}\Im s_N(\ri R).
\end{equation}
The expectation of the final imaginary part, restricted to
$\Gamma_{N,R}$, is uniformly bounded by \eqref{rs:height-moments}.
The centered bulk tail therefore satisfies the required $C_B/R$ estimate.

For the retained indices, $\vartheta(\lambda_i)=1$ for large $N$, and
\[
V^{-1/2}(\bmg_i \bmg_i^T-V(\lambda_i))V^{-1/2}
=\boxi_i\boxi_i^T-I_2+O\bigl((R/N)(1+|\boxi_i|^2)\bigr).
\]
Theorem~\ref{th:stars} bounds
$\E_*[(1+|\boxi_i|^2)(1+|\boxi_j|^2)]$ uniformly. The squared retained
error is therefore at most
$C_{R,B}N^{-2}\mu_N([-R,R])^2$ in conditional expectation. Since
\[
\mu_N([-R,R])\leq2R\Im s_N(\ri R),
\]
its expectation on $\Gamma_{N,R}$ tends to zero. The mean error in
\eqref{rs:marked-mean} also tends to zero there by
\eqref{rs:compact-moments}.

We explain the last analytic step, which will recur below.
Cover $B$ by finitely many disks $D_j$ whose concentric disks of
twice the radius remain off the real axis. For an analytic scalar
entry $F$ of the error field, Cauchy's formula on the larger circle
gives
\[
 \sup_{w\in D_j}|F(w)|^2
 \le C_{B,j}\int_{\partial(2D_j)}|F(z)|^2\,|\dd z|.
\]
The same bound applies entrywise to matrices. The preceding
pointwise second-moment estimates are uniform on these circles;
Tonelli's theorem then gives the squared-supremum estimate
\eqref{rs:tail-bound}. The constants may depend on the fixed compact
set but not on $R$ once $R$ exceeds that set.
\end{proof}

\subsubsection{The scalar measure and the finite prefixes}

\begin{lemma}[Scalar Nevanlinna subsequential limits]\label{rs:scalar-limit}
Every sequence has a subsequence on which $(\mu_N,s_N)$ converges in
distribution, vaguely in the first coordinate and locally uniformly in
the second. Each limit has
\begin{equation}\label{rs:nevanlinna}
s(w)=a+\int_\R\left(\frac1{x-w}-\frac{x}{1+x^2}\right)\dd\mu(x),
\qquad a\in\R,\qquad
\int_\R\frac{\dd\mu(x)}{1+x^2}<\infty.
\end{equation}
The measure $\mu$ is a locally finite integer counting measure. For
every compact $B\subset\bC\setminus\R$ and fixed $0<r<\infty$,
\begin{equation}\label{rs:limit-moments}
\E\sup_{w\in B}|s(w)|^r<\infty,
\qquad
\sup_{Y\geq1}\E|s(\ri Y)|^r<\infty.
\end{equation}
\end{lemma}

\begin{proof}
The representing point measure and the analytic function must be
extracted jointly. Vague convergence of point measures records only
particles in fixed bounded intervals; an analytic limit could in
principle retain an additional contribution from particles escaping
to infinity. In the Nevanlinna representation this contribution is
a nonnegative linear term. The uniform bound at large imaginary
height is what rules it out.

Analytic tightness follows from Proposition~\ref{rs:scalar}. The bound
\[
\mu_N([-R,R])\leq(1+R^2)\Im s_N(\ri)
\]
gives tightness of the point measures on each compact interval. A joint
subsequence therefore exists. The Nevanlinna representation and
continuity of its representing measure on finite intervals give
\[
s(w)=a+\alpha w+\int_\R
 \left(\frac1{x-w}-\frac{x}{1+x^2}\right)\dd\mu(x),\qquad\alpha\geq0.
\]
Indeed $\int(1+x^2)^{-1}\dd\mu_N(x)=\Im s_N(\ri)$ is tight.
Weighted-measure convergence identifies the finite-interval measure
with $\mu$; mass escaping to infinity contributes to $\alpha$.
Locally bounded integer masses remain integer under vague convergence.

By bounded truncation and Fatou's lemma, \eqref{rs:compact-moments}
holds in the limit, since the omitted events have vanishing probability.
Applying
\eqref{rs:height-moments} at $\eta=Y/(N\sigma)$ gives the bound
uniform in $Y\geq1$. Since $\alpha Y\leq\Im s(\ri Y)$, letting
$Y\to\infty$ in expectation gives $\alpha=0$.
For the passage from capped to limiting moments, fix a bounded
continuous truncation of the relevant supremum and let $N\to\infty$.
Its contribution on the complement of the cap is at most its bound
times $\PP(\Gamma_N^c)$, which tends to zero. Increasing the truncation
and applying monotone convergence gives the desired limiting moment
bound. Thus no unproved uniform integrability of the uncapped
finite-$N$ transform is being invoked.

The representing measure can also be recovered locally from
$\pi^{-1}\Im s(x+\ri\eta)\,\dd x$ as $\eta\downarrow0$.
This inversion formula identifies it with the vague limit of the
particle measures on intervals whose endpoints carry no atoms.
A locally bounded sequence of integer counting measures has an
integer counting limit: isolate each limiting point in a sufficiently
small interval and use that the number of approximating points is an
integer. Multiplicities may merge at this stage; simplicity will
follow from identification or, in the tagged setting, from a separate
argument.
\end{proof}

\begin{lemma}[Conditional Gaussian representation of the marks]\label{rs:gaussian-limit}
The triples $(\mu_N,s_N,S_N)$ are jointly tight. Each subsequential
limit can be realized with marks $\boxi_x$, independent standard real
two-dimensional Gaussians conditionally on $(\mu,s)$, such that
\begin{equation}\label{rs:gaussian-series}
S(w)=s(w)I_2+\sum_{x\in\mu}\frac{\boxi_x\boxi_x^T-I_2}{x-w}.
\end{equation}
Points with multiplicity receive separate marks. The series converges
conditionally in $L^2$, locally uniformly off the real axis. Its tail
outside $[-R,R]$ has squared unconditional $L^2$ supremum on a fixed
compact set bounded by $C_B/R$, for sufficiently large $R$. Moreover,
\begin{equation}\label{rs:all-marked-moments}
\E\sup_{w\in B}\bigl(|s(w)|+\|S(w)\|\bigr)^r<\infty
\end{equation}
for every compact $B\subset\bC\setminus\R$ and fixed $0<r<\infty$, and
the same holds for each fixed spectral derivative.
\end{lemma}

\begin{proof}
There are two logically separate claims: conditional Gaussianity of
any fixed finite prefix, and control of the infinite remainder.
The first uses only Corollary~\ref{en:qualitative}. The second uses
the polynomial fourth-moment estimates through Lemma~\ref{rs:tail}.
We never apply qualitative convergence to a number of eigenvectors
that grows with $N$.

Fix a subsequential scalar limit and a cutoff $R$ such that
$\mu(\{-R,R\})=0$ almost surely. Such cutoffs can be chosen to tend to
infinity: \eqref{rs:limit-moments} gives finite expected counts on compact
intervals, so only countably many deterministic points can have positive
probability of being atoms.

On the event of at most $k$ eigenvalues in $[-R,R]$, select their
eigenvectors in spectral order and pad with nearby orthonormal
eigenvectors so that all $k$ selected eigenvalues tend to $E$.
For example, if $a>0$ is the coarse-rigidity exponent,
the window $[E-N^{-a/2},E+N^{-a/2}]$ contains at least $k$ bulk
eigenvalues for all large $N$ and accommodates every retained vector. For fixed $k$,
Corollary~\ref{en:qualitative} gives conditional joint edge convergence
in probability, uniformly over these graph-measurable selections.
The two edges are independent; the bounded linear maps
defining $\bmg_i$ and $V(\lambda_i)^{-1/2}$ therefore give independent
standard Gaussian limiting marks, by \eqref{rs:mark-covariance}.

The conditional bounded-Lipschitz error is bounded and tends to
zero in probability, hence also in expectation. This remains true
after multiplication by any bounded graph-measurable function.
For clarity, let $b_N$ be a bounded graph-measurable function and
let $f$ be a bounded Lipschitz test of the retained marks. If
$\epsilon_N$ is the conditional Gaussian approximation error, then
\[
 \left|\E\{b_N[\E_*f(\boxi_1,\ldots,\boxi_k)
                  -\E_{\bm G}f(\bm G_1,\ldots,\bm G_k)]\}\right|
 \le\|b_N\|_\infty\E\epsilon_N\longrightarrow0.
\]
Here $\bm G_1,\ldots,\bm G_k$ are independent standard Gaussian
vectors, and $\E_{\bm G}$ refers only to them. The outer expectation
is over the graph, since the proposal was averaged by $\E_*$.
Choose $b_N$ among bounded tests of the scalar spectral state and
pass to its joint limit. A monotone-class argument then gives the
product Gaussian conditional law given the full state $(\mu,s)$,
including the real constant in its Nevanlinna representation, not
merely given the particle locations.
Counting tightness removes the cap $k$. Changing the order at repeated
limiting points only permutes independent identically distributed marks.
A consistent exhaustion by cutoffs therefore defines the Gaussian
family on all particles.

Finite marked sums are analytically tight. Lemma~\ref{rs:tail}
approximates $S_N-s_NI_2$ by these sums on every compact set, taking
$N\to\infty$, then $k\to\infty$, then $R\to\infty$. This proves
joint tightness. In a subsequential coupling, Fatou's lemma transfers the
truncated
bounds at each fixed cutoff to the limit. The indicators tend to one
almost surely
along a further subsequence simultaneously at the countable cutoffs,
since their failure probabilities vanish.

Conditionally on $(\mu,s)$, the centered summands are independent
quadratic Gaussian chaoses. Their variance sum is bounded by
\[
C\sum_{x\in\mu}|x-w|^{-2}
=C\frac{\Im s(w)}{\Im w}<\infty.
\]
Thus the centered Gaussian series converges conditionally in $L^2$.
Cauchy's formula on surrounding circles proves its locally uniform
version. For $R$ large relative to a compact set, the limiting version of
\eqref{rs:tail-ward}, together with
$\sup_{Y\geq1}\E\Im s(\ri Y)<\infty$, gives the unconditional squared
$L^2$ tail bound $C_B/R$. The finite-prefix identification and the
vanishing tail identify the limit with \eqref{rs:gaussian-series}.

The required higher moments are obtained only after taking the
Gaussian limit. For a centered scalar quadratic Gaussian chaos $Z$,
hypercontractivity gives
$\|Z\|_{L^r}\le(r-1)\|Z\|_{L^2}$ for $r\ge2$.
Apply this conditionally to each entry of a finite centered marked
sum. The bound is independent of the number of summands, so it
passes to the $L^2$ limit. For $0<r<2$, conditional
Cauchy--Schwarz is enough. This bounds the conditional moments in
terms of the variance. Spectral derivatives introduce fixed powers of $|\Im w|^{-1}$.
Nevanlinna
bounds control these variances on surrounding circles by $C_B\Im s(\ri)$.
Thus \eqref{rs:limit-moments} and Cauchy's formula give
\eqref{rs:all-marked-moments}, including derivatives, without higher
eigenvector moments.
One should distinguish the two series in this construction.
The scalar Nevanlinna sum is compensated by $x/(1+x^2)$.
The additional matrix sum is instead centered in the Gaussian marks
and converges in conditional $L^2$. It need not be absolutely
convergent term by term. Its imaginary part, combined with the
scalar part, is
\[
 \Im S(w)=\sum_{x\in\mu}
   \frac{\Im w}{|x-w|^2}\boxi_x\boxi_x^T\succeq0.
\]
This positive series converges by the same weighted second-moment
bound. The distinction between compensation in space and centering
in marks is useful when the switching map is applied.
\end{proof}

\subsubsection{The stationary marked state}

\begin{proposition}[The limiting switching involution]\label{rs:marked}
Every subsequential limit $(\mu,s,S)$ has the representation
\eqref{rs:nevanlinna}, the conditional Gaussian marks
\eqref{rs:gaussian-series}, and all moments in
\eqref{rs:all-marked-moments}. Its law is invariant under
\begin{equation}\label{rs:rotation}
(\mu,s,S)\longmapsto(\mu,s,OSO^T),\qquad O\in O(2).
\end{equation}
Almost surely $S(w)+qD$ is invertible for every nonreal $w$. The map
\begin{equation}\label{rs:involution}
\widetilde S=-qD-c^2L(S+qD)^{-1}L^T,\qquad
\widetilde s=s-\partial_w\log\det(S+qD)
\end{equation}
together with the representing point measure $\widetilde\mu$ of
$\widetilde s$ is a law-preserving involution. The joint law of a state
and its image is invariant under exchange of the two states. Furthermore,
\begin{equation}\label{rs:limit-rank}
|\widetilde s(w)-s(w)|\leq\frac{2\pi}{|\Im w|}.
\end{equation}
For nonreal $w,v$, the conditional Gaussian identities are
\begin{align}
\E[S_{22}(w)-S_{11}(w)\mid\mu,s]&=0,\label{rs:wick-one}\\
\E[S_{11}(w)S_{22}(w)\mid\mu,s]&=s(w)^2,\label{rs:wick-two}\\
\E[S_{12}(w)S_{12}(v)\mid\mu,s]
 &=\frac{s(w)-s(v)}{w-v}.\label{rs:wick-three}
\end{align}
These conditional expectations average only the Gaussian marks,
with the full scalar state fixed. The products are bilinear: no
complex conjugation is intended. Conjugates are obtained by taking
$v$ in the opposite half-plane. At coincident arguments the divided
difference means $s'(w)$.
\end{proposition}

\begin{proof}
The finite-$N$ Woodbury identity is converted into the normalized
map by a deterministic change of coordinates. Its coefficients are
of order one at fixed $d$; no expansion in $1/d$ is used.
There is also a genuine continuity issue: convergence of $S_N$ does
not automatically imply convergence of $(S_N+qD)^{-1}$.
We resolve it by expressing the inverse in terms of the reversed
field, which is tight by reversibility.

The Gaussian representation and moment assertions were proved above.
Replacing each $\boxi_x$ by $O\boxi_x$ preserves its conditional product
law and changes the centered series to $O(S-sI_2)O^T$. This proves
\eqref{rs:rotation}.

We first compute the marked switching exactly at finite $N$. Substitution
of $a_E$ and $\sigma$ gives
\[
H+D/2=\sigma qV^{1/2}DV^{1/2},\qquad
\sigma^{-1}\tfrac12V^{-1/2}PDV^{-1/2}=cL.
\]
Inserting these identities into \eqref{rs:compression-shift} and
\eqref{rs:marked-definition} yields, on accepted proposals,
\begin{equation}\label{rs:finite-involution}
\widetilde S_N=-qD-c^2L(S_N+qD)^{-1}L^T.
\end{equation}
The determinant identity becomes
\[
\mathcal D(z_N(w))=-c^{-2}\det(S_N(w)+qD).
\]
Equations \eqref{rs:trace-shift} and $z_N'(w)=1/(N\sigma)$ therefore
give the exact normalized identities
\begin{equation}\label{rs:finite-scalar}
\widetilde s_N=s_N-\partial_w\log\det(S_N+qD),\qquad
|\widetilde s_N(w)-s_N(w)|\leq\frac{2\pi}{|\Im w|}.
\end{equation}

The accepted graph-proposal measure is invariant under reversal. Its
total mass is $1-O(N^{-1})$. Conditioning this joint measure on acceptance
changes the original law by $O(N^{-1})$ in total variation and preserves
the reversal identity. In particular the original and reversed triples
have identical marginal laws under this conditioning and are both tight.
To control inversion in the limit, extract a joint analytic subsequential
limit of the two triples. Put $M_N=S_N+qD$. Equation
\eqref{rs:finite-involution} writes the whole
inverse field as
\[
M_N^{-1}=-c^{-2}L^T(\widetilde S_N+qD)L.
\]
The right side is tight in the analytic topology. Passing to the
limit in $M_NM_N^{-1}=I_2$ gives directly
\[
 (S+qD)\bigl[-c^{-2}L^T(\widetilde S+qD)L\bigr]=I_2.
\]
Thus the bracketed matrix is the inverse. The identity holds
analytically throughout
both half-planes on one probability-one set.

Cauchy's formula gives convergence of spectral derivatives, while the
last display gives convergence of the inverse. Thus
\eqref{rs:finite-involution} and
\eqref{rs:finite-scalar} give \eqref{rs:involution} and
\eqref{rs:limit-rank}. The reversed scalar function has the representing
point measure given by the vague limit of the reversed eigenvalues, by
Lemma~\ref{rs:scalar-limit}. It is therefore $\widetilde\mu$.
The symmetry under exchange of the finite state and its image passes
to the limiting joint law.

To check the involution directly, write $M=S+qD$. Then
\[
\widetilde M=-c^2LM^{-1}L^T,\qquad
\widetilde M^{-1}=-c^{-2}LML^T,\qquad
\det\widetilde M=\frac{c^4}{\det M}.
\]
Since $L^2=-I_2$, the matrix transformation applied twice returns $M$.
The two scalar logarithmic derivatives cancel by the determinant
identity. Thus the map is an involution on the full state, including its
representing point measure.

Finally the two Gaussian coordinates are independent. Their centered
quadratic series therefore give \eqref{rs:wick-one} and
\eqref{rs:wick-two}, while
\[
\E[S_{12}(w)S_{12}(v)\mid\mu,s]
=\sum_{x\in\mu}\frac1{(x-w)(x-v)}.
\]
The absence of a linear term in \eqref{rs:nevanlinna} identifies this
sum with the divided difference in \eqref{rs:wick-three}. Cauchy--Schwarz
gives absolute convergence of the sums. The conditional
and unconditional moment bounds above also allow differentiation and
passage to coincident arguments in these identities.
For reference, the Gaussian contractions can be read directly from
one particle. If $\boxi=(\xi_1,\xi_2)^T\sim N(0,I_2)$, then
\[
 \E(\xi_1^2-1)(\xi_2^2-1)=0,\qquad
 \E(\xi_1\xi_2)^2=1,\qquad
 \E(\xi_1^2-1)=\E(\xi_2^2-1)=0.
\]
Different particles have independent marks. Thus diagonal centered
series in different coordinates have zero covariance, while the
covariance of the off-diagonal series is the sum of the products of
its scalar kernels. The algebraic identity
\[
 \frac1{x-w}-\frac1{x-v}
 =\frac{w-v}{(x-w)(x-v)}
\]
gives the divided difference. If the scalar representation had
contained a term $\alpha w$, the divided difference would also have
contained $\alpha$, with no corresponding mark contraction. This
explains why removal of the linear term precedes the loop-equation
calculation.
\end{proof}

\subsection{The marked switching flow}
\label{lp:sec}

We derive the bulk loop equations from the marked state of
Proposition~\ref{rs:marked}, retaining its matrices $D,P,L$ and constants
$q>0$, $c^2=1+q^2$. The marked switching map is
\[
\mathcal T(s,S)=
\left(
s-\partial_w\log\det(S+qD),
-qD-c^2L(S+qD)^{-1}L^T
\right).
\]
Here and below a logarithmic derivative means
$\partial_w\log\det M=\Tr(M^{-1}\partial_wM)$, independently of a
logarithm branch.

The scalar representation in Proposition~\ref{rs:marked} gives
\[
\sum_i\frac1{(x_i-w)(x_i-z)}=\frac{s(w)-s(z)}{w-z},\qquad
\sum_i\frac1{|x_i-w|^2}=\frac{\Im s(w)}{\Im w}.
\]
Both sums converge absolutely; the first is interpreted by a derivative
at coincident arguments. In applying the flow below, we will use the
finiteness of all polynomial moments of marked values and spectral
derivatives, established in \eqref{rs:all-marked-moments}.

\begin{lemma}[The stationary conjugated-rotation flow]\label{lp:flow}
Let
\[
O_\theta=e^{\theta L},\qquad
\mathcal R_\theta(s,S)=(s,O_\theta SO_\theta^T),\qquad
\Phi_\theta=\mathcal T\mathcal R_\theta\mathcal T.
\]
For each fixed real $\theta$, $\Phi_\theta$ is defined almost surely
and preserves the joint marked law. Its infinitesimal vector field
is
\begin{equation}\label{lp:vector}
\mathcal VS=\alpha[L,S]+\kappa(P+SPS),
\qquad
\mathcal Vs=2\kappa\,\partial_wS_{12},
\qquad
\alpha=\frac{1-q^2}{1+q^2},\quad
\kappa=\frac{2q}{1+q^2}.
\end{equation}
For every polynomial $F$ in finitely many scalar and marked values
and their spectral derivatives,
\begin{equation}\label{lp:stationary}
\E\mathcal VF=0.
\end{equation}
\end{lemma}

\begin{proof}
The switching is a discrete involution, not a small perturbation.
The continuous parameter comes from rotating the Gaussian marks.
Conjugating that rotation by switching preserves the law and creates
a nontrivial scalar motion. We first compute its derivative
algebraically and then justify differentiating expectations; these
are separate steps.

Put $M=S+qD$ and write $T(S)$ for the matrix component of
$\mathcal T(s,S)$. Since $L^T=-L$,
\[
T(S)+qD=-c^2LM^{-1}L^T,\qquad
\bigl(T(S)+qD\bigr)^{-1}=-c^{-2}LML^T.
\]
Proposition~\ref{rs:marked} gives the involution and its invariance,
while \eqref{rs:rotation} gives invariance under $\mathcal R_\theta$.
For each fixed angle the successive inputs have the original marked
law and lie in the domain of $\mathcal T$ almost surely. Hence
$\Phi_\theta$ preserves that law, and
$\Phi_{\theta+\varphi}=\Phi_\theta\Phi_\varphi$ wherever defined.

The vector field follows by differentiating the conjugated rotation.
For a symmetric matrix increment $\mathcal H$,
\[
dT_Z[\mathcal H]
=c^2L(Z+qD)^{-1}\mathcal H(Z+qD)^{-1}L^T.
\]
At $Z=T(S)$, take $\mathcal H=[L,T(S)]$. The preceding inverse identity gives
\[
dT_{T(S)}[\mathcal H]=-c^{-2}M\mathcal H M.
\]
Indeed $\mathcal H$ is symmetric and traceless, so $L\mathcal H L=\mathcal H$. Also,
\[
\mathcal H=-2qP+c^2(-M^{-1}L+LM^{-1}),
\]
and hence
\[
\mathcal VM=[L,M]+\kappa MPM.
\]
To track the constants, expand
\[
 [L,M]=[L,S]+2qP,\qquad
 MPM=SPS-q[L,S]-q^2P.
\]
Indeed $PD=L$, $DP=-L$, and $DPD=-P$.
The coefficient of $[L,S]$ is $1-\kappa q=\alpha$,
and that of $P$ is $2q-\kappa q^2=\kappa$.
Therefore
\[
[L,M]+\kappa MPM
=\alpha[L,S]+\kappa P+\kappa SPS,
\]
which is the matrix equation in \eqref{lp:vector}.

Let $M_\theta=S_\theta+qD$ be the matrix component along the flow.
Composing the scalar determinant identities gives
\[
s_\theta-s
=\partial_w\log\frac{\det M_\theta}{\det M}.
\]
Since $\Tr(M^{-1}[L,M])=0$,
\[
\mathcal Vs
=\partial_w\Tr(M^{-1}\mathcal VM)
=\kappa\,\partial_w\Tr(PM)
=2\kappa\,\partial_wS_{12}.
\]
The group identity gives \eqref{lp:vector} at every regular angle of
an orbit.

For a polynomial observable $F$, its orbit and derivative are rational
in $\sin\theta,\cos\theta$, with nonzero denominators at zero.
Invariance at each deterministic angle and
\eqref{rs:all-marked-moments} give
\[
\E\int_0^T|\mathcal VF(\Phi_\theta(s,S))|\dd\theta
=T\E|\mathcal VF(s,S)|<\infty.
\]
A nonremovable real pole of the orbit would make this integral
infinite. Its finitely many real singularities on $[0,T]$ are therefore removable;
after filling them, the orbit is absolutely continuous, even when $\Im S$
is only semidefinite. The fundamental theorem of calculus yields
\[
0
=\E\bigl[F(\Phi_T(s,S))-F(s,S)\bigr]
=\int_0^T\E\mathcal VF(\Phi_\theta(s,S))\dd\theta
=T\,\E\mathcal VF.
\]
This proves \eqref{lp:stationary}.
The integration argument is preferable to an unsupported exchange
of a derivative and an expectation at zero. It uses stationarity at
each deterministic time and integrability of the derivative along
the orbit. Tonelli's theorem first gives an almost surely integrable
derivative on every bounded time interval. Any nonremovable pole of
a rational orbit would contradict that integrability. Filling the
removable singularities then permits the fundamental theorem of
calculus. No deterministic lower bound on
$\det(S+qD)$ is required.
\end{proof}

\subsection{Identification of the point-process limit}\label{fe:sec}

We now close the main-text argument. Averaging the stationary flow
conditionally on its scalar data gives the full microscopic loop
hierarchy. The characterization theorem of \cite{BH} then identifies
the limit, and tightness turns subsequential identification into
convergence of the full sequence.

\begin{proposition}[The full microscopic loop hierarchy]\label{lp:criterion}
Let $(\mu,s,S)$ be any subsequential marked limit in
Proposition~\ref{rs:marked}. Then, for every integer $p\geq1$ and every
$w_1,\ldots,w_p\in\bC\setminus\R$,
\begin{equation}\label{lp:hierarchy}
\begin{aligned}
0={}&
\E\left[
\bigl(s'(w_1)+s(w_1)^2+1\bigr)
\prod_{j=2}^p s(w_j)
\right]\\
&+
2\sum_{j=2}^p
\E\left[
\partial_{w_j}\frac{s(w_1)-s(w_j)}{w_1-w_j}
\prod_{\substack{2\leq\ell\leq p\\\ell\neq j}}s(w_\ell)
\right].
\end{aligned}
\end{equation}
At coincident arguments the divided difference is continued
holomorphically. Consequently $\mu$ has the law of $\Theta_1$.
\end{proposition}

\begin{proof}[Proof of Proposition~\ref{lp:criterion}]
For $p=1$, the desired identity resembles a Stein identity: it says
that the expectation of a differential expression in the resolvent
vanishes. That single equation is not the characterization theorem.
We derive all $p$ simultaneously, including mixed upper- and
lower-half-plane arguments, so that the full cited hierarchy applies.

By~\eqref{rs:wick-one}--\eqref{rs:wick-three}, conditionally on
$(\mu,s)$,
\[
 \begin{gathered}
 \E[S_{11}(w)-S_{22}(w)\mid\mu,s]=0,\qquad
 \E[S_{11}(w)S_{22}(w)\mid\mu,s]=s(w)^2,\\
 \E[S_{12}(w)S_{12}(v)\mid\mu,s]
   =\frac{s(w)-s(v)}{w-v}.
 \end{gathered}
\]
At $v=w$, the last expression is $s'(w)$. The moment bounds justify
spectral differentiation of these identities.
Since $[L,S]_{12}=S_{11}-S_{22}$ and
$(SPS)_{12}=S_{11}S_{22}+S_{12}^2$, equation~\eqref{lp:vector} gives
\[
 \mathcal V S_{12}
 =\alpha(S_{11}-S_{22})
   +\kappa(1+S_{11}S_{22}+S_{12}^2),\qquad
 \mathcal V s=2\kappa S_{12}'.
\]
Apply~\eqref{lp:stationary} to
$F=S_{12}(w_1)\prod_{j=2}^p s(w_j)$. The product rule gives
\begin{align*}
 \mathcal VF={}&
 \bigl[\alpha(S_{11}(w_1)-S_{22}(w_1))
    +\kappa(1+S_{11}(w_1)S_{22}(w_1)+S_{12}(w_1)^2)\bigr]
       \prod_{j=2}^p s(w_j)\\
 &+2\kappa\sum_{j=2}^p S_{12}(w_1)S_{12}'(w_j)
       \prod_{\substack{2\le\ell\le p\\\ell\ne j}}s(w_\ell).
\end{align*}
First average the marks conditionally on $(\mu,s)$. The $\alpha$
term vanishes, and the three displayed contractions evaluate the
remaining terms. Then average the scalar state and divide by
$\kappa>0$. This is the tower property
$\E\mathcal VF=\E[\E(\mathcal VF\mid\mu,s)]$, and it gives
\eqref{lp:hierarchy}. The same real marks are used in both
half-planes, so mixed arguments require no separate construction.
The moment bounds allow holomorphic continuation to coincident
arguments.

In the terminology of \cite[Definition~1.3]{BH}, particle-generated
means that the representing measure of the Nevanlinna function is
a locally finite integer atomic measure, with multiplicities
retained. Equation \eqref{lp:hierarchy} is exactly
\cite[Assumption~1.4]{BH} at $\beta=1$: the derivative coefficient is
$(2-\beta)/\beta=1$ and the interaction coefficient is $2/\beta=2$.
All its expectations are finite by the compact moment bounds.
Since $\beta=1$ is rational, \cite[Theorem~1.5]{BH} identifies $\mu$
as $\Theta_1$ in the intensity-$1/\pi$ normalization fixed in the
introduction.
\end{proof}

\begin{proposition}[Fixed-energy point-process convergence]\label{th:pointprocess}
For each fixed bulk energy $E$, the point measure
$\mu_N=\sum_j\delta_{N\sigma(\lambda_j-E)}$, where
$\sigma=\pi\rho_d(E)$, converges in law for the vague topology to
$\Theta_1$, the $\mathrm{Sine}_1$ process with intensity $1/\pi$.
This is precisely the \hyperref[th:main-pointprocess]{main point-process
theorem}.
\end{proposition}

\begin{proof}[Proof of Proposition~\ref{th:pointprocess} and the main theorem]
Proposition~\ref{rs:marked} gives a joint marked limit along a further
subsequence of every subsequence. Proposition~\ref{lp:criterion}
identifies its configuration law as that of $\Theta_1$.
Thus every subsequential limit of $\mu_N$ has the same law.

For completeness, if the full sequence failed to converge to this
law, some bounded continuous test on the space of point measures
would have expectations separated from the limiting value along a
subsequence. Joint tightness would provide a further convergent
subsequence; the preceding identification would force its expectation
to converge to that same value, a contradiction.

Since $\sigma=\pi\rho_d(E)$, this is already the asserted convergence
\[
 \sum_{j=1}^N\delta_{N\pi\rho_d(E)(\lambda_j-E)}
 \ \Longrightarrow\ \Theta_1.
\]
The eigenvalue $d$ leaves every compact microscopic interval, so it
may be included or omitted. This proves the main theorem. Only
convergence in law has been used; no uniform integrability of
uncapped finite-$N$ counting observables is asserted here.
\end{proof}

\paragraph{The stronger conclusions.}
The main-text argument is now complete. Theorem~\ref{th:gaps} is proved
in the appendices. Appendix~\ref{app:correlations} adds uniform
integrability and obtains the fixed-energy correlation measures.
Appendix~\ref{sec:gap-universality} treats a deterministic eigenvalue
label, derives conditional Gibbs laws for its tagged limit, and
identifies the consecutive gaps using
Appendix~\ref{app:gap-identification}. These are additional arguments,
not formal consequences of vague point-process convergence.

\appendix
\section{Fixed-energy correlation universality}\label{app:correlations}

We prove Theorem~\ref{th:gaps}(ii). The main theorem already gives
convergence in law at each fixed energy. The additional issue is
uniform integrability: the $k$-point observable is bounded by a
constant times the $k$th power of a local particle count, and is
not a bounded point-process test. We combine the scalar tail from
Section~\ref{rs:section} with a high-probability counting cap, then
pass to expectations. Throughout this appendix, $E$ is fixed and
all graph probabilities and expectations refer to the original
uniform graph law, unless a different graph is explicitly indicated.

Set
\[
 n_N(I)=\#\{j:N\pi\rho_d(E)(\lambda_j-E)\in I\}
\]
for a fixed bounded interval $I$. Thus $n_N(I)=\mu_N(I)$ in the
same coordinates as the main theorem and Theorem~\ref{th:gaps}.
Proposition~\ref{rs:scalar} gives, for some $b>0$,
\begin{equation}\label{co:tail}
 \PP(n_N(I)>M)\leq C_Ie^{-c_IM\log(2+M)}+C_IN^{-b},\qquad M\geq M_I.
\end{equation}
Indeed the Poisson kernel at a fixed imaginary parameter bounds the
count on $I$ by $C_I\Im s_N$. The graph-event failure contributes
$O(N^{-b})$, with $b$ independent of the moment order.

\begin{lemma}[Repair of finitely many geometric defects]\label{co:repair}
Fix $0<c_0<1/100$, $D_0>0$, and
$R=\lfloor(c_0/4)\log_{d-1}N\rfloor$. Outside an event of probability
$O(N^{-D_0})$, a uniform simple regular graph can be transformed by at most
$K_{D_0}$ valid switchings into a simple regular graph such that every
radius-$R$ ball has excess at most one and fewer than $N^{c_0}$ of these
balls are not trees. The number $K_{D_0}$ is independent of $N$.
\end{lemma}
\begin{proof}
For a finite graph $H$, its excess is
$|E(H)|-|V(H)|+$ the number of connected components. It counts the
edges that must be removed to obtain a spanning forest. In a
connected ball, excess zero means a tree and excess one permits a
single independent cycle. The goal is not to make the entire graph
a tree, but to eliminate the finitely many clusters containing two
or more short independent cycles. We also need a margin in the
number of non-tree balls, since a repair can modify some nearby
neighborhoods.

We first obtain a margin in the number of non-tree balls. If $C_\ell$
counts cycles of lengths at most $\ell=O(\log N)$, then for each fixed
integer $r\geq1$,
\begin{equation}\label{co:cycles}
 \E C_\ell^r\leq (C_r(\ell+1))^{C_r}(d-1)^{r\ell}.
\end{equation}
To see this, expand over ordered $r$-tuples of cycles, allowing
repetitions, and consider their union. It has minimum degree at least two.
Write $h=e-v\geq0$. Suppress degree-two vertices in every component that is not an
isolated cycle. The resulting kernel has minimum degree at least
three. If it has $v_0$ vertices and $e_0$ edges, then
$2e_0\ge3v_0$ and $e_0-v_0\le h$. Thus $v_0\le2h$ and
$e_0\le3h$. Suppression records a length for each kernel edge;
these lengths are at most $r\ell$.
The kernel, subdivisions, isolated cyclic components, and the $r$ cycle
markings have at most $(C_r(\ell+1))^{C_r(h+1)}$ possibilities.
For each union the expected number of embeddings in the configuration
model is at most
\[
 \frac{N^v\prod_x(d)_{\deg x}}
 {(Nd-1)(Nd-3)\cdots(Nd-2e+1)}
 \leq2N^{-h}(d-1)^e.
\]
Here $(d)_j=d(d-1)\cdots(d-j+1)$ is the falling factorial,
$(d)_j\le d(d-1)^{j-1}$ for $j\ge1$, and $e=O(\log N)$; degree greater than
$d$ gives zero embeddings. The sum over $h\leq r\ell$ is bounded by a
geometric series, since a
fixed power of $\log N$ divided by $N$ tends to zero.
Conditioning on simplicity costs a bounded factor. This proves
\eqref{co:cycles}.

Take $\ell=2R+1$ and $r>8(D_0+3)/c_0$. Markov's inequality gives
$C_\ell\leq N^{5c_0/8}$ outside probability $O(N^{-D_0-2})$.
Every non-tree radius-$R$ ball contains a cycle of length at most $2R+1$.
Each such cycle can belong to at most $C_d\ell(d-1)^R$ of these balls.
Thus their number is at most
\begin{equation}\label{co:margin}
 C_d\ell(d-1)^RC_\ell\leq N^{7c_0/8+o(1)}<\tfrac12N^{c_0}.
\end{equation}

Call a vertex bad if its radius-$2R$ ball has excess at least two.
For fixed prescribed roots, expose their radius-$2R$ balls sequentially
in the configuration model, including all half-edges at depth $2R$ and
therefore discovering vertices at depth $2R+1$. Stop if the exposed
vertex sets meet. Each exploration makes at most
$M=C_d(d-1)^{2R+1}=N^{c_0/2+o(1)}$ pairing
queries. Before stopping, each query has conditional probability at most
$C_dM/N$ of a collision within the current exploration. Each pairing to a previously discovered vertex increases the
excess by at most one; a pairing to a new vertex adds one edge and
one vertex and does not increase it. Starting from a root,
excess at least two therefore requires two collision queries.
For one exploration there are at most $M^2$ choices of their
locations and each pair has conditional probability at most
$(C_dM/N)^2$, giving the bound $C_dM^4/N^2$. If the prescribed roots are at
distance greater than $4R+2$, stopping cannot occur. Choosing two query
locations for each root and summing over ordered tuples therefore gives
\[
 \PP(\text{there are }t\text{ separated bad roots})
 \leq N^t(C_dM^4/N^2)^t=N^{-(1-2c_0)t+o(1)}.
\]
Choose the fixed integer $t$ with $(1-2c_0)t>D_0+3$.
A maximal such collection has fewer than $t$ vertices and covers every
bad vertex within distance $4R+2$.

By \cite[Proposition 2.1]{HY}, applied with parameter $8c_0$, every
radius-$8R$ ball has excess at most a fixed $K_0$, outside probability
$O(N^{-D_0-2})$, provided $(1-8c_0)K_0>D_0+3$.
For each chosen center, select a spanning tree of its radius-$8R$ ball.
Delete the union $\mathcal F$ of the omitted edges; its size is at most
$tK_0$. Each originally bad radius-$2R$ ball is contained in one of
these radius-$8R$ balls, since $6R+2\leq8R$ for large $N$.
The remaining graph $\mathcal G\setminus\mathcal F$, of maximum
degree $d$, consequently satisfies
\begin{equation}\label{co:forest}
 \operatorname{excess}B_{2R}(v,\mathcal G\setminus\mathcal F)\leq1\qquad\text{for every }v.
\end{equation}
For vertices that were not bad, this also follows by monotonicity of
excess under taking subgraphs.

Write the deleted edges as $\{l_j,a_j\}$. For each one choose an oriented edge
$(b_j,c_j)$ of the original graph such that $B_{10R}(b_j)$ is a tree
and $\{b_j,c_j\}$ is at distance greater than $20R$ from all deletion
endpoints and previously chosen remote edges. Such choices exist:
the structural estimate with parameter $10c_0$ excludes at most
$N^{10c_0}=o(N)$ roots, with arbitrarily large polynomial probability,
and the distance exclusions cover at most
$C_dtK_0(d-1)^{20R}=O_{D_0}(N^{5c_0})=o(N)$ vertices. Perform the
switches
\[
 \{l_j,a_j\},\{b_j,c_j\}\longmapsto
 \{l_j,b_j\},\{a_j,c_j\}.
\]
The four endpoints of each switch are distinct, and its added edges were absent by the
distance restrictions. Thus the switches are valid and restore regularity.

Deleting $\{b_j,c_j\}$ separates its two branches inside the remote tree
ball. A path of length at most $2R$ entering either branch through a new
edge cannot leave through its distant boundary or through the other
branch. It also cannot encounter another repair site. Hence the new
branches attach trees and give no new short connection between the old
vertices. The cyclic part of a final radius-$R$ ball is contained either
in an old radius-$R$ ball or in the old radius-$2R$ ball of its first
attachment vertex. Equation \eqref{co:forest} bounds its excess by one.
Only $O_{D_0}((d-1)^R)$ roots can have their radius-$R$ ball changed:
each is within distance $R$ of a switching endpoint in the original or
final graph. Together with \eqref{co:margin}, this gives fewer than
$N^{c_0}$ non-tree balls. Choose all centers, spanning trees, and remote
edges in a fixed label order to obtain a deterministic repair map.
The order of the repair choices matters only for defining the map,
not for the estimates. The list of deleted edges is fixed first.
Remote edges are chosen away from \emph{all} deletion endpoints, so
an earlier switching cannot remove a later deletion edge or create a
forbidden edge for it. Removing all deletion and remote edges first,
and then adding the new edges, describes the same final graph.
This viewpoint makes~\eqref{co:forest} directly applicable to the
old part of every final short ball. The number of repair sites is
bounded independently of $N$, which is essential for the counting
cap in the next lemma.
\end{proof}

\begin{lemma}[A high-probability microscopic counting cap]\label{co:cap}
For every $D_0>0$ and bounded interval $I$, there exist $B,C<\infty$ such
that
\begin{equation}\label{co:cap-eq}
 \PP(n_N(I)>C(\log N)^B)\leq C_{D_0,I}N^{-D_0}.
\end{equation}
\end{lemma}
\begin{proof}
The repaired graph is not uniformly distributed in its geometric
class. We therefore cannot simply apply a uniform-graph local law to
its distribution. Instead we bound how many original graphs can
have the same repaired image, and choose the failure exponent of
the local law large enough to absorb this multiplicity.

Let $\Psi$ be the map in Lemma \ref{co:repair}, using at most $K$ switches.
Its image belongs to the geometric class $\overline\Omega(c_0)$ of
\cite[Definition 1.1]{HY}, because $\omega_d\geq1$.
Each simple regular graph has at most $C_dN^2$ switching neighbors.
By reversibility every fiber of $\Psi$ has size at most
$(K+1)(C_dN^2)^K$. If $\mathcal A\subset\overline\Omega(c_0)$, then
\[
 \PP(\Psi(\cG)\in\mathcal A,\ \cG\text{ is repairable})
 \leq(K+1)(C_dN^2)^K\PP(\mathcal G\in\mathcal A).
\]
The local law \cite[Theorem 4.2]{HY} holds with arbitrarily large
polynomial probability under the uniform law on
$\overline\Omega(c_0)$. Take its failure exponent greater than
$2K+D_0+2$. At height $\eta=(\log N)^B/N$, with $B$ sufficiently large,
it gives $\Im m_{\Psi(\cG)}(E+\ri\eta)\leq C$ and hence at most
$C(\log N)^B$ eigenvalues in $E+I/(N\pi\rho_d(E))$.
The fiber bound transfers this event to the repaired graph with failure
$O(N^{-D_0})$. Finally
\[
 \operatorname{rank}(A_{\cG}-A_{\Psi(\cG)})\leq2K,
\]
so rank interlacing at the two interval endpoints bounds the change
in the eigenvalue count by $4K$, proving \eqref{co:cap-eq}.
The fiber estimate follows by reversing at most $K$ switchings:
a fixed image has at most
$\sum_{j=0}^K(C_dN^2)^j\le(K+1)(C_dN^2)^K$ possible preimages.
For a bad subset $\mathcal A$ of the geometric class, its relative
size under all regular graphs is no larger than its relative size
under that class. Thus the local-law failure probability multiplied
by the fiber bound is still $O(N^{-D_0})$ with the chosen exponent.
Finally, an interval count is a difference of two half-line counts.
A rank-$2K$ perturbation changes either half-line count by at most
$2K$, explaining the factor $4K$ rather than $2K$ for the interval.
\end{proof}

\begin{lemma}[Uniform integrability of microscopic counts]\label{co:ui}
For each fixed integer $p\geq1$, the family $n_N(I)^p$ is uniformly
integrable.
\end{lemma}
\begin{proof}
Uniform integrability means that the expected contribution from
large values tends to zero uniformly in $N$, not merely that those
values have small probability. We combine two estimates with
different purposes: the strong scalar tail controls moderate counts,
while the arbitrarily high-probability cap handles the rare graphs
on which that tail estimate has only a polynomial error.

Choose $D_0>p+1$ in Lemma \ref{co:cap}, and let $L_N=C(\log N)^B$ be
its bound. Integrating \eqref{co:tail} up to $L_N$, and using
$n_N(I)\leq N$ beyond it, gives
\[
 \limsup_{N\to\infty}\E[n_N(I)^p\mathbf1_{n_N(I)>M}]
 \leq C_p M^pe^{-c_IM\log(2+M)}
 +C_p\int_M^\infty t^{p-1}e^{-c_It\log(2+t)}\dd t.
\]
The discarded terms are at most $CN^{-b}L_N^p+CN^{p-D_0}$ and tend to
zero. The right-hand side tends to zero as $M\to\infty$, which proves
uniform integrability.
The tail-integration identity behind the estimate is
\[
 \E[X^p\mathbf1_{M<X\le L}]
 \le M^p\PP(X>M)+\int_M^Lpt^{p-1}\PP(X>t)\,\dd t
 \quad(X\ge0).
\]
Here $X=n_N(I)$ and $L=L_N$. The remaining event $X>L_N$
contributes at most $N^p\PP(X>L_N)\le CN^{p-D_0}$.
For the finitely many $N$ below any chosen threshold, the variable
is bounded by $N$, so its tail eventually vanishes. The displayed
limsup estimate therefore proves uniform integrability for the
whole family, not just along a subsequence.
\end{proof}

\begin{proof}[Proof of Theorem~\ref{th:gaps}(ii)]
For a point measure $\zeta$, let $\mathcal C_\varphi(\zeta)$ be the sum
of $\varphi$ over its ordered distinct particles. Take a bounded interval
$I$ containing all coordinate projections of $\supp\varphi$ in its
interior. Then
\[
 |\mathcal C_\varphi(\zeta)|\leq\|\varphi\|_\infty\zeta(I)^k.
\]
At a simple locally finite configuration with no particles at the
boundary of $I$, this functional is continuous for vague convergence:
match the finitely many particles in a compact neighborhood of
$\supp\varphi$ and pass their coordinates to the limit.
The \hyperref[th:main-pointprocess]{main point-process theorem} gives
convergence of the graph point measure to $\Theta_1$ in these same
coordinates. This process is simple. Its intensity measure is
$\pi^{-1}\dd x$, so it almost surely has no particle at either fixed
endpoint of $I$. Lemma \ref{co:ui} gives convergence of the expectation
of the first sum in~\eqref{th:corrformula} to $\E\mathcal C_\varphi(\Theta_1)$.
The GOE expectation has the same limit by the classical bulk correlation
asymptotics, see~\cite[(2.3)]{BEYY}, after multiplying the usual
density-normalized coordinates by $\pi$ on both sides. Equivalently,
one applies that result to the test
$\bmv\mapsto\varphi(\pi\bmv)$. This proves~\eqref{th:corrformula}.
To justify the expectation passage explicitly, compose
$\mathcal C_\varphi$ with a continuous truncation of the real line.
The resulting observable is bounded and continuous at almost every
configuration under the limiting law. The almost-sure-continuity
version of the continuous mapping theorem therefore gives convergence
of its expectation. Then remove the truncation using
$|\mathcal C_\varphi|\le\|\varphi\|_\infty n_N(I)^k$ and
Lemma~\ref{co:ui}. Repeated labels are excluded in the factorial
sum, even if two finite-$N$ eigenvalues happen to coincide; at the
simple limiting configuration the matched-particle description is
unambiguous.
\end{proof}

\section{Gap universality at a deterministic label}\label{sec:gap-universality}

We prove Theorem~\ref{th:gaps}(i). A deterministic eigenvalue label
cannot be recovered by simply centering the fixed-energy limit at
one of its particles. We therefore adapt the marked construction of
Section~\ref{sec:fixed-energy} while retaining the chosen label.
The rooted loop equation gives rigidity; coordinate integration by
parts gives conditional Gibbs laws. The comparison in
Appendix~\ref{app:gap-identification} then identifies the gap law.
This argument uses the eigenvector and switching inputs from the main
text, but does not require the counting upgrade in
Appendix~\ref{app:correlations}.

\subsection{Tagged marked limits}\label{gp:section}

For a deterministic bulk label $i=i_N$, pass to a subsequence with
$\gamma_i\to E$ and, in the descending eigenvalue order, put
\[
 \sigma=\pi\rho_d(E),\qquad
 x_{N,j}=N\sigma(\lambda_{i-j}-\lambda_i),\qquad j\in\{i-N,\dots,i-1\}.
\]
The relative labels increase, $x_{N,0}=0$, and every fixed offset stays
in the bulk. Throughout this appendix the point measure is centered
at the tag:
\[
 \mu_N=\sum_{j=i-N}^{i-1}\delta_{x_{N,j}}.
\]
Thus $\mu_N$ now denotes the tagged measure, not the fixed-energy
measure from Section~\ref{sec:fixed-energy}. Its limit $\mu$ is
likewise rooted at label zero; it is not identified with the
unrooted law of $\Theta_1$. The proof retains the limiting labelled
array and identifies its gaps separately. All estimates are uniform
in the original bulk label.

The change of sign in the index is deliberate: although the original
spectrum is decreasing, the relative coordinates $(x_{N,j})$ are
increasing. For example,
$x_{N,-1}=N\sigma(\lambda_{i+1}-\lambda_i)\le0$, while
$x_{N,1}=N\sigma(\lambda_{i-1}-\lambda_i)\ge0$.
A switching can move the chosen eigenvalue by more than a prescribed
microscopic amount, so fixed-energy tightness cannot simply be
reused after substituting $E=\lambda_i$. The next lemma supplies
the needed random-center estimate.

Retaining $V,H$ from \eqref{rs:marked-definition} and writing
$z_i(w)=\lambda_i+w/(N\sigma)$, define
\[
 s_{N,i}(w)=\sigma^{-1}(m_N(z_i(w))-a_E),\qquad
 S_{N,i}(w)=\sigma^{-1}V^{-1/2}(K(z_i(w))-H)V^{-1/2}.
\]

\begin{lemma}[Scalar tightness around a deterministic label]\label{gp:tag}
Fix $c_*>0$ and a sufficiently small $\eta_*>0$. There are
$C,c,\beta>0$ and a graph event $\Omega_N$ with
$\PP(\Omega_N^c)\leq CN^{-\beta}$ such that, for
$T_i=|m_N(\lambda_i+\ri \eta)|$,
\begin{equation}\label{gp:tag-tail}
 \PP(\Omega_N\cap\{T_i>M\})
 \leq C e^{-c(\log M)^2}+CN^{-\beta},
 \qquad c_*/N\leq\eta\leq\eta_*,\quad M\geq C.
\end{equation}
Consequently the events
$\Gamma_{N,i,\eta}=\Omega_N\cap\{T_i\leq\log N\}$ have probability
tending to one and
$\sup_{N,i,\eta}\E[\mathbf1_{\Gamma_{N,i,\eta}}T_i^r]<\infty$
for every fixed $0<r<\infty$.
\end{lemma}

\begin{proof}
The proof differs from the fixed-energy argument in one essential
way: after switching, the old center need not be the new eigenvalue
with the same label. We first find a nearby reversed eigenvalue with
one of five possible labels, then control the reversed compression
uniformly over all possible small shifts of that center. Only after
this uniform estimate is available do we select the nearby label.

Choose three nested compact bulk intervals around $E$ and include coarse
rigidity in the event of Lemma~\ref{rs:variance}. The uniform counting
bound and separation from the outside spectrum make
\eqref{rs:variance-bound} uniform in the middle interval, including
all centers below and all labels within $O(\log N)$ of $i$.

Write $z=\lambda_i+\ri\eta$ and $T=T_i$. Since $D^2=I_2$,
\eqref{rs:mean} gives
$K+D/2=m_N(z)V(z)+O(1)+(K-\overline K)$.
The least singular value of $V(z)$ is bounded below. For fixed small
$\delta>0$ and large $T$,
\[
 \|K-\overline K\|\leq\delta T
 \quad\Longrightarrow\quad \|(K+D/2)^{-1}\|\leq C/T.
\]
The conditional failure probability is $C/T$. Let
$\bmg_i=\sqrt N W^T\bmu_i$. By \eqref{rs:mark-covariance},
$\E_*|\bmg_i|^2=2$, so $|\bmg_i|^2\leq c_1N\eta T$ also fails with
probability at most $C/T$. On an accepted proposal satisfying both
bounds, \eqref{rs:woodbury} gives
\[
 \langle \bmu_i,\widetilde G(z)\bmu_i\rangle
 =\frac{\ri}{\eta}+\frac{\bmg_i^T(K+D/2)^{-1}\bmg_i}{N\eta^2},
 \qquad
 \Im\langle \bmu_i,\widetilde G(z)\bmu_i\rangle\geq\frac{3}{4\eta},
\]
for small $c_1$. Let $\nu$ be the spectral measure of $\widetilde A$ at the unit
vector $\bmu_i$, and write
$p=\nu([\lambda_i-\eta,\lambda_i+\eta])$.
The Poisson kernel is at most $1/\eta$ everywhere and at most
$1/(2\eta)$ outside this interval. Consequently
\[
 \frac{3}{4\eta}
 \le\int\frac{\eta}{(t-\lambda_i)^2+\eta^2}\,\dd\nu(t)
 \le\frac p\eta+\frac{1-p}{2\eta},
\]
so $p\ge1/2$. In particular, the interval contains a reversed
eigenvalue. Interlacing supplies a nearby eigenvalue with
$|j-i|\leq2$: either extreme label can replace a nearby eigenvalue
beyond this range.

For a scalar Herglotz transform, parameters of equal height whose real
parts differ by at most that height satisfy
\[
 \Im m(z)\asymp\Im m(z_0),\qquad
 |m(z)-m(z_0)|\leq C\Im m(z_0),\qquad |m(z)|\asymp|m(z_0)|.
\]
These bounds follow by comparing denominators and remain valid on fixed
scaled compact subsets of $\bC^+$. Together with \eqref{rs:rank} they give
$\widetilde T_j:=|\widetilde m_N(\widetilde\lambda_j+\ri \eta)|\asymp T$. By
\eqref{rs:compression-shift}, $\widetilde K(z)$ is bounded, whereas its
conditional mean has size comparable with $\widetilde T_j$.

To estimate this deviation at the old, proposal-dependent center, set
\[
 W_j(w)=\widetilde K(\widetilde\lambda_j+\eta w)
             -\overline{\widetilde K}(\widetilde\lambda_j+\eta w).
\]
Cauchy's formula on the circle $|w-2\ri|=7/4$ and
\eqref{rs:variance-bound} give
\begin{equation}\label{gp:adaptive}
 \E_*\sup_{|u|\leq1}\|W_j(u+\ri)\|^2\leq C(1+\widetilde T_j).
\end{equation}
The circle has heights in $[1/4,15/4]$ and surrounds $[-1,1]+\ri$,
so we use the variance lemma with its height range adjusted accordingly.
The centers used here correspond to fixed labels of the reversed graph.
In this display $\E_*$ averages proposals with the reversed graph
held fixed, and $\overline{\widetilde K}$ is that conditional
proposal mean. The old center is selected only after taking the
supremum in~\eqref{gp:adaptive}.

Put $F_i(M)=\PP(\Omega_N\cap\{T_i>M\})$. Exact proposal reversal,
\eqref{gp:adaptive}, and conditional Markov give, after absorbing the
initial failures and rejection,
\[
 F_i(M)\leq\frac CM\sum_{|j-i|\leq2}F_j(c_0M)+CN^{-\beta},
 \qquad 0<c_0<1.
\]
At depth $n$, every label lies within $2n$ of the starting label.
The sum over the at most five choices at each step can be absorbed
by replacing $C$ with $5C$. If
$n\asymp\log M/|\log c_0|$ is the first time that $c_0^nM$
reaches a sufficiently large constant, the homogeneous product is
bounded by
\[
 \prod_{j=0}^{n-1}\frac{5C}{c_0^jM}.
\]
Its logarithm is
$n\log(5C)-n\log M-\frac{n(n-1)}2\log c_0
=-c(\log M)^2+O(\log M)$, with a positive $c$ after reducing
constants. Choosing the terminal threshold large makes all
successive coefficients at most $1/2$, so the accumulated exceptional
terms are bounded by a geometric sum. Since $T_i\leq N/c_*$, all labels remain
within $O(\log N)$ of $i$. This proves \eqref{gp:tag-tail};
integration of the tail bound up to $\log N$ gives the capped moments.
This tail gives all capped moments because
$\int_C^\infty M^{r-1}e^{-c(\log M)^2}\,\dd M<\infty$
for each fixed $r$. At $M=\log N$, however, the leading failure
bound is $e^{-c(\log\log N)^2}$. It tends to zero but is not asserted
to be a negative power of $N$. In the tagged subsequence arguments
below, only a vanishing cap failure probability is needed.
\end{proof}

\begin{proposition}[Construction and normalization of a tagged limit]\label{gp:state}
Every sequence has a joint subsequential limit of all fixed relative
labels and the scalar and marked transforms centered at $\lambda_i$.
It has $x_0=0$, infinitely many particles in both directions, and the
representations \eqref{rs:nevanlinna} and \eqref{rs:gaussian-series}.
The marks are independent standard real two-dimensional Gaussians
conditionally on the full scalar state. The bounds
\eqref{rs:limit-moments} and \eqref{rs:all-marked-moments} hold at
every fixed label, with uniform constants, and
\begin{equation}\label{gp:infty}
 s(\ri Y)\longrightarrow\ri\quad\hbox{in probability},\qquad
 \E\|S(\ri Y)-s(\ri Y)I_2\|^2\leq C/Y.
\end{equation}
Its law is invariant under Gaussian rotations and under
$\mathcal T$ followed by recentering at the same ordered label.
\end{proposition}

\begin{proof}
We separate construction of a rooted point measure from tightness
of its labels. Initially a fixed neighboring label is allowed to
escape to infinity. We first prove the normalization at imaginary
infinity and show that infinitely many particles remain on both
sides. Only then can fixed neighboring labels be recovered without
circular reasoning.

\proofstep{1. Construct the scalar field and Gaussian marks}
First apply Lemmas~\ref{rs:tail}--\ref{rs:gaussian-limit} at
$\lambda_i$, using Lemma~\ref{gp:tag} for scalar moments. The uniform
variance estimate above supplies the centered tail bound, and
Corollary~\ref{en:qualitative} applies to graph-measurable finite
prefixes. All retained eigenvalues converge to $E$ by coarse rigidity.
For padding, use the window
$[\lambda_i-N^{-a/2},\lambda_i+N^{-a/2}]$, where $a$ is the
coarse-rigidity exponent; it contains enough eigenvalues for every
fixed prefix. This allows $\gamma_i$ to approach $E$ arbitrarily slowly.
The number of selected vectors is fixed while $N$ tends to infinity.
At the fixed constants $E,\sigma,V,H$, the mean error is
\[
 O\bigl((|\lambda_i-E|+N^{-1})(1+|m_N|)\bigr)=o_{\PP}(1)
\]
on compact sets by coarse rigidity and capped moments.
The squared conditional whitening error is bounded by
$C_{R,B}(|\lambda_i-E|+R/N)^2\mu_N([-R,R])^2$ and vanishes on
the two height caps. The outside polynomial approximation still
converges to $(1-\vartheta(t))/(t-E)$.
We obtain a rooted marked limit, initially allowing escaping neighbors.
Uniform height moments remove its scalar linear coefficient and give
the second bound in \eqref{gp:infty}.

\proofstep{2. Define switching without changing the origin}
The old-origin reverse $(s^{\rm o},S^{\rm o})=\mathcal T(s,S)$
exists analytically. Two particles make $\Im S$ positive definite.
With only the root, $S(w)=aI_2-\boxi\boxi^T/w$; a kernel of $S+qD$
would lie in $\boxi^\perp\cap\ker(aI_2+qD)$, an event of probability
zero since $q>0$. Passing the inverse identities to the limit gives
\begin{equation}\label{gp:old-reverse}
 |s^{\rm o}(w)-s(w)|\leq\frac{2\pi}{|\Im w|},\qquad
 (S^{\rm o}+qD)L(S+qD)L^T=-c^2I_2.
\end{equation}

\proofstep{3. Rule out a reverse matrix linear term}
If $u\le e\le v$ are consecutive eigenvalues, every spectral point
$t$ lies outside $(u,v)$. For $t\le u$, its distance from $e$ is
at least its distance from $u$; for $t\ge v$, use $v$ instead.
Comparing each spectral projection's nonnegative Poisson weight gives
\[
 \Im G(e+\ri \eta)\preceq\Im G(u+\ri \eta)+\Im G(v+\ri \eta).
\]
Let $S_N^{\rm o}$ denote the reverse field at the old origin.
Fix $Y\geq1$, put $\eta=Y/(N\sigma)$, and intersect acceptance
with the five reversed caps $\widetilde\Gamma_{N,i+k,\eta}$ to obtain
$\Theta_{N,Y}$. Interlacing, reversal and
$\Tr\overline K=2\widehat m_N$ give
\[
 \E\bigl[\mathbf1_{\Theta_{N,Y}}\Tr\Im S_N^{\rm o}(\ri Y)\bigr]
 \leq C\sum_{|k|\leq2}
 \E\bigl[\mathbf1_{\widetilde\Gamma_{N,i+k,\eta}}
       \Im\widehat{\widetilde m}_N(\widetilde\lambda_{i+k}+\ri \eta)\bigr]\leq C.
\]
Since $\PP(\Theta_{N,Y})\to1$, Fatou yields
\begin{equation}\label{gp:reverse-height}
 \sup_{Y\geq1}\E\Tr\Im S^{\rm o}(\ri Y)<\infty.
\end{equation}
Thus $S^{\rm o}$ has no matrix linear coefficient.

\proofstep{4. Show that the reverse process is nonempty}
An empty reverse scalar measure would make the old measure finite
by the rank interval bound. Absorb its compensators into $a$ and,
writing $n\geq1$ for its multiplicity, set
\[
 Q(w)=\prod_{j=1}^n(x_j-w)\det\left(aI_2+qD+
                  \sum_{j=1}^n\frac{\boxi_j\boxi_j^T}{x_j-w}\right).
\]
The polynomial $Q$ satisfies $s^{\rm o}=a-Q'/Q$. Its degree is
$n$ unless $a=\pm q$, when it is $n-1$ almost surely: the leading
coefficient is a nonzero multiple of a sum of Gaussian squares in
the null direction of $aI_2+qD$. Its zeros are real by the Herglotz
property. Emptiness therefore forces $n=1$, $a=\pm q$. Direct
inversion gives the positive linear coefficient
$c^2\xi_2^{-2}\bme_1\bme_1^T$ or $c^2\xi_1^{-2}\bme_2\bme_2^T$, contradicting
\eqref{gp:reverse-height}.

\proofstep{5. Identify the value at imaginary infinity}
Extract jointly the five reverse anchored states and their shifts
$\Delta_k=N\sigma(\widetilde\lambda_{i+k}-\lambda_i)$ in the extended real
line. A finite reverse pole and interlacing imply that some
$\Delta_k$ is finite. At each fixed anchor the centered Gaussian
field $Z_k=S_k-s_kI_2$ satisfies
$\E\sup_{|u|\leq R}\|Z_k(u+\ri Y)\|^2\leq C_R/Y$ for large $Y$,
by Cauchy's formula and its Gaussian series. Hence
\[
 \PP(\|S^{\rm o}(\ri Y)-s^{\rm o}(\ri Y)I_2\|>\varepsilon)
 \leq \PP(\min_k|\Delta_k|>R)+\frac{C_R}{\varepsilon^2Y}.
\]
Let $Y\to\infty$ and then $R\to\infty$. Here the supremum over possible
shifts is taken before the anchor is selected. Write $S=sI_2+Z$ and $S^{\rm o}=s^{\rm o}I_2+Z^{\rm o}$.
Both centered fields tend to zero in probability at $\ri Y$, and
both scalar fields are tight uniformly in $Y$ by the height moments
and the rank bound. Since $LDL^T=-D$, the leading part of the
matrix product in~\eqref{gp:old-reverse} is
\[
 (s^{\rm o}I_2+qD)(sI_2-qD)
 =(s^{\rm o}s-q^2)I_2+q(s-s^{\rm o})D.
\]
Taking half its trace and using $c^2=1+q^2$ gives
$s^{\rm o}(\ri Y)s(\ri Y)+1\to0$ in probability. The scalar rank bound
then gives $s(\ri Y)^2+1\to0$, and the upper half-plane selects $\ri$.

\proofstep{6. Prevent labels from escaping}
If only finitely many particles lie to the left, remove their
terms. The remaining real summands
$x/(x^2+Y^2)-x/(1+x^2)$ decrease with $Y$. An almost surely
convergent subsequence in \eqref{gp:infty} makes their limit finite,
so monotone convergence gives $\sum_{x\geq1}x^{-1}<\infty$.
The bound $Y/(x^2+Y^2)\leq(2x)^{-1}$ would then make the imaginary
part tend to zero. Reflection excludes a finite right side.
For any fixed positive integer $j$, the limiting process has at
least $j$ particles to the right in $[0,R]$ with probability tending
to one as $R\to\infty$. Choose $R$ to avoid deterministic atoms
and use convergence of the integer counts in that interval. The
same is then true of the finite-$N$ process for large $N$, which
bounds its $j$th neighboring position. Apply the same argument on
the left. This is the additional ingredient, beyond local counting
tightness, that bounds each label. Rank interlacing then bounds the
same-label reverse shift. Diagonal extraction
gives the joint labelled state and the asserted moments and marking
at each offset.

\proofstep{7. Retain the deterministic label under reversal}
The label in the inverse map is fixed by the exact inertia identity
\begin{equation}\label{gp:inertia}
 N_>(\widetilde A,t)-N_>(A,t)=n_-(D/2+K(t))-1
                        =n_-(S_{N,i}(w)+qD)-1,
 \qquad t=\lambda_i+\frac{w}{N\sigma},
\end{equation}
where $w$ is real and $t$ lies in a common spectral gap. Here $N_>$
counts eigenvalues strictly above its argument, and $n_-$
counts negative eigenvalues. For completeness, apply Sylvester's law of inertia to the block
matrix
\[
 \begin{pmatrix}A-t&W\\W^T&-D/2\end{pmatrix}.
\]
Taking the Schur complement of $-D/2$ gives $\widetilde A-t$,
whereas taking that of $A-t$ gives $-(D/2+K(t))$.
For a common real gap, all the relevant matrices are invertible.
Since $-D/2$ has one positive and one negative eigenvalue, equality
of the positive inertias gives the first equality in
\eqref{gp:inertia}. The second follows from the positive congruence
$D/2+K=\sigma V^{1/2}(S_{N,i}+qD)V^{1/2}$. Integer-valued measure convergence excludes
poles on compact subsets of common gaps; Nevanlinna bounds then give
analytic convergence there. Locally constant inertia therefore fixes
the limiting counting offset and label zero. Exact proposal reversal
preserves the recentered law; Gaussian rotations preserve it with
all labels fixed.
The seven steps also show the roles of the two normalizations.
Keeping the old origin makes the rank-two algebra exact. Recentring
at the same ordered label is postponed until tightness and the
inertia identity specify that label in the limit. At every other
fixed relative label the original index is still deterministic and
bulk, so the same construction gives the same moment bounds. No
translation invariance or label-stationarity of the limiting point
process is assumed in this argument.
\end{proof}

\begin{lemma}[Simplicity of the tagged limiting process]\label{gp:simple}
The labelled scalar configuration in Proposition~\ref{gp:state} is
almost surely simple.
\end{lemma}

\begin{proof}
We compare the multiplicities before and after one switching.
The matrix field can cancel at most two scalar particles at an old
location, while a new location is generated by a zero of a
$2\times2$ determinant. Independent Gaussian marks make the
exceptional degeneracies have probability zero. Stationarity then
transfers simplicity of the image back to the original configuration.

At an old pole of multiplicity $m\geq2$, its Gaussian matrix pole weight
$R$ is positive definite. With $t=w-x$,
$S+qD=-R/t+O(1)$, so $(S+qD)^{-1}=-tR^{-1}+O(t^2)$.
The matrix image is analytic. Moreover
$\det(S+qD)=t^{-2}\det R\,(1+O(t))$, so
$\partial_w\log\det(S+qD)=-2/t+O(1)$.
The scalar pole is $-m/t$ in the variable $t=w-x$; after subtracting
this logarithmic derivative its coefficient is $-(m-2)/t$.
Thus its multiplicity is $m-2$.
Stationary Gaussian marking excludes $m\geq3$, since every scalar
atom has a nonzero matrix pole weight.

An old double pole disappears. At an old simple pole write
$S+qD=-\boxi\boxi^T/t+A_0+O(t)$. The determinant pole coefficient is
$-\boxi^T\operatorname{adj}(A_0)\boxi$. The matrix $A_0$ is independent
of $\boxi$ and is nonzero almost surely: after conditioning on all but
one other mark, it is a fixed matrix plus a nonzero multiple of a
Gaussian rank-one matrix. Thus this coefficient is nonzero almost
surely, and the scalar logarithmic derivative cancels the simple pole.

In an old spectral gap, $S'(x)$ is positive definite. If one eigenvalue of $S(x)+qD$ vanishes and the other does not,
the derivative of the vanishing eigenvalue is
$\bm v^TS'(x)\bm v>0$ for its unit eigenvector $\bm v$.
The determinant therefore has a simple zero. A multiple zero can
only occur if both eigenvalues vanish, namely $S(x)+qD=0$. To exclude the latter,
condition on all but one Gaussian mark at a particle $x_p$ outside
the gap and write $S(x)+qD=A_0(x)+\boxi\boxi^T/(x_p-x)$.
Since $A_0'(x)>0$, its eigenvalues strictly increase and
$\det A_0(x)$ has at most two zeros in the gap. At either conditionally
fixed point, $\boxi\boxi^T=-(x_p-x)A_0(x)$ has probability zero.
A countable union over gaps and stationarity prove simplicity.
The strict inequalities used here hold almost surely
simultaneously in every gap: after removing any one fixed mark,
there remain infinitely many independent Gaussian marks spanning
$\R^2$, and each contributes a positive coefficient to $S'(x)$.
There are only countably many particles and gaps, so all the
zero-probability exceptional events can be discarded at once.
Each old scalar pole then disappears, and every newly created pole
is simple. By the law-preserving recentered switching, the original
scalar configuration is simple as well.
\end{proof}

\begin{lemma}[The recentered stationary flow]\label{gp:flow}
Apply $\Phi_\theta=\mathcal T\mathcal R_\theta\mathcal T$ with
the label convention \eqref{gp:inertia}, and recenter at its label
zero. This is a law-preserving flow with locally Lipschitz ordered
particle trajectories. Put $R_j=\boxi_j\boxi_j^T$,
$\zeta_j=(R_j)_{12}$, and
$H_j=\lim_{w\to x_j}(S(w)-R_j/(x_j-w))$.
Thus $R_j$ is the matrix pole weight and $H_j$ the regular part;
the complex-analytic residue at $w=x_j$ is $-R_j$.
Its moving-frame vector field is
\begin{align*}
 \mathcal V_0x_j&=-2\kappa(\zeta_j-\zeta_0),&
 \mathcal V_0R_j&=\alpha[L,R_j]+\kappa(H_jPR_j+R_jPH_j),\\
 \mathcal V_0s&=2\kappa(S_{12}'-\zeta_0s'),&
 \mathcal V_0S&=\alpha[L,S]+\kappa(P+SPS)-2\kappa\zeta_0S',
\end{align*}
with $\alpha,\kappa$ from \eqref{lp:vector}. Invariant
differentiation $\E\mathcal V_0F=0$ holds for smooth functions of
finitely many off-real fields whose function and derivatives have
polynomial growth. It also holds for finite-coordinate tests involving
pole weights, with compact position cutoffs and collision cutoffs at each
pole whose regular part occurs in the derivative.
\end{lemma}

\begin{proof}
Off the real axis the formulas follow from
Lemma~\ref{lp:flow}. For coordinate integration by parts we need
more: particle paths must remain controlled when they approach a
collision, and the chosen label must not jump. We establish these
facts by a finite symmetric matrix linearization on each bounded
real interval before matching pole coefficients.

For the intermediate unrooted field $S^1=\mathcal T S$, the total
counting shift at a common real gap is
\[
 [n_-(S+qD)-1]+[n_-(O_\theta S^1O_\theta^T+qD)-1]\in[-2,2].
\]
Its changes are precisely crossings of final poles. Thus each ordered
pole stays between the original labels offset by two, and the count
at $\theta=0$ fixes its continuous branch.

To handle possible collisions, we use a finite matrix representation.
On a bounded real interval, remove the finitely many intermediate poles
$t_1,\ldots,t_m$ and write
\[
 O_\theta S^1(w)O_\theta^T+qD
 =H_\theta(w)+V_\theta\diag(t_\ell-w)^{-1}V_\theta^T.
\]
The regular part is analytic across the interval, and
$H_\theta'(x)$ is positive definite, since the infinitely many
outside Gaussian pole weights span $\R^2$. The finite symmetric matrix
\[
 Q_\theta(w)=
 \begin{pmatrix}H_\theta(w)&V_\theta\\
 V_\theta^T&\diag(w-t_\ell)\end{pmatrix}
\]
satisfies
\[
 \det Q_\theta(w)=\prod_{\ell=1}^m(w-t_\ell)
 \det(O_\theta S^1(w)O_\theta^T+qD),\qquad
 \partial_wQ_\theta=\diag(H_\theta',I_m)>0.
\]
The product cancels the intermediate scalar poles in the logarithmic
derivative. Thus the determinant zeros are exactly the final scalar
poles, with multiplicities. On compact parameter intervals this derivative
has a strictly positive lower bound. Weyl's inequality controls
the change of each ordered matrix eigenvalue by the norm of the
matrix increment, which is locally Lipschitz in $\theta$.
Increasing $w$ increases every eigenvalue at a rate bounded below
by the smallest eigenvalue of $\partial_wQ_\theta$.
Dividing the first bound by this positive lower bound gives a
finite random Lipschitz constant for each ordered zero.
This proves local Lipschitz continuity even at multiple roots. At simple poles the implicit
function theorem makes them analytic.

Near a simple pole, write
$S(w)=R_j/(x_j-w)+H_j+O(w-x_j)$.
The rank-one identity
$R_jPR_j=(\boxi_j^TP\boxi_j)R_j=2\zeta_jR_j$
identifies the double pole of $\kappa SPS$.
On the other hand,
\[
 \mathcal V\left(\frac{R_j}{x_j-w}\right)
 =\frac{\mathcal VR_j}{x_j-w}
   -\frac{R_j\mathcal Vx_j}{(x_j-w)^2}.
\]
Matching double poles yields $\mathcal Vx_j=-2\kappa\zeta_j$.
Matching simple poles yields
$\mathcal VR_j=\alpha[L,R_j]+\kappa(H_jPR_j+R_jPH_j)$.
The minus sign in the double-pole differentiation is important.
Subtracting the root velocity gives the moving-frame formulas.
By stationarity, the configuration is almost surely simple at each fixed
time. Fubini's theorem therefore implies that collision times form a set
of Lebesgue measure zero almost surely.
Off-real derivatives are integrable by the marked moments and
Gaussian moments. If $|x_j|\leq R$ and its two neighboring gaps are
at least $\delta$, the matrix Nevanlinna representation gives
\[
 \|H_j\|\leq C_{R,\delta}\bigl(\|S(\ri)\|+|\boxi_j|^2\bigr).
\]
Indeed the compensated kernel of every other particle is bounded by
$[R+(1+R^2)/\delta]/(1+x_\ell^2)$. This bounds the regular parts
on the supports of the collision cutoffs. The locally Lipschitz trajectories and
these integrable derivatives justify the fundamental theorem of
calculus. Stationarity at both endpoints then gives
$\E\mathcal V_0F=0$ for the stated tests.
A collision cutoff is a smooth function that vanishes whenever a
specified neighboring gap is smaller than $\delta$, and is one
when that gap is larger than $2\delta$. It is needed only when the
observable involves a regular part $H_j$ or a pole-weight derivative.
On its support all the bounds above are uniform; near a collision
the observable is identically zero. Thus the local absolutely
continuous pieces fit together, and the exceptional collision times
of Lebesgue measure zero do not contribute to the integrated
stationarity identity.
\end{proof}

\subsection{Rooted loop equation and rigidity}\label{gp:rigidity-section}

Let $\bmx=(x_j)_{j\in\mathbb Z}$ be the increasing configuration of
Proposition~\ref{gp:state}, which is simple by Lemma~\ref{gp:simple}, with $x_0=0$. Retain $\zeta_j=\xi_{j1}\xi_{j2}$. The unrooted vector field
\eqref{lp:vector} moves the pole $x_j$ at velocity $-2\kappa\zeta_j$.
Consequently its expression $\mathcal V_0$ in the frame following label zero is
\[
 \mathcal V_0s=2\kappa(S_{12}'-\zeta_0s'),\qquad
 \mathcal V_0S_{12}=\alpha(S_{11}-S_{22})
 +\kappa(1+S_{11}S_{22}+S_{12}^2)-2\kappa\zeta_0S_{12}'.
\]
Deleting the tagged atom leaves the Herglotz function $r(w)=s(w)+1/w$.
For finitely many nonreal arguments $v$ and a differentiable
function $F$ of the corresponding values $r(v)$, invariant differentiation
of $S_{12}(w)F$ gives
\begin{equation}\label{gp:loop}
 0=\E[(1+r(w)^2+r'(w)-2r(w)/w)F]
 +2\sum_v\E[C(w,v)D_vF],
\end{equation}
Here $D_vF$ denotes differentiation with respect to the \emph{value}
$r(v)$ in the test function, whereas $\partial_v$ in the next display
differentiates the \emph{spectral parameter}. These two derivatives
should not be confused. When conjugate values occur, they are
interpreted as Wirtinger derivatives: for $z=x+\ri y$,
$\partial_z=(\partial_x-\ri\partial_y)/2$ and
$\partial_{\bar z}=(\partial_x+\ri\partial_y)/2$.
We have
\[
 C(w,v)=\partial_v\frac{r(w)-r(v)}{w-v}+\frac{r'(v)}w,
 \qquad C(w,w)=\frac12r''(w)+\frac{r'(w)}w.
\]
Indeed, conditional Gaussian averaging gives
\[
 \E[S_{12}(w)\zeta_0\mid s]=-\frac1w,\qquad
 \E[S_{12}'(w)\zeta_0\mid s]=\frac1{w^2},\qquad
 \E[S_{12}(w)S_{12}(v)\mid s]=\frac{s(w)-s(v)}{w-v}.
\]
Together with \eqref{rs:wick-two}, these identities give
\eqref{gp:loop}. Here is the cancellation explicitly. Since $s=r-1/w$,
\[
 1+s(w)^2+s'(w)-\frac2{w^2}
 =1+r(w)^2+r'(w)-\frac{2r(w)}w.
\]
For the cross term, Gaussian averaging before deleting the root
gives
\[
 \partial_v\frac{s(w)-s(v)}{w-v}+\frac{s'(v)}w.
\]
The identity
$(s(w)-s(v))/(w-v)=(r(w)-r(v))/(w-v)+1/(wv)$
shows that its additional $-1/(wv^2)$ cancels the
$1/(wv^2)$ from $s'(v)/w$, leaving exactly $C(w,v)$.
At $v=w$, Taylor expansion of the divided difference gives
$\frac12r''(w)$, accounting for the factor one half.

We use \eqref{gp:loop} for functions whose first derivatives have
polynomial growth in the fixed off-real fields and whose denominators
are separated from zero on the state space. Their derivatives along
$\mathcal V_0$ are integrable by the marked moment bounds. The labelled
flow of Lemma~\ref{gp:flow}, stationarity at each fixed time, and Fubini
justify differentiation of the expectation. The displayed derivatives
involve only off-real fields and the root mark, so this argument requires
no derivative of $\zeta_0$.

\begin{lemma}[Local law for the tagged limit]\label{gp:local}
There is an absolute constant $C$ such that, for every integer $p\geq1$
and every $w\in\bC^+$,
\begin{equation}\label{gp:local-bound}
 \|s(w)-\ri\|_{2p}\leq\frac{C\sqrt p}{\Im w}.
\end{equation}
The constants are uniform over all deterministic tag marginals. Moreover,
\begin{equation}\label{gp:gap-tail}
 \PP(x_1>R)+\PP(-x_{-1}>R)\leq Ce^{-cR^2},\qquad R\geq1.
\end{equation}
\end{lemma}
\begin{proof}
The test is chosen so that the leading nonlinear term becomes a
positive even moment. The factor $z+\ri$ in the denominator never
vanishes for a Herglotz value $z$, and retaining it in the estimates
cancels the otherwise troublesome factor $\Im z$ in the derivative
bounds. This avoids any prior quantitative local law for the
limiting process.

Put $Y=\Im w$, $z=r(w)$, and $\Delta=z-\ri$. The sums below omit the root label $j=0$.
For $w=u+\ri Y$,
\[
 \Im r(w)=\sum_{j\ne0}\frac{Y}{|x_j-w|^2},\quad
 r'(w)=\sum_{j\ne0}\frac1{(x_j-w)^2},\quad
 r''(w)=2\sum_{j\ne0}\frac1{(x_j-w)^3}.
\]
Use $|x_j-w|\ge Y$ to compare these absolutely convergent sums.
The same comparison for the divided difference, at $v=w$ and
$v=\bar w$, gives
\[
 |r'(w)|\leq\frac{\Im z}{Y},\qquad
 |r''(w)|\leq\frac{2\Im z}{Y^2},\qquad
 |C(w,w)|+|C(w,\bar w)|\leq\frac{4\Im z}{Y^2}.
\]
In \eqref{gp:loop} take the two scalar arguments $z=r(w)$ and
$\bar z=r(\bar w)$, and set
\[
 F(z,\bar z)=\frac{(z-\ri)^{p-1}(\bar z+\ri)^p}{z+\ri}.
\]
Since $\Im z\geq0$, the denominator has modulus at least one, and
$|\Delta|$, $|z|$, and $\Im z$ are each at most $|z+\ri|$. Thus
\[
 (1+z^2)F=|\Delta|^{2p},\qquad
 |F_z|+|F_{\bar z}|\leq
 \frac{2p|\Delta|^{2p-2}}{|z+\ri|}.
\]
The term with coefficient $p-1$ is omitted when $p=1$. Keeping the
denominator in this bound cancels the factor $\Im z$ from the derivative
estimates. If $M=\|\Delta\|_{2p}$, \eqref{gp:loop} and H\"older yield
\[
 M^{2p}\leq\frac3Y M^{2p-1}+\frac{8p}{Y^2}M^{2p-2}.
\]
For $M>0$, division by $M^{2p-2}$ gives a quadratic inequality, so
\[
 M\leq\frac{3+\sqrt{9+32p}}{2Y}.
\]
The case $M=0$ is immediate. Adding $|1/w|\leq1/Y$ proves
\eqref{gp:local-bound}. In particular, for a sufficiently small $c>0$,
\[
 \sup_{w\in\bC^+}\E\exp\bigl(c(\Im w)^2|s(w)-\ri|^2\bigr)<\infty.
\]

On $x_1>R$, every atom is at distance at least $R/2$ from the midpoint $R/2$.
For $\eta=R/100$, termwise comparison of the Poisson kernels gives
\[
 \Im s(R/2+\ri \eta)\leq\frac1{20}\Im s(R/2+\ri R).
\]
If both imaginary parts were in $[3/4,5/4]$,
the displayed comparison would imply $3/4\le(5/4)/20$, a
contradiction. Therefore at least one of the two values of
$|s-\ri|$ is at least $1/4$. Their imaginary heights are fixed
positive multiples of $R$, so exponential Markov gives a bound
$Ce^{-cR^2}$. This gives the right-gap estimate in \eqref{gp:gap-tail};
the left-gap estimate follows in the same way on $(-R,0)$.
The exponential moment assertion follows directly from the even
moment bounds: expand the exponential and use
$\E[Y^{2p}|s(w)-\ri|^{2p}]\le(C\sqrt p)^{2p}$.
Stirling's estimate $p!\ge(p/e)^p$ makes the resulting series
geometric when $c$ is sufficiently small. All constants therefore
remain uniform in the original deterministic tag.
\end{proof}

\begin{lemma}[Rigidity and principal-value normalization]\label{gp:rigidity}
For every $\epsilon,D>0$ there is $C_{\epsilon,D}$ such that, for every
integer $K\geq2$,
\begin{equation}\label{gp:rigidity-bound}
 \PP\left(\max_{|j|\leq2K}|x_j-\pi j|>K^\epsilon\right)
 \leq C_{\epsilon,D}K^{-D}.
\end{equation}
For every fixed $1\le p<\infty$, the maximum on the left has $L^p$ norm at
most $C_{p,\epsilon}K^\epsilon$. Almost surely,
\begin{equation}\label{gp:pv}
 s(w)=\lim_{M\to\infty}\sum_{j=-M}^M\frac1{x_j-w},
\end{equation}
locally uniformly off the poles.
\end{lemma}
\begin{proof}
The local law is a bound for a smoothed counting measure. To turn
it into labelled rigidity, we first make it uniform on a collection
of complex rectangles, then integrate along a contour to estimate
interval counts, and finally invert the count starting at $x_0=0$.
The principal-value identity is obtained only after this rigidity
has been proved.

For a dyadic height $h\geq1$, Cauchy's formula on disks of radius
comparable to $h$ and \eqref{gp:local-bound} give
\[
 \left\|\sup_{|u-u_0|\leq h/4,\ h\leq v\leq2h}
 v|s(u+\ri v)-\ri|\right\|_{2p}\leq C\sqrt p.
\]
The surrounding circles have heights comparable to $h$. Covering
$|u|\leq4R$, $1\leq v\leq8R$ requires $O(R+\log R)$ such rectangles.
At a threshold $A_D\sqrt{\log R}$, the probability for one
rectangle is at most
$(C\sqrt p/(A_D\sqrt{\log R}))^{2p}$.
Choose $p$ to be a sufficiently large multiple of $\log R$ and
then $A_D$ large enough. This probability is $O(R^{-D-2})$.
A union bound over the $O(R+\log R)$ rectangles gives,
with probability at least $1-C_DR^{-D}$,
\begin{equation}\label{gp:uniform}
 \sup_{|u|\leq4R,\ 1\leq v\leq8R}v|s(u+\ri v)-\ri|
 \leq C_D\sqrt{\log R}.
\end{equation}
Here we work throughout with the infinite labelled limit constructed above.

Fix $a,b\in[-R,R]$, $a\leq b$, and let
\[
 f(t)=\frac1\pi\bigl(\arctan(b-t)-\arctan(a-t)\bigr).
\]
Positivity and the scalar representation imply
\[
 \int f\dd \mu-\frac{b-a}{\pi}
 =\frac1\pi\Im\int_a^b(s(u+\ri)-\ri)\dd u.
\]
Deform the last integral to height $2R$. Equation~\eqref{gp:uniform}
bounds the top by $C_D\sqrt{\log R}$ and the two vertical sides by
$C_D(\log R)^{3/2}$.

For $c\in[-R,R]$ and $1\leq r\leq4R$, the same event gives
\[
 \mu([c-r,c+r])\leq2r\Im s(c+\ri r)
 \leq C(r+\sqrt{\log R}).
\]
The difference between $f$ and the interval indicator is bounded by one
within distance one of an endpoint, by $C/r$ on an endpoint annulus of
radius $r\in[1,4R]$, and by $CR/r^2$ outside $[-4R,4R]$.
Dyadic summation bounds the near contribution by
$C_D(\log R+\sqrt{\log R})$. For the far annuli,
\eqref{gp:local-bound} at $\ri 2^nR$ ensures simultaneously
$\mu([-2^nR,2^nR])\leq C2^nR$, $n\geq0$, outside an event of
probability $Ce^{-cR^2}$. Their contributions sum to a constant. Thus
\[
 \sup_{a,b\in[-R,R],\ a\leq b}
 \left|\mu((a,b])-\frac{b-a}{\pi}\right|
 \leq C_D(\log R)^{3/2}
\]
with probability at least $1-C_DR^{-D}$. Endpoint atoms contribute at
most two. With $R=10\pi K$, the estimate first confines all particles with
$|j|\leq2K$ to the interval; counting from $x_0=0$ then gives
\eqref{gp:rigidity-bound}.

For the $L^p$ conclusion, let
$H_K=\max_{|j|\le2K}|x_j-\pi j|$.
The nearest-gap tail at every fixed tag gives a uniform bound for
the $L^{2p}$ norm of each adjacent gap. Expressing a position as a
sum of at most $2K$ gaps and applying Minkowski gives
$\|H_K\|_{2p}\le C_pK$. On the good rigidity event,
$H_K\le K^\epsilon$; on its complement, Cauchy--Schwarz gives
\[
 \E H_K^p\mathbf1_{H_K>K^\epsilon}
 \le\|H_K\|_{2p}^p
       \PP(H_K>K^\epsilon)^{1/2}
 \le C K^{p-D/2}.
\]
The failure exponent $D$ can be chosen arbitrarily large. This
proves the stated moment bound, which will later be used for
conditional means, not only for typical boundary locations.

Finally, apply \eqref{gp:rigidity-bound} on dyadic scales with a fixed
$\epsilon<1$ and summable failure probabilities. Almost surely,
$x_j=\pi j+O(|j|^\epsilon)$ for all sufficiently large $|j|$. The paired
sum in \eqref{gp:pv} therefore converges locally, and its difference
from the equally spaced sum tends to zero at imaginary infinity by
domination with $C|j|^{\epsilon-2}$. The equally spaced paired sum is
\[
 \lim_{M\to\infty}\sum_{j=-M}^M\frac1{\pi j-w}=-\cot w,
\]
which tends to $\ri$ at $w=\ri Y$, because
$-\cot(\ri Y)=\ri\coth Y$.
Moreover, for large $|j|$,
\[
 \left|\frac1{x_j-\ri Y}-\frac1{\pi j-\ri Y}\right|
 \le C\frac{|j|^\epsilon}{j^2+Y^2}
 \le C|j|^{\epsilon-2}.
\]
The majorant is summable when $\epsilon<1$, and each fixed term
tends to zero as $Y\to\infty$. Dominated convergence proves the
normalization of the paired sum.
The paired sum and $s$ have the same representing measure and zero
linear coefficient. Their difference is therefore a real constant,
which vanishes by the normalization at infinity.

\end{proof}

\subsection{Conditional Gibbs measures}\label{gp:gibbs-section}

\subsubsection{Coordinate integration by parts}\label{gp:coordinate}

Let $\mathsf P$ denote the law of the simple rooted array
$\bmx=(x_j)_{j\in\mathbb Z}$ constructed above. Its transform $s$
is determined by the principal-value formula~\eqref{gp:pv}.
Retain $\zeta_j=\xi_{j1}\xi_{j2}$ and write the regular scalar part as
\[
 h_j=\lim_{w\to x_j}\left(s(w)-\frac1{x_j-w}\right).
\]
The expectation in the integration-by-parts identities below is
under this rooted state and, when marks occur, their conditional
Gaussian law. Since $s$ is determined by $\bmx$, conditioning on
$\bmx$ already fixes the entire scalar state. Write
\[
 \psi(t)=\frac{t}{1+t^2},\qquad a=\Re s(\ri),\qquad
 s(w)=a+\sum_{\ell\in\mathbb Z}
 \left(\frac1{x_\ell-w}-\psi(x_\ell)\right).
\]
The series converges locally off the poles, and
$\sum_\ell(1+x_\ell^2)^{-1}=\Im s(\ri)<\infty$.
The moving-frame pole-coefficient calculation gives
\[
 \mathcal V_0x_\ell=-2\kappa(\zeta_\ell-\zeta_0),\qquad
 \mathcal V_0a=-2\kappa\sum_{\ell\ne0}
       (\zeta_\ell-\zeta_0)\psi'(x_\ell),\qquad
 \E(\mathcal V_0\zeta_j\mid\bmx)=2\kappa h_j.
\]
Taking real parts of $\mathcal V_0s$ at $\ri$ gives the middle
identity; $|\psi'(t)|\leq(1+t^2)^{-1}$ makes its series integrable.

For $j\ne0$, define $D_j=\partial_{x_j}+\psi'(x_j)\partial_a$.
Gaussian contraction in the invariant derivatives of $\zeta_jF$ and
$\zeta_0F$ gives
\begin{equation}\label{gp:ibp}
 \E D_jF=\E h_jF,\qquad
 \E\sum_{\ell\ne0}D_\ell F=-\E h_0F.
\end{equation}
We use smooth cylinder tests compactly supported in $a$ and finitely
many positions, with neighboring gaps at least $\delta$ at each
differentiated pole weight. For $0<\eta<\delta/2$,
\[
 |h_j-\Re s(x_j+\ri \eta)|
 \leq\frac\eta\delta\Im s(x_j+\ri \eta).
\]
With $\eta=\delta/4$, the fixed-tag moments bound the drift.
For the second identity, we separate the root; its scalar-derivative
coefficient $\sum_{\ell\ne0}\psi'(x_\ell)$ converges absolutely.
All other derivatives involve Gaussian mark products.

Here is the Gaussian contraction yielding~\eqref{gp:ibp}.
For a scalar cylinder test $F$, the chain rule gives
\[
 \mathcal V_0F=-2\kappa\sum_{\ell\ne0}
                   (\zeta_\ell-\zeta_0)D_\ell F.
\]
The sum is finite except for the absolutely convergent coefficient
of $\partial_aF$. The conditional Gaussian law gives
$\E[\zeta_j\mid\bmx]=0$ and
$\E[\zeta_j\zeta_\ell\mid\bmx]=\mathbf1_{j=\ell}$.
Thus, for $j\ne0$,
\[
 0=\E\mathcal V_0(\zeta_jF)
   =2\kappa\E h_jF-2\kappa\E D_jF.
\]
For $j=0$, the contraction
$\E[\zeta_0(\zeta_\ell-\zeta_0)\mid\bmx]=-1$ for $\ell\ne0$
changes the sign and gives the second identity.
To check the remaining mean, $H_j$ is independent of the mark
$\boxi_j$ conditionally on the scalar state, and
$\E[H_j\mid\bmx]=h_jI_2$. Taking the $(1,2)$ entry of the
pole-weight equation and using $\E[R_j\mid\bmx]=I_2$ gives
$\E(\mathcal V_0\zeta_j\mid\bmx)=2\kappa h_j$.
The commutator term has zero conditional expectation.

The correction in $D_j$ is not optional. Moving a single particle
while holding the regular background fixed changes its compensator
$\psi(x_j)$, so the scalar coordinate $a$ must change by
$\psi'(x_j)$ times that displacement.
For a finite $I\subset\mathbb Z\setminus\{0\}$ and $j\in I$,
\[
 D_j\left(a-\sum_{k\in I}\psi(x_k)\right)
 =\psi'(x_j)-\psi'(x_j)=0.
\]
Every outside position is also annihilated by $D_j$.
Thus, if we use $a-\sum_{k\in I}\psi(x_k)$ and the outside positions
as the fixed coordinates, $D_j$ is simply the ordinary partial
derivative in the inside position $x_j$. Smooth exterior-cylinder
factors may depend on this corrected scalar. This is the coordinate
change used to disintegrate the identity below.

\subsubsection{Conditional law of a finite block}\label{gp:conditional}

To free the root, translate by $\tau\in\R$: set
$z_j=x_j+\tau$, $s^\tau(w)=s(w-\tau)$, and
\begin{equation}\label{gp:translate}
 a^\tau=a+\sum_\ell\bigl(\psi(x_\ell+\tau)-\psi(x_\ell)\bigr).
\end{equation}
This sum converges absolutely. For $|\tau|\leq R$,
\[
 \sum_\ell|\psi'(x_\ell+\tau)|
 \leq C_R\sum_\ell(1+x_\ell^2)^{-1}=C_R\Im s(\ri).
\]
This bound and Gaussian moments dominate the truncated flow
derivatives, extending the identities to bounded smooth tests of
\eqref{gp:translate}.

Integrate the translated law against $\dd\tau$ to obtain the
sigma-finite measure $\mathsf Q$. It is not a probability measure:
since $z_0=\tau$, its mass on $\{|z_0|\le R\}$ is $2R$.
The normalized probability measure is introduced below.
The derivative
$\widetilde D_j=\partial_{z_j}+\psi'(z_j)\partial_{a^\tau}$
is the pushforward of $D_j$ for $j\ne0$, and of
$\partial_\tau-\sum_{\ell\ne0}D_\ell$ for $j=0$.
Lebesgue integration and \eqref{gp:ibp} give, for every label,
\begin{equation}\label{gp:lift-ibp}
 \int\widetilde D_jF\dd \mathsf Q=\int h_jF\dd \mathsf Q.
\end{equation}
The regular parts are translation invariant, and we retain their
collision cutoffs.

Fix a finite consecutive block $I$, which may contain zero. Choose
$J\notin I$ and a nonnegative smooth compactly supported $\varphi$
with integral one. Then
\begin{equation}\label{gp:normalization}
 \dd\mathsf Q_\varphi=\varphi(z_J)\dd \mathsf Q,
 \qquad \int_\R\varphi(x_J+\tau)\dd \tau=1.
\end{equation}
This probability law preserves every relative observable.

More precisely, for any bounded measurable function $F$ of the
relative array,
\[
 \int F((z_j-z_0)_j)\,\dd\mathsf Q_\varphi
 =\E^{\mathsf P}\left[F((x_j)_j)
                    \int_\R\varphi(x_J+\tau)\,\dd\tau\right]
 =\E^{\mathsf P}F((x_j)_j).
\]
The infinite translation measure $\mathsf Q$ has thus been
normalized without biasing any gap statistic. It is important that
$J\notin I$: using a weight attached to an inside particle would
produce an extra derivative of that weight in the conditional
integration-by-parts formula.
 Since the
normalizing label lies outside the block, inside derivatives annihilate
the weight. Removing a translation cutoff with
derivative $O(R^{-1})$ costs $O(R^{-1})$ by
\eqref{gp:normalization}, so \eqref{gp:lift-ibp} holds under
$\mathsf Q_\varphi$ for inside labels. Translation invariance preserves
the localized drift bounds.

\begin{lemma}[Conditional Gibbs laws from integration by parts]\label{gp:gibbs}
Under $\mathsf Q_\varphi$, condition on the outside positions
$\bmz=(z_\ell)_{\ell\notin I}$. The inside variables are
$\bmt=(t_j)_{j\in I}$, and $\mathcal D_I(\bmz)$ is the chamber
\[
 z_{\min I-1}<t_{\min I}<\cdots<t_{\max I}<z_{\max I+1}.
\]
There is an exterior-measurable real scalar $b$ such that the
conditional law of $(z_j)_{j\in I}$ is
\begin{equation}\label{gp:density}
 \frac{\mathbf1_{\mathcal D_I(\bmz)}(\bmt)}{Z_I(\bmz,b)}
 \prod_{j<k,\ j,k\in I}|t_j-t_k|
 \exp\left(\sum_{j\in I}\mathcal U_{\bmz,b}(t_j)\right)
 \prod_{j\in I}\dd t_j,
\end{equation}
where
\[
 \mathcal U_{\bmz,b}(t)
 =-bt+\sum_{\ell\notin I}
 \left[\log\frac{|t-z_\ell|}{\sqrt{1+z_\ell^2}}
       +t\psi(z_\ell)\right].
\]
The combined summands converge absolutely at infinity. All particles
of $I$, including the former root, are ordinary variables in
\eqref{gp:density}.
\end{lemma}

\begin{proof}
We do not assume beforehand that a conditional density exists.
The integration-by-parts identities first identify the logarithmic
derivative of the conditional measure in the distributional sense.
An elementary argument on the connected ordered chamber then proves
absolute continuity and identifies the density. Finally we show
that the auxiliary scalar being conditioned on is already determined
by the exterior.

\proofstep{1. Disintegrate with the correct scalar fixed}
First condition on the outside positions and also on
$b=a^\tau-\sum_{j\in I}\psi(z_j)$. Let $\nu$ denote this conditional
probability measure on the inside variables $\bmt$; at this point
we do not yet assume it has a density. The inside derivatives
annihilate $b$ and all outside coordinates. Moreover
\[
 h_j=\sum_{k\in I\setminus\{j\}}\frac1{t_k-t_j}
       +b+\sum_{\ell\notin I}
       \left(\frac1{z_\ell-t_j}-\psi(z_\ell)\right).
\]
We localize to rational compact boxes inside the chamber, whose exterior
separation is determined by its two neighbors. A monotone-class
argument extends \eqref{gp:lift-ibp} to bounded exterior factors.
Countable dense test families on these boxes give the conditional
identities on one full-measure set. Conditional probabilities and
identities in this proof are understood for almost every exterior
and scalar $b$, not for every arbitrary boundary configuration.

\proofstep{2. Recover the density from its logarithmic derivative}
For the candidate unnormalized density $w_I$ in~\eqref{gp:density},
\[
 \partial_j\log w_I
 =\sum_{k\in I\setminus\{j\}}\frac1{t_j-t_k}
       -b+\sum_{\ell\notin I}
                  \left(\frac1{t_j-z_\ell}+\psi(z_\ell)\right)
 =-h_j.
\]
The conditional identities say
$\int\partial_jf\,\dd\nu=\int h_jf\,\dd\nu$ for compactly
supported smooth $f$ in the chamber. Hence, in distributions,
$\partial_j\nu=-h_j\nu=(\partial_j\log w_I)\nu$.
On a compact subset of the chamber $w_I$ is strictly positive and
smooth, so the distributional product rule gives
\[
 \partial_j(w_I^{-1}\nu)=0\quad\text{for every }j\in I.
\]
To see why this identifies a density, convolve the locally finite
measure $w_I^{-1}\nu$ with a smooth approximate identity on a ball
compactly contained in the chamber. Every derivative of the
convolution is zero on a slightly smaller ball, so that convolution
is constant there. Passing to the limit shows that the measure is
a constant multiple of Lebesgue measure on that ball. Overlapping
balls have the same constant; connectedness of the ordered chamber
therefore gives $\nu=cw_I\,\dd\bmt$ throughout it.

\proofstep{3. Normalize the density}
For bounded $t$ and large $|z|$, Taylor expansion gives
\[
 \log\frac{|t-z|}{\sqrt{1+z^2}}+t\psi(z)
 =-\frac{1+t^2}{2z^2}+O_R(|z|^{-3}),\qquad |t|\le R.
\]
Thus the tail, and its derivatives on compact subintervals, converge
uniformly because $\sum_\ell(1+z_\ell^2)^{-1}<\infty$.
There are only finitely many nearby exterior factors. Those at the
two boundary particles vanish as an inside particle approaches the
boundary; the Vandermonde factors vanish at inside collisions.
There is no positive blow-up at the chamber boundary. Thus
$0<Z_I<\infty$; simplicity gives mass one in the chamber and
$c=Z_I^{-1}$.

\proofstep{4. Remove the extra scalar conditioning}
For $|\tau|\le R$ and $Y\ge2R+1$, kernel comparison gives
\[
 |s(\ri Y-\tau)-s(\ri Y)|
 \le\frac{C_R}{Y}\Im s(\ri Y).
\]
The height moments and~\eqref{gp:infty} therefore imply
$s^\tau(\ri Y)\to\ri$ in probability uniformly for bounded
$\tau$. Under the probability law $\mathsf Q_\varphi$, the
translation variable is finite almost surely, so its tail can be
removed by taking $R\to\infty$ after $Y\to\infty$. Taking real parts yields
\[
 b=\lim_{Y\to\infty}\sum_{\ell\notin I}
 \left(\psi(z_\ell)-\frac{z_\ell}{z_\ell^2+Y^2}\right)
 \quad\text{in probability}.
\]
The finite inside resolvent sum vanishes. An almost surely convergent
subsequence makes $b$ measurable in the completed exterior sigma-field,
so the additional conditioning can be removed.
This final measurability step matters: conditioning only on the
exterior would otherwise produce a mixture of densities with
different linear potentials $-bt$. The normalization at imaginary
infinity shows that there is a single exterior-determined $b$,
so the conditional law is exactly the Gibbs density stated in the
lemma. The former root is now one of the ordinary inside variables;
there is no hidden atom constrained to stay at zero.
\end{proof}

\subsection{Identification of the gap law}\label{gp:comparison-section}

The preceding argument supplies rigidity and a conditional log-gas
structure. The following proposition isolates the consequence for gaps.
Let $\Theta_1^0$ be the Palm process of $\Theta_1$, with increasing
points $(y_j)_{j\in\mathbb Z}$ and $y_0=0$, and write
$G_j=y_{j+1}-y_j$ for its consecutive gaps. These gaps have mean $\pi$;
they are not divided by $\pi$.

The Palm law describes the configuration seen from a typical
particle, rather than from a deterministic point in space. For a
bounded interval $I$ of positive length and a bounded measurable
rooted test $F$, it is characterized by Campbell's identity
\[
 \E\sum_{x\in\Theta_1\cap I}F(\Theta_1-x)
 =\frac{|I|}{\pi}\,\E^0 F(\Theta_1^0).
\]
Here $\E^0$ is expectation under the Palm probability law. Taking
$F=1$ recovers $\E\Theta_1(I)=|I|/\pi$; in particular, the prefactor
is one for $I=[0,\pi]$. This avoids interpreting Palm conditioning
as conditioning on the probability-zero event of an atom at one
prescribed location.

\begin{proposition}[Gap identification from rigidity and conditional Gibbs laws]
\label{gp:rigid-gibbs}
Let $\mathsf P$ be the law of a strictly increasing array
$(x_j)_{j\in\mathbb Z}$ with $x_0=0$. Suppose the following hold.
\begin{enumerate}[label=\textup{(\roman*)},leftmargin=*]
\item For every $\epsilon,D>0$ and $1\le p<\infty$, there are constants
such that, for every integer $K\ge2$,
\begin{equation}\label{gp:abstract-rigidity}
\begin{split}
 \mathsf P\left(\max_{|j|\le2K}|x_j-\pi j|>K^\epsilon\right)
 &\le C_{\epsilon,D}K^{-D},\\
 \left\|\max_{|j|\le2K}|x_j-\pi j|\right\|_{L^p(\mathsf P)}
 &\le C_{p,\epsilon}K^\epsilon.
\end{split}
\end{equation}
\item For every sufficiently large even $K$, there is a probability law
$\mathsf Q_K$ on strictly increasing arrays $(z_j)_{j\in\mathbb Z}$
such that
\[
 \Law_{\mathsf Q_K}\bigl((z_j-z_0)_{j\in\mathbb Z}\bigr)
 =\mathsf P.
\]
For $I=\{-K,\ldots,K\}$, conditionally on all outside coordinates,
the inside coordinates have density
\begin{equation}\label{gp:abstract-density}
 \frac{\mathbf1_{\{z_{-K-1}<t_{-K}<\cdots<t_K<z_{K+1}\}}}{Z_I}
 \prod_{\substack{j<k\\j,k\in I}}|t_j-t_k|
 \exp\left(-\sum_{j\in I}W_I(t_j)\right)\prod_{j\in I}\dd t_j,
\end{equation}
where $0<Z_I<\infty$ and $W_I$ is an exterior-measurable $C^1$ function
on $(z_{-K-1},z_{K+1})$ satisfying
\begin{equation}\label{gp:abstract-force}
 W_I'(t)=\lim_{R\to\infty}\sum_{K<|\ell|\le R}\frac1{z_\ell-t}.
\end{equation}
The principal value pairs opposite labels and converges locally
uniformly on this interval. The auxiliary laws may depend on $K$.
\end{enumerate}
Then, for every fixed $r\in\mathbb Z$ and $k\ge1$,
\[
 \left(x_{r+1}-x_r,\ldots,x_{r+k}-x_{r+k-1}\right)
 \stackrel{\mathrm d}= (G_0,\ldots,G_{k-1}).
\]
\end{proposition}

The auxiliary laws in (ii) need not be the same for different $K$,
and they need not be invariant under translations. What is required
is that each preserves the same relative array law and allows all
particles of the chosen block to move. This is exactly what
\eqref{gp:normalization} provides. The proposition does not assume
that the original rooted process is already stationary under changing
the tag.

The proof is given in Appendix~\ref{app:gap-identification}. It compares
large conditional blocks with GOE blocks using the local log-gas
universality theorem of Erd\H{o}s and Yau~\cite{EYGap}. Rigidity controls
the effect of completing the distant exterior with Gaussian classical
locations. All particles in the conditional block, including label zero,
are ordinary variables in~\eqref{gp:abstract-density}.

\begin{proof}[Proof of Theorem~\ref{th:gaps}(i)]
We verify the two hypotheses of the abstract gap-identification
proposition for every subsequential tagged law, and only then return
to the finite-$N$ normalization and decreasing eigenvalue order.
Consider any deterministic bulk indices $i_N$. Pass to a subsequence
with $\gamma_{i_N}\to E$ in the strict bulk and then to a jointly
labelled tagged limit with law $\mathsf P$.
Proposition~\ref{gp:state} and Lemma~\ref{gp:simple} give a strictly
increasing array with $x_0=0$. Lemma~\ref{gp:rigidity} gives the first
hypothesis of Proposition~\ref{gp:rigid-gibbs}.

For the second hypothesis, take $I=\{-K,\ldots,K\}$ and use
$\mathsf Q_K=\mathsf Q_\varphi$ from~\eqref{gp:normalization}, with
normalizing label $J=K+2$. Its relative array has law $\mathsf P$.
Lemma~\ref{gp:gibbs} gives~\eqref{gp:abstract-density} with
$W_I=-\mathcal U_{\bmz,b}$. The translated principal-value
identity~\eqref{gp:pv} gives
\[
 a^\tau=\operatorname{PV}\sum_{\ell\in\mathbb Z}\psi(z_\ell),
 \qquad
 b=\operatorname{PV}\sum_{\ell\notin I}\psi(z_\ell).
\]
Differentiating the conditional potential therefore cancels the
compensators and yields~\eqref{gp:abstract-force}. Rigidity gives
local uniform convergence of the paired sum. The proposition now
identifies every fixed consecutive gap vector of the tagged limit.

Since $\rho_d(\gamma_{i_N})/\rho_d(E)\to1$, replacing $N\sigma$
by $N\pi\rho_d(\gamma_{i_N})$ gives the theorem's normalization,
without any further dilation. Take $r=-k$ and reverse the
$k$ coordinates to obtain the descending graph gaps. Finite GOE
reflection maps the corresponding reversed block to right gaps at
label $M/2+1$, whose limit is $(G_0,\ldots,G_{k-1})$ by
Lemma~\ref{gp:goe}. The bulk-label comparison in
\cite[Theorem~2.3]{EYGap} gives the same limit for the GOE gaps at
the original indices, with the factor $N\pi\rho_{\mathrm{sc}}$ used
in the theorem. The fixed factor $\pi$ is absorbed by composing the
test function in the cited comparison with this dilation.
Every fixed neighboring gap vector is tight by Proposition~\ref{gp:state};
thus smooth compactly supported tests identify the weak limit and
convergence extends to $F\in C_b(\R^k)$.
If uniformity over the bulk indices failed, there would be
$\varepsilon_0>0$ and deterministic indices $i_N$ along a subsequence
for which the displayed difference in the theorem is at least
$\varepsilon_0$. Compactness of the bulk quantile interval gives a
further subsequence with $\gamma_{i_N}\to E$.
The tagged construction and identification above force both the graph
and GOE expectations along that subsequence to have the same limit,
a contradiction. Thus the convergence is uniform over the full bulk
index range, not merely pointwise for a fixed proportion $i/N$.
\end{proof}

\section{Gap identification by Gaussian completion}\label{app:gap-identification}

We prove Proposition~\ref{gp:rigid-gibbs}, which is invoked at the end
of Appendix~\ref{sec:gap-universality}. Gaussian completion, exterior
rigidity, and conditional mean bounds reduce its abstract gap law to
the GOE reference law. The fixed-energy point-process theorem and the
correlation upgrade do not use this comparison.
Throughout this appendix, $\mathsf P$ and $\mathsf Q_K$
satisfy the two hypotheses of that proposition. In the conditional
comparison, probabilities of exterior events and unsubscripted
expectations are under $\mathsf Q_K$. Superscripts $\mu_\infty$,
$\mu_M$ and $\mathrm{GOE}_M$ specify, respectively, the two
fixed-exterior conditional laws and the unconditional GOE law.
The GOE identification proof states its own expectation convention.
Only in this appendix, GOE eigenvalues are ordered increasingly,
$\mu_1\le\cdots\le\mu_M$, to agree with the labelled arrays.

Gaussian completion here means replacing only the \emph{distant
exterior particle locations} of a limiting conditional log gas by
suitable GOE classical locations. It does not mean adding a Gaussian
matrix perturbation to the graph. The nearby boundary particles and
the whole inside chamber are kept unchanged. A force estimate then
compares the two conditional densities, after which the cited local
log-gas theorem compares their gaps. The use of that theorem is
isolated in this appendix.

\subsection{Exterior rigidity and conditional means}

Set $\delta_*=1/32$, choose $\xi>0$ sufficiently small for
\cite[Proposition 5.3]{EYGap}, and put
\[
 \xi_*=\frac{\xi^2}{16},\qquad \nu=\frac{\xi^2}{64},\qquad
 0<\epsilon<\min(\xi^2/1000,1/5),\qquad M=K^{16},\qquad L=K^5.
\]
Here $K\to\infty$ through even integers.

These powers provide comfortable separation rather than optimized
rates. The inside block has order $K$ particles, the original
exterior is kept out to order $L=K^5$, and the auxiliary GOE has
size $M=K^{16}$. Thus $K\ll L\ll M$, while
$K^3/L=K^{-2}$ will control the full conditional Hamiltonian error.
The small exponents are chosen before $K$ tends to infinity.
 Work under the auxiliary law
$\mathsf Q_K$ in Proposition~\ref{gp:rigid-gibbs}, put
$I=\{-K,\ldots,K\}$, and condition on the exterior sigma-field
$\mathcal E$.
Define the exterior-measurable midpoint and interval length by
\[
 m=\frac{z_{-K-1}+z_{K+1}}2,\qquad
 \ell=z_{K+1}-z_{-K-1},\qquad J=(-\ell/2,\ell/2).
\]
Use the single centered array $u_j=z_j-m$ for every label $j$.
After conditioning on $\mathcal E$, the coordinates with $j\notin I$
are fixed. Let $\mu_\infty$ be the conditional law of the inside
array $(u_j)_{j\in I}$, and put
$\alpha_j=j\ell/(2K+2)$ for its equally spaced reference locations. The rigidity assumption
\eqref{gp:abstract-rigidity} and a dyadic union give, outside an event of probability $O(K^{-D})$,
\begin{equation}\label{gp:exterior}
 \ell=2\pi(K+1)+O(K^\epsilon),\qquad
 |u_j-\pi j|\leq C(K+|j|)^\epsilon\quad (|j|>K).
\end{equation}
Indeed centering changes each rooted position by $O(K^\epsilon)$.
Moreover $H=\max_{j\in I}|u_j-\alpha_j|$ is translation invariant
and $\|H\|_p\leq C_{p,\epsilon}K^\epsilon$. Conditional Jensen's inequality and
Markov's inequality give
\[
 \PP\left(\max_{j\in I}|\E^{\mu_\infty}u_j-\alpha_j|>K^{\xi_*}\right)
 \leq C_{p,\epsilon}K^{-p(\xi_*-\epsilon)}.
\]
Thus, for every $D$, an exterior event of probability $1-O(K^{-D})$
satisfies \eqref{gp:exterior} and
\begin{equation}\label{gp:means}
 \max_{j\in I}|\E^{\mu_\infty}u_j-\alpha_j|\leq K^{\xi_*}.
\end{equation}

To check that the event of good conditional means is an exterior
event, observe that
$\max_j|\E[u_j\mid\mathcal E]-\alpha_j|
\le\E[H\mid\mathcal E]$. Conditional Jensen gives
$\E(\E[H\mid\mathcal E])^p\le\E H^p$.
This proves the displayed probability estimate without assuming
independence between the inside and outside coordinates. Since
$\epsilon<\xi_*$, taking $p$ sufficiently large gives any prescribed
power $K^{-D}$. The $L^p$ part of the rigidity hypothesis is used
here; a high-probability position estimate alone would not directly
control these conditional means.

Fix an integer $r$ and $k\ge1$. The two gap vectors are
\[
 \begin{aligned}
 \bmg_r&=(u_{r+1}-u_r,\ldots,u_{r+k}-u_{r+k-1}),\\
 \bmg_{M,r}&=M
       (\mu_{M/2+r+1}-\mu_{M/2+r},\ldots,
        \mu_{M/2+r+k}-\mu_{M/2+r+k-1}).
 \end{aligned}
\]
The first is unaffected by centering at $m$. For fixed $r,k$, all
its labels lie in the inside block once $K$ is large enough.
The GOE factor is $M$, because $\pi\rho_{\mathrm{sc}}(0)=1$;
both gap vectors therefore have the mean-spacing-$\pi$ convention.

\subsection{Comparison by Gaussian completion}

\begin{lemma}[Gaussian completion and conditional comparison]\label{gp:comparison}
For $O\in C_c^\infty(\R^k)$, uniformly over exteriors satisfying
\eqref{gp:exterior} and \eqref{gp:means},
\[
 \left|\E^{\mu_\infty}O(\bmg_r)
       -\E^{\mathrm{GOE}_M}O(\bmg_{M,r})\right|
 \leq C_O\bigl(K^{-2}+K^{-c_*}+e^{-cM^{c_2}}\bigr)
\]
for some positive $c_*,c,c_2$.
\end{lemma}
\begin{proof}
\proofstep{1. Keep the near boundary and complete only the far boundary}
The comparison measure must have the same ordered chamber as
$\mu_\infty$. Keeping all exterior locations with $K<|j|\le L$
ensures this. The small outward shift in the farther Gaussian
locations prevents an ordering defect where the two descriptions
meet.

Use GOE coordinates $u=M\lambda$ and relative classical positions
\[
 \mathcal J_M=\{-M/2+1,\ldots,M/2\},\qquad
 \int_{-2}^{\Gamma_j/M}\rho_{\mathrm{sc}}(v)\dd v
       =\frac12+\frac jM.
\]
Then $\Gamma_{-j}=-\Gamma_j$ for $1\leq j<M/2$, the unpaired
position is $\Gamma_{M/2}=2M$, and
\[
 |\Gamma_j|\geq\pi|j|,\qquad
 \Gamma_j=\pi j+O(j^3/M^2)\quad (|j|\leq2L).
\]
With $\Delta=L^{2\epsilon}$, complete the exterior by
\[
 y_j=
 \begin{cases}
  u_j,&K<|j|\leq L,\\
  \Gamma_j+\operatorname{sgn}(j)\Delta,&j\in\mathcal J_M,\ |j|>L.
 \end{cases}
\]
The outward shift dominates both junction errors. Hence the array is
ordered, its configuration interval is $J$, and
\begin{equation}\label{gp:good-boundary}
 \max_{j\in\mathcal J_M\setminus I}|y_j-\Gamma_j|\leq CK^{10\epsilon}.
\end{equation}
Its conditional GOE law $\mu_M$ has one-particle Hamiltonian
\[
 W_M(v)=\frac{v^2}{4M}
       -\sum_{j\in\mathcal J_M\setminus I}\log|v-y_j|.
\]
\proofstep{2. Compare the forces and then the whole Hamiltonians}
By the conditional-force assumption~\eqref{gp:abstract-force}, the
corresponding force for $\mu_\infty$ is
$W_\infty'(v)=\operatorname{PV}\sum_{|j|>K}(u_j-v)^{-1}$.
The forces agree for $K<|j|\leq L$. Their remaining terms are
\[
 R_\infty(v)=\operatorname{PV}\sum_{|j|>L}\frac1{u_j-v},\qquad
 R_M(v)=\frac v{2M}+\sum_{j\in\mathcal J_M,\ |j|>L}\frac1{y_j-v}.
\]
Opposite-label pairing and the inverse-square tail give
\[
 |R_\infty(0)|\leq C\sum_{j>L}j^{\epsilon-2}\leq CL^{\epsilon-1},\qquad
 R_M(0)=\frac1{2M+\Delta},\qquad
 \sup_J(|R_\infty'|+|R_M'|)\leq C/L.
\]
Since $L^\epsilon\leq K$, the force difference is at most $CK/L$.
Choose additive constants so that $W_M(0)=W_\infty(0)$;
this does not change either conditional probability law. The
force bound and $|v|\le CK$ imply
\[
 \sup_{v\in J}|W_M(v)-W_\infty(v)|\le CK^2/L.
\]
There are $2K+1$ inside particles. Since their Vandermonde factors
are identical, the \emph{full} Hamiltonians therefore differ by at
most $CK^3/L=CK^{-2}$ uniformly throughout the chamber.
If two unnormalized densities differ by a multiplicative factor
between $e^{-\delta}$ and $e^\delta$, their normalizing constants
satisfy the same bound, and the ratio of normalized densities lies
between $e^{-2\delta}$ and $e^{2\delta}$. Hence their total
variation distance is at most $e^{2\delta}-1$.
Also $|u_j|\le CK$ on this common chamber, so a coordinate mean
changes by at most $CK$ times the total variation distance.
These observations give
\begin{equation}\label{gp:finite-error}
 \|\mu_\infty-\mu_M\|_{\TV}\leq CK^{-2},\qquad
 \max_{j\in I}|\E^{\mu_\infty}u_j-\E^{\mu_M}u_j|\leq CK^{-1}.
\end{equation}
In particular, \eqref{gp:means} transfers to $\mu_M$ with an
adjusted constant.

\proofstep{3. Verify the good-boundary and conditional-mean hypotheses}
We verify the hypotheses of \cite{EYGap} with size $M$ and block
center $M/2$. The relevant powers are
\[
 M^{\delta_*}\leq K\leq M^{1/4},\qquad
 M^\nu=K^{\xi^2/4},\qquad
 M^{\xi_*\delta_*/2}=K^{\xi^2/64}.
\]
Both tolerances dominate the bound in \eqref{gp:good-boundary}. After division by
$M$, this gives all three bounds in \cite[(4.7)]{EYGap}, with exponent
$\nu=\xi^2\delta_*/2$ or $\xi_*\delta_*/2$ and a fixed bulk margin;
the nonbulk bounds are weaker. Equation~\eqref{gp:exterior} gives
\cite[(4.9)]{EYGap}, absorbing its $O(1)$ particle-count difference
in $K^\xi$.

Apply \cite[Theorem 4.2]{EYGap} with exponent $\xi_*$. The stronger
boundary condition and transferred means give Gaussian rigidity on
scale $K^{\xi_*}$. In the units of~\cite{EYGap}, the particle coordinates are $u_j/M$;
thus a tolerance $M^{-1}K^{\xi_*}$ there is $K^{\xi_*}$ here.
The target threshold is $M^\nu=K^{\xi^2/4}$, whereas the
conditional fluctuation scale is $K^{\xi_*}=K^{\xi^2/16}$.
Their ratio is $K^{3\xi^2/16}$, and squaring it in the Gaussian
tail gives $K^{3\xi^2/8}$. Since
$\max_{j\in I}|\alpha_j-\Gamma_j|\le CK^\epsilon$,
\begin{equation}\label{gp:conditional-rigidity}
 \PP^{\mu_M}(|u_j-\Gamma_j|>M^\nu)
 \leq C\exp(-cK^{3\xi^2/8}),\qquad j\in I.
\end{equation}
Decrease the exponent $c_2$ in the global Gaussian estimates
\cite[(5.1)--(5.4)]{EYGap} so that $16c_2<3\xi^2/8$.
Then \eqref{gp:conditional-rigidity} is at most
$\exp(-c_1M^{c_2}/2)$ for large $K$. This is exactly the conditional
bound defining $R^*$ in \cite[(5.8)]{EYGap}; the weaker global
constants remain valid in its Lemma~5.2 and Proposition~5.3.

\proofstep{4. Compare with a typical Gaussian boundary}
Apply that proposition with both ensembles Gaussian and integrate
the second exterior over the $R^*$ event of its Lemma~5.2, whose
complement has probability $O(e^{-cM^{c_2}})$. For the fixed anchor
$r$, take the two block centers to be $M/2$ and $M/2-r$, respectively;
the comparison gap then has label $M/2$, as required in the proposition.
The cited theorem uses anchored differences, which for an
increasing array are $d_j=u_r-u_{r+j}$, $1\le j\le k$.
Our consecutive positive gaps are recovered by
$u_{r+1}-u_r=-d_1$ and
$u_{r+j}-u_{r+j-1}=d_{j-1}-d_j$ for $j\ge2$.
This invertible linear change of variables takes a smooth compactly
supported gap test to a smooth compactly supported anchored test,
with controlled derivative norm. Where the cited comparison uses
mean-spacing-one gaps, apply it to $\bmv\mapsto O(\pi\bmv)$; this
fixed dilation preserves smoothness, compact support, and the error
bounds up to constants depending on $O$. The factor $\pi$ times the
semicircle density at label $M/2+r$ is $1+O(M^{-2})$, and
Theorem~2.3 of \cite{EYGap} moves the Gaussian comparison from label
$M/2$ to $M/2+r$. These errors and \eqref{gp:finite-error} prove the
lemma.
\end{proof}

\subsection{Identification of the GOE gap limit}

\begin{lemma}[Identification of the GOE reference gaps]\label{gp:goe}
For each fixed $r,k$ and $O\in C_c^\infty(\R^k)$,
$\E^{\mathrm{GOE}_M}O(\bmg_{M,r})\to\E O(G_0,\ldots,G_{k-1})$
as $M\to\infty$ through even integers.
\end{lemma}
\begin{proof}
A fixed-label gap law is not literally a fixed-energy point-process
law. We first use the GOE gap comparison to replace the fixed label
by a short label average. Rigidity turns that average into a spatial
average over actual particles. Campbell's formula then identifies
it with a Palm expectation, to which the GOE kernel asymptotics
apply.

In this proof, $\E$ averages the GOE eigenvalues, except for an
expectation explicitly involving the limiting Palm gaps
$(G_0,\ldots,G_{k-1})$. The notation $\E^0_{M,t}$ below denotes
the finite-GOE Palm expectation, not the graph or exterior law.
Using the increasing GOE eigenvalues $\mu_j$ fixed above, put
$\Xi_M=\sum_j\delta_{M\mu_j}$,
$n_M=\lfloor M^{1/4}\rfloor$, and
$\mathcal A_M=\{j\in\mathbb Z:|j|\leq n_M\}$.
Theorem~2.3 of \cite{EYGap} gives
\[
 \E O(\bmg_{M,r})-
 \frac1{|\mathcal A_M|}\sum_{j\in\mathcal A_M}\E O(\bmg_{M,j})
 \longrightarrow0.
\]
On these labels, $\pi$ times the local semicircle density is
$1+O(M^{-3/2})$, a negligible change for $O$. Gaussian rigidity
replaces this label set by the points of $\Xi_M$ in
$[-\pi n_M,\pi n_M]$, with only $O(M^\eta)$
endpoint errors for any $0<\eta<1/4$, outside an exponentially rare event.
Let $F_O$ apply $O$ to the first $k$ right gaps of a rooted
configuration, with value zero if fewer neighbors exist. Write
$\rho_{M,q}$ for the correlation densities of $\Xi_M$ and
$\E^0_{M,t}$ for its Palm expectation at $t$, translated to zero.
The label average consequently differs by $o(1)$ from
\[
 \frac1{|\mathcal A_M|}\E
 \sum_{u\in\Xi_M\cap[-\pi n_M,\pi n_M]}F_O(\Xi_M-u)
 =\frac1{|\mathcal A_M|}\int_{-\pi n_M}^{\pi n_M}
       \rho_{M,1}(t)\E^0_{M,t}F_O\dd t.
\]
Under this Palm law the nonroot correlation densities are
\[
 \frac{\rho_{M,q+1}(t,t+v_1,\ldots,t+v_q)}{\rho_{M,1}(t)}.
\]
Apply \cite[arXiv v3, Theorem 1.1 and Corollary 1.2]{DG} with
$V(x)=x^2/2$, so its GOE weight is $e^{-x^2/4}$ and $x=\sqrt M\mu$.
The centers $t/\sqrt M=O(M^{-1/4})$ remain in a compact set.
The resulting kernel convergence, expressed in the coordinates
$M\mu$, is therefore uniform for $|t|\leq\pi n_M$ and bounded
$\bmv$, and $\rho_{M,1}(t)=1/\pi+o(1)$.
The bounded matrix entries, including the sign entry, give the
Pfaffian bound $C^q(2q)^{q/2}$ by Hadamard's inequality.
After fixing the root and dividing by $q!$, these bounds are summable
in every bounded interval. For example, for a nonnegative continuous $f$ of compact support,
the reduced Palm Laplace functional is
\[
 \sum_{q=0}^\infty\frac{(-1)^q}{q!}
 \int\prod_{j=1}^q(1-e^{-f(v_j)})
 \frac{\rho_{M,q+1}(t,t+v_1,\ldots,t+v_q)}{\rho_{M,1}(t)}
 \,\dd\bm v.
\]
For finite $M$ the series terminates; we can regard the remaining
terms as zero. The Pfaffian bounds, divided by $q!$ and multiplied
by the fixed support volume to the power $q$, give a summable
majorant independent of $M,t$. Dominated convergence proves uniform
convergence of these local Laplace functionals. This identifies
uniform local convergence of the reduced Palm laws; adding the
fixed root gives the Palm laws $\Theta_1^0$.
If $\supp O\subset[-R_O,R_O]^k$, then $F_O$ depends only on
$[0,kR_O+1]$ and is almost surely continuous under the simple Palm law.
Consequently $\E^0_{M,t}F_O\to\E O(G_0,\ldots,G_{k-1})$ uniformly.
The prefactor in the Campbell identity satisfies
\[
 \frac1{|\mathcal A_M|}\int_{-\pi n_M}^{\pi n_M}\rho_{M,1}(t)\,\dd t
 =\frac{2\pi n_M}{|\mathcal A_M|}
       \left(\frac1\pi+o(1)\right)\longrightarrow1.
\]
Thus the spatial-window length and the intensity cancel, giving the
claimed Palm expectation.
For the endpoint replacement, divide the number
$O(M^\eta)$ of possible mismatched labels by
$|\mathcal A_M|\asymp M^{1/4}$; it vanishes for $\eta<1/4$.
The boundedness of $O$ controls the rare exceptional event as well.
For locality of $F_O$, a nonzero value requires all first $k$ gaps
to be at most $R_O$, so the required neighbors lie within
$[0,kR_O]$. If a required next gap is larger, its absence is detected
within $[0,kR_O+1]$ and the test is zero. Under a simple Palm process
with no particle at the fixed outer boundary, this is a bounded,
almost surely continuous local observable. Thus local Palm
convergence suffices; no assertion about an unbounded search for
neighbors is needed.
\end{proof}

\begin{proof}[Proof of Proposition~\ref{gp:rigid-gibbs}]
At this stage the microscopic graph limit has already been taken.
The only large parameters are the conditional block size $K$ and
the auxiliary GOE size $M=K^{16}$. In particular, none of the
qualitative eigenvector limits is being used uniformly in $K$.
For $O\in C_c^\infty(\R^k)$, integrate Lemma~\ref{gp:comparison}
over the exterior. Equations~\eqref{gp:exterior} and~\eqref{gp:means}
hold outside an event of probability $O(K^{-D})$, so
\[
 \left|\E^{\mathsf Q_K}O(\bmg_r)
       -\E^{\mathrm{GOE}_M}O(\bmg_{M,r})\right|
 \le C_O\bigl(K^{-2}+K^{-D}+K^{-c_*}+e^{-cM^{c_2}}\bigr).
\]
Relative-gap expectations under $\mathsf Q_K$ equal those under
$\mathsf P$ by hypothesis, even though $\mathsf Q_K$ depends on $K$.
Let $K\to\infty$ through even integers, with $M=K^{16}$, and apply
Lemma~\ref{gp:goe}. Equality for smooth compactly supported tests
determines the claimed distribution of each finite gap vector.
Indeed such tests determine finite Borel measures on $\R^k$ by
approximation of continuous compactly supported functions. Both
sides are probability laws. The left side is independent of $K$
because each $\mathsf Q_K$ preserves the same relative law
$\mathsf P$, so passing to larger conditional blocks does not
change the object being identified.
\end{proof}

\end{document}